\documentclass[final,oneeqnum,onefignum,onetabnum]{siamltex}

\usepackage{amssymb}
\usepackage{amsmath}
\usepackage{amscd}
\usepackage{graphics}
\usepackage{epsfig}
\usepackage{url}
\usepackage{mdwlist}
\usepackage{tabls}
\usepackage{array}
\usepackage{algorithmic}
\usepackage{algorithm}

\usepackage{pgf}
\usepackage{tikz}
\usepgflibrary{arrows}

\usepackage{verbatim} 

\newtheorem{preremark}[theorem]{Remark}
\newenvironment{remark}%
{\begin{preremark} \upshape}{\end{preremark}}

\newcommand{\bm}[1]{\mbox{\boldmath${#1}$}}

\newcommand{\domain}{\Omega}
\newcommand{\cdomain}{\widebar{\Omega}}
\newcommand{\boundary}{\partial \domain}

\newcommand{\xbar}{\bar{x}}

\newcommand{\bx}[1]{\mbox{\boldmath$x_{#1}$}}

\newcommand{\x}{\bm{x}}

\newcommand{\bb}{\bm{b}}

\newcommand{\bn}{\bm{n}}

\newcommand{\y}{\bm{y}}

\newcommand{\widebar}{\overline}
\newcommand{\R}{\bm{R}} 

\newcommand{\blds}[1]{\mbox{\scriptsize \boldmath $#1$}}

\DeclareMathOperator*{\argmax}{argmax}
\DeclareMathOperator*{\argmin}{argmin}

\newif\ifnever\neverfalse

\newif\iffullversion\fullversiontrue

\newcommand{\marginfix}{
\setlength{\parskip}{0.01cm}
\setlength{\textwidth}{6.03in}
\setlength{\oddsidemargin}{0.2 in}
\setlength{\evensidemargin}{0.2 in}
\setlength{\topmargin}{-0.45in}
\setlength{\textheight}{9.03 in}
}

\marginfix

  \renewenvironment{thebibliography}[1]{%
    \begin{oldthebibliography}{#1}%
      \setlength{\parskip}{.3ex}%
      \setlength{\itemsep}{.3ex}%
  }%
  {%
    \end{oldthebibliography}%
  }

\begin{document}


{\Large{\bf
\centerline{Fast two-scale methods for Eikonal equations.}
}}

\vspace*{.1in}
{\Large
\centerline{Adam Chacon and Alexander Vladimirsky\footnote{
\mbox{This research is supported in part by 
the National Science Foundation grants DMS-0514487 and DMS-1016150.}
\mbox{The first author's research is also supported    
by Alfred P. Sloan Foundation Graduate Fellowship.             }
\iffullversion
\mbox{This manuscript is an extended version of the paper submitted for 
publication} in SIAM J. on Scientific Computing.
In the journal version, subsections \ref{ss:coarser} and \ref{ss:general_bc}
and parts of section \ref{s:our_methods} were omitted due to space limitations.
\fi
}}

\vspace*{.1in}
\centerline{Department of Mathematics and Center for Applied Mathematics}
\centerline{Cornell University, Ithaca, NY 14853}
}

\vspace*{.1in}
\begin{abstract}
\noindent

Fast Marching and Fast Sweeping are the two most commonly used
methods for solving the Eikonal equation. 
Each of these methods performs best on a different set of problems.  
Fast Sweeping, for example, will outperform Fast Marching on problems 
where the characteristics are largely straight lines.  
Fast Marching, on the other hand, is usually more efficient than 
Fast Sweeping on problems where characteristics frequently change 
their directions and on domains with complicated geometry.
In this paper we explore the possibility of combining the best 
features of both of these approaches, by using marching on 
a coarser scale and sweeping on a finer scale.
We present three new hybrid methods based on this idea and illustrate
their properties on several numerical examples with continuous and 
piecewise-constant speed functions in $R^2$.
\end{abstract}

\section{Introduction}
\label{s:intro}

Static Eikonal PDEs arise in a surprisingly wide range of applications:
from robotic path planning, to isotropic optimal control, 
to isotropic front propagation, to shape-from-shading computations;
see \cite{SethBook2} and references therein for a detailed description.
As a result, efficient numerical methods for Eikonal PDEs are of interest
to many practitioners and numerical analysts.  In this paper we introduce
two hybrid methods intended to blend and combine the best properties
of the most popular current approaches (Fast Marching and Fast Sweeping).

These methods are built to solve the non-linear boundary value problem\footnote{
For simplicity, we will restrict our exposition to 
first-order accurate discretizations of these problems 
on Cartesian grids in $R^2$, although generalizations 
to higher dimensional domains are straightforward and similar
approaches are applicable to 
higher-order accurate discretizations on unstructured meshes
in $R^n$ and on manifolds.}
\begin{eqnarray}
\nonumber
| \nabla u (\x) | F(\x) &=& 1, \text{ on } \domain \subset R^2;\\
u(\x) &=& q(\x),  \text{ on } \boundary.
\label{eq:Eikonal}
\end{eqnarray}



A discretized version of equation \eqref{eq:Eikonal} 
is posed at every gridpoint, using upwind divided differences
to approximate the partial derivatives of $u$.  
The exact form of this discretization is introduced in section \ref{s:prior_fast};
here we simply note that these discretized equations form a system of 
$M$ coupled non-linear equations (where $M$ is the number of gridpoints) and that the key challenge addressed
by many ``fast'' methods is the need to solve this system efficiently.
Of course, an iterative approach is certainly a possibility,
but its most straightforward and naive implementation typically leads
to $O(M^2)$ algorithmic complexity for Eikonal PDE 
(and potentially much worse for its anisotropic generalizations).
This is in contrast to the ``fast'' methods, whose worst-case
computational complexity is $O(M)$ or $O(M \log M)$.

Interestingly, most fast Eikonal-solvers currently in use 
are directly related 
to the fast algorithms developed much earlier 
to find
the shortest paths in directed graphs with non-negative
edge-lengths; 
see, e.g., \cite{Ahuja}, \cite{Bertsekas_NObook, Bertsekas_DPbook}.  
Two such algorithmic families are particularly prominent:
the {\em label-setting methods}, which have the optimal worst-case 
asymptotic computational complexity, and 
the {\em label-correcting methods}, whose worst-case asymptotic 
complexity is not as good, but the practical performance is at times
even better than that of label-setting.  We provide a basic overview
of both families in section \ref{ss:intro_graphs}.  The prior fast 
Eikonal-solvers based on label-setting and label-correcting 
are reviewed in sections 
\ref{ss:FM} and \ref{ss:FS}-\ref{ss:correcting_Eikonal} respectively.

The most popular methods from these two categories (Fast Marching 
and Fast Sweeping) have been shown to be efficient on a wide range
of Eikonal equations.  However, each of these methods has its
own preferred class of problems, on which it significantly
outperforms the other.
Despite experimental comparisons already conducted in
\cite{HysingTurek} and \cite{GremaudKuster}, the exact
delineation of a preferred problem-set for each method
is still a matter of debate.
Fast Sweeping (reviewed in section \ref{ss:FM}) is usually
more efficient on problems with constant characteristic directions.
But for general functions $F(\x)$, its
computational cost is clearly impacted by the frequency 
and magnitude of directional changes of characteristic curves.
Fast Marching (reviewed in section \ref{ss:FM})
is generally more efficient on domains with complicated geometry and on
problems with characteristic directions frequently changing.
Its causal algorithmic structure results in a provably converged solution
on explicitly determined parts of the computational domain 
even before the method terminates 
-- a very useful feature in many applications.  
Moreover, its efficiency is much more ``robust''; i.e., 
its computational cost is much less affected 
by any changes in functions $F$ and $q$ or the grid orientation.
But as a result, the Fast Marching also is not any faster in the simplest 
cases where $F$ is constant on a convex domain and 
all characteristics are straight lines --
the exact scenario where the Fast Sweeping is at its most efficient.

The fundamental idea underlying our hybrid two-scale methods is 
to take advantage of the best features of both marching and sweeping.
Suppose the domain is split in a moderate number of cells
such that $F$ is almost-constant on each of them.  
(Such cell splitting is possible for any piecewise-smooth $F$.)
On the top scale, a version of Fast Marching can be used on 
a coarse grid (with each gridpoint representing a cell of the fine grid).
Once the ordering of coarse gridpoints is established, the Fast Sweeping
is applied to individual cells of the fine grid in the same order.
This is the basis of our Fast Marching-Sweeping Method
(FMSM) described in section \ref{ss:FMSM}.

Unfortunately, the coarse grid ordering captures the information flow
through the fine grid cells only approximately: a coarse gridpoint $\y_i$
might be ``accepted'' by Fast Marching before another coarse gridpoint $\y_j$,
even if on the fine grid the characteristics cross both from cell $i$ 
to cell $j$ and from cell $j$ to cell $i$.  The ``one-pass'' nature of 
Fast Marching prevents FMSM  from acting on such interdependencies
between different cells even if they are revealed during the application
of Fast Sweeping to these cells.  To remedy this, we introduce the 
Heap-Cell Method (HCM) described in section \ref{ss:FHCM}.
The idea is to allow multiple passes through fine grid cells 
sorted by the representative ``cell-values'' and updated 
as a result of cell-level fast sweeping.
We also describe its heuristic version, the Fast Heap-Cell Method (FHCM), 
where the number of cell-level sweeps is determined based on 
the cell-boundary data.

Similarly to Fast Marching and Fast Sweeping, our HCM provably converges 
to the exact solution of the discretized equations on the fine scale.
In contrast, the even faster FHCM and FMSM usually introduce 
additional errors. 
But based on our extensive numerical experiments 
(section \ref{s:experiments}), 
these additional errors are small compared to the errors already 
present due to discretization.  The key advantage of 
all three new methods is their computational efficiency -- 
with properly chosen cell sizes, we can significantly outperform both
Fast Sweeping and Fast Marching on examples difficult for those
methods, while matching their performance on the examples which 
are the easiest for each of them.
We conclude by discussing the current limitations of our approach and 
several directions for future work in section \ref{s:conclusions}.


\subsection{Fast algorithms for paths on graphs}
\label{ss:intro_graphs}
We provide a brief review of common fast methods
for the classical shortest/cheapest path
problems on graphs.  
Our exposition follows \cite{Bertsekas_NObook} and \cite{Bertsekas_DPbook}, 
but with modifications needed to emphasize the parallels with 
the numerical methods in sections \ref{s:prior_fast} and \ref{s:our_methods}.

Consider a directed graph with nodes $X= \{\x_1, ... , \x_M\}$.
Let $N(\x_i)$ be the set of nodes to which $\bx{i}$ is connected.
We will assume that $\kappa \ll M$ is an upper bound on outdegrees; i.e.,
$\left| N(\x_i) \right | \leq \kappa.$
We also suppose that all arc-costs $C_{ij} = C(\x_i,\x_j)$ are positive
and use $C_{ij} = + \infty$ whenever $\x_j \not \in N(\x_i)$.
Every path terminates upon reaching the specified exit set $Q \subset X$,
with an additional exit-cost $q_i= q(\x_i)$ for each $\x_i \in Q.$
Given any starting node $\x_i \in X$, the goal is to find the cheapest
path to the exit starting from $\x_i$.
The {\em value function} $U_i = U(\x_i)$ is defined to be the optimal 
path-cost (minimized over all paths starting from $\x_i$).
If there exists no path from $\x_i$ to $Q$, then $U_i = + \infty$, but for simplicity 
we will henceforth assume that $Q$ is reachable from each $\x_i$ and all $U_i$'s are finite.
The optimality principle states that the ``tail'' of 
every optimal path is also optimal; hence,
\begin{eqnarray}
\nonumber
U_i &=& \min\limits_{\x_j \in N(\x_i)} \left\{C_{ij} + U_j\right\},
 \qquad \text{for } \forall \x_i \in X \backslash Q;\\
\label{eq:DP_discrete}
U_i &=& q_i, 
\qquad \text{for } \forall \x_i \in Q.
\end{eqnarray}
This is a coupled system of $M$ non-linear equations,
but it possesses a nice ``causal'' property:
if $\x_j \in N(\x_i)$ is the minimizer, then $U_i>U_j$.


In principle, this system could be solved by ``value iterations'';
this approach is unnecessarily expensive (and is usually reserved for 
harder {\em stochastic} shortest path problems),  but we describe it here for
methodological reasons, to emphasize the parallels with ``fast iterative'' 
numerical methods for Eikonal PDEs.
An operator $T$ is defined on $\R^M$ component-wise by applying 
the right hand side of equation \eqref{eq:DP_discrete}.
Clearly, 
$U = 
\left[
\begin{array}{c}
U_1\\
\vdots\\
U_M
\end{array}
\right]
$
is a fixed point of $T$ and one can, in principle, recover $U$ 
by {\em value iterations}:
\begin{equation}
W^{k+1} \, := \, T \, W^k \qquad
\text{ starting from any initial guess $W^0 \in \R^M$.}
\label{eq:generic_value_it}
\end{equation}
Due to the causality of system \eqref{eq:DP_discrete}, 
value iterations will converge to $U$ regardless of $W^0$
after at most $M$ iterations, 
resulting in $O(M^2)$ computational cost.  
(It is easy to show by induction that $W^k_i = U_i$
for every $\x_i$ from which there exists an optimal path 
with at most $k$ transitions.)
A Gauss-Seidel relaxation of this iterative process is  
a simple practical modification, where the entries of $W^{k+1}$
are computed sequentially and 
the new values are used as soon as they become available:
$W_i^{k+1} = T_i (W_1^{k+1}, \ldots, W_{i-1}^{k+1}, W_i^k, \ldots, W_M^k).$
The number of iterations required to converge will now heavily
depend on the ordering of the nodes (though $M$ is still the upper bound).  
We note that, again due to causality of \eqref{eq:DP_discrete},
if the ordering is such that $U_i > U_j \Longrightarrow i>j$,
then only one full iteration will be required (i.e., $W^1 = U$ 
regardless of $W^0$).  Of course, $U$ is not known in advance 
and thus such a causal ordering is usually not available a priori
(except in acyclic graphs).
If several different node orderings are somehow known to capture 
likely dependency chains among the nodes, then a reasonable approach
would be to perform Gauss-Seidel iterations alternating through that list of
preferred orderings -- this might potentially result in a substantial reduction 
in the number of needed iterations.  
In section \ref{ss:FS} we explain how such preferred orderings arise
from the geometric structure of PDE discretizations,
but no such information is typically available in problems on graphs.
As a result, instead of alternating through a list of predetermined
orderings, efficient methods on graphs are based on finding advantageous
orderings of nodes {\em dynamically}.  
This is the basis for {\em label-correcting} and {\em label-setting} methods.

A generic label-correcting method is summarized below in algorithm \ref{alg:generic_lc}.
\begin{algorithm}[hhhh]
\caption{Generic Label-Correcting pseudocode.}
\label{alg:generic_lc}
\algsetup{indent=2em}
\begin{algorithmic}[1]
\STATE Initialization:
\FOR{each node $\x_i$}
	\IF{$\x_i \in Q$}
		\STATE $V_i \gets q_i$
	\ELSE
		\IF{$N(\x_i) \bigcap Q \neq \emptyset$}
			\STATE $V_i \gets \min\limits_{\x_j \in N(\x_i) \bigcap Q} \left\{C_{ij} + q_j\right\}$
			\STATE add $\x_i$ to the list $L$
		\ELSE
			\STATE $V_i \gets \infty$
		\ENDIF
	\ENDIF
\ENDFOR
\STATE
\STATE Main Loop:
\WHILE{$L$ is nonempty}
	\STATE Remove a node $\x_j$ from the list $L$
	\FOR{each $\x_i \not \in Q$ such that $\x_j \in N(\x_i)$ and $V_j < V_i$}
		\STATE $\widetilde{V} \gets C_{ij} + V_j$
		\IF{$\widetilde{V} < V_i$}
			\STATE $V_i \gets \widetilde{V}$
			\IF{$\x_i \not \in L$}
				\STATE add $\x_i$ to the list $L$
			\ENDIF 
		\ENDIF
	\ENDFOR
\ENDWHILE
\end{algorithmic}
\end{algorithm}
It is easy to prove that this algorithm always terminates and that upon its
termination $V= U$; e.g., see \cite{Bertsekas_NObook}.  
Many different label-correcting methods are obtained by
using different choices on how to add the nodes to the list $L$
and which node to remove (in the first line inside the while loop).
If $L$ is implemented as a queue, the node is typically removed from 
the top of $L$.  Always adding the nodes at the bottom of $L$ yields 
the {\em Bellman-Ford method} \cite{Bellman_DP_book}.    
(This results in a first-in/first-out policy
for processing the queue.)  Always adding nodes at the top of $L$ produces the
{\em depth-first-search} method, with the intention of minimizing the memory 
footprint of $L$.  Adding nodes at the top if they have already been in $L$ before,
while adding the ``first-timers'' at the bottom yields {\em D'Esopo-Pape method} \cite{Pape}.
Another interesting version is the so called 
{\em small-labels-first} (SLF) method \cite{Bertsekas_SLF},
where the node is added at the top only if its value is smaller than that
of the current top node and at the bottom otherwise.  Another variation is 
{\em large-labels-last} (LLL) method \cite{Bertsekas_LLL}, 
where the top node is removed only if 
its value is smaller than the current average of the queue; otherwise it's 
simply moved to the bottom of the queue instead.  Yet another popular approach
is called {\em thresholding method}, where $L$ is split into two queues,
nodes are removed from the first of them only and added to the first or the second
queue depending on whether the labels 
are smaller than some (dynamically changing) threshold value \cite{Glover}.
We emphasize that the convergence is similarly obtained for all of these methods,
their worst-case asymptotic complexity is $O(M^2)$, but their 
comparative efficiency for specific problems can be dramatically different.

Label-setting algorithms can be viewed as a subclass of the above 
with an additional property: nodes removed from $L$ never need to be re-added later.
Dijkstra's classical method \cite{Diks} is the most popular in this category and is based 
on always removing the node with the smallest label of those currently in $L$.
(The fact that this results in no re-entries into the list is yet another consequence 
of the causality; the inductive proof is simple; e.g., see \cite{Bertsekas_NObook}.)
The need to find the smallest label entails additional computational 
costs.  A common implementation of $L$ using heap-sort data structures will
result in $O(M \log M)$ overall asymptotic complexity of the method on sparsely connected
graphs (i.e., provided $\kappa \ll M$).
Another version, due to Dial \cite{Dial}, implements $L$ as a list of ``buckets'',
so that all nodes in the current smallest bucket can be safely removed simultaneously,
resulting in the overall asymptotic complexity of $O(M)$.
The width of each bucket is usually set to be $\delta = \min_{i,j} C_{ij}$ to ensure
that the nodes in the same bucket could not influence or update each other even if 
they were removed sequentially.

We note that several label-correcting methods were designed
to mimic the ``no-re-entry'' property of label-setting, but without
using expensive data structures.  (E.g., compare SLF/LLL to Dijkstra's 
and thresholding to Dial's.)
Despite the lower asymptotic complexity of label-setting methods, 
label-correcting algorithms can be more efficient on many problems.
Which types of graphs favor which of these algorithms remains largely
a matter of debate.  We refer readers to \cite{Bertsekas_NObook, Bertsekas_DPbook}
and references therein for additional details and asynchronous (parallelizable) 
versions of label-correcting algorithms.

\section{Eikonal PDE, upwind discretization \& prior fast methods}
\label{s:prior_fast}

Static Hamilton-Jacobi equations frequently arise in
exit-time optimal control problems.  
The Eikonal PDE \eqref{eq:Eikonal} describes an important 
subset: isotropic time-optimal control problems.
The goal is to drive a system starting from a point $\x \in \domain$
to exit the domain as quickly as possible.  In this setting,
$F:\domain \to \R_+$ is the local speed of motion, 
and $q: \boundary \to \R$ is the exit-time penalty 
charged at the boundary.  
We note that more general control problems 
(with an exit-set $Q \subset \boundary$
and trajectories constrained to remain inside $\domain$ until reaching $Q$)
can be treated similarly by setting $q = +\infty$ on $\boundary \backslash Q$.
 
The {\em value function} $u(\x)$ is defined 
to be the minimum time-to-exit starting from $\x$ and a formal argument 
shows that $u$ should satisfy the equation \eqref{eq:Eikonal}.
Moreover, characteristics of this PDE, coinciding with the gradient lines of $u$,
provide the optimal trajectories for moving through the domain.
Unfortunately, the equation \eqref{eq:Eikonal} usually does not have 
a classical (smooth) solution on the entire domain, 
while weak solutions are not unique.  Additional test conditions are used
to select among them the unique {\em viscosity solution}, which coincides with the
value function of the original control problem \cite{CranLion, CranEvanLion}.
A detailed treatment of general optimal control problems in 
the framework of viscosity solutions can be found in \cite{BardiDolcetta}.

Many discretization approaches for the Eikonal equation have been extensively studied
including first-order and higher-order Eulerian discretizations on grids and meshes 
in $\R^n$ and on manifolds \cite{RouyTour, SethFastMarcLeveSet, KimmSethTria, SethVlad1},
semi-Lagrangian discretizations \cite{Falc_Dial, GonzalezRofman}, and 
the related approximations with controlled Markov chains \cite{KushnerDupuis, BoueDupuis}.
For the purposes of this paper, we will focus on the simplest first-order upwind 
discretization on a uniform Cartesian grid $X$ (with gridsize $h$) on $\cdomain \subset \R^2$.
To simplify the description of algorithms, we will further assume that 
both $\boundary$ and $Q$ are naturally discretized on the grid $X$.
Our exposition here closely follows \cite{SethSIAM, SethBook2}.

To introduce the notation, we will refer to
gridpoints $\x_{ij}=(x_i, y_j)$, value function approximations 
$U_{ij} = U(\x_{ij}) \approx u(\x_{ij})$, and the speed $F_{ij} = F(\x_{ij})$.
A popular first-order accurate discretization of 
\eqref{eq:Eikonal} is obtained by using upwind finite-differences to approximate
partial derivatives:
\begin{equation}
\label{eq:Eik_discr}
\left( \max \left( D^{-x}_{ij}U, \, -D^{+x}_{ij}U, \, 0 \right)\right)^{2} 
\; + \;
\left( \max \left( D^{-y}_{ij}U, \, -D^{+y}_{ij}U, \, 0 \right)\right)^{2} 
\; = \; \frac{1}{F^2_{ij}},
\end{equation}
$$
\text{ where } \qquad
u_x(x_i, y_j) \approx D^{\pm x}_{ij} U = 
\frac{ U_{i \pm 1, j} - U_{i,j} }{ \pm h}; \qquad
u_y(x_i, y_j) \approx D^{\pm y}_{ij} U = 
\frac{ U_{i, j \pm 1} - U_{i,j} }{ \pm h}.
$$
If the values at four surrounding nodes are known,
this equation can be solved to recover $U_{ij}$.
This is best accomplished by computing updates 
from individual quadrants as follows.
Focusing on a single node $\x_{ij}$, we will simplify 
the notation by using $U = U_{ij}$, $F = F_{ij}$, and $\{U_E, U_N, U_W, U_S\}$
for the values at its four neighbor nodes.

First, suppose that $\max \left( D^{-x}_{ij}U, \, -D^{+x}_{ij}U, \, 0 \right) = 0$
and $\max \left( D^{-y}_{ij}U, \, -D^{+y}_{ij}U, \, 0 \right) = -D^{+y}_{ij}U$.
This implies that $U \geq U_N$ and the resulting equation yields 
\begin{equation}
\label{eq:one_sided_update}
U = h/F + U_N.
\end{equation}

To compute ``the update from the first quadrant'', 
we now suppose that\\
$\max \left( D^{-x}_{ij}U, \, -D^{+x}_{ij}U, \, 0 \right) = -D^{+x}_{ij}U$
and $\max \left( D^{-y}_{ij}U, \, -D^{+y}_{ij}U, \, 0 \right) = -D^{+y}_{ij}U$.\\
This implies that $U \geq U_N$ and $U \geq U_E$.
The resulting quadratic equation is 
\begin{equation}
\label{eq:two_sided_update}
\left(\frac{U-U_E}{h}\right)^2 + \left(\frac{U-U_N}{h}\right)^2 = \frac{1}{F^2}.
\end{equation}
We define ``the update from the first quadrant'' $U^{NE}$ to be 
the root of the above quadratic satisfying $U \geq \max(U_N, U_E)$.
If no such root is available, we use the smallest of the ``one-sided'' updates,
similar to the previous case; i.e., $U^{NE} = h/F + \min(U_N, U_E).$
If we similarly define the updates from the remaining three quadrants,
it is easy to show that $U = \min(U^{NE}, U^{NW}, U^{SW}, U^{SE})$ satisfies 
the original equation \eqref{eq:Eik_discr}.

\begin{remark}
\label{rem:discr_properties}
It is also easy to verify that this discretization is\\ 
$\bullet \,$ {\em consistent}, i.e., suppose both sides of \eqref{eq:Eik_discr}
are multiplied by $h^2$; if the true solution $u(\x)$ is smooth,
it satisfies the resulting discretized equation up to $O(h^2)$;\\ 
$\bullet \,$ {\em monotone},
i.e., $U$ is a non-decreasing function of each of
its neighboring values;\\ 
$\bullet \,$ {\em causal}, i.e., $U$ depends only on the neighboring values smaller than itself
\cite{SethFastMarcLeveSet, SethBook2}.\\
The consistency and monotonicity can be used to prove the convergence to
the viscosity solution $u(\x)$; see \cite{BarlesSouganidis}.  
\end{remark}
However, since \eqref{eq:Eik_discr} has to hold at every gridpoint 
$\x_{ij} \in X \backslash Q$, this discretization results 
in a system of $M$ coupled non-linear equations, where $M$ is the number of 
gridpoints in the interior of $\domain$.  
In principle, this system can be solved iteratively 
(similarly to the ``value iterations'' process
described in  \eqref{eq:generic_value_it}) with or without Gauss-Seidel relaxation,
but a naive implementation of this iterative algorithm would be unnecessarily expensive,
since it does not take advantage of the causal properties of the discretization.
Several competing approaches for solving the discretized system efficiently 
are reviewed in the following subsections.

\subsection{Label-setting methods for the Eikonal}
\label{ss:FM}

The causality property observed above is the basis of Dijkstra-like methods for
the Eikonal PDE.  The first such method was introduced by Tsitsiklis 
for isotropic control problems using
first-order semi-Lagrangian discretizations on uniform Cartesian grids 
\cite{Tsitsiklis_conference, Tsitsiklis}.
The Fast Marching Method was introduced by Sethian \cite{SethFastMarcLeveSet}
using first-order upwind-finite differences in the context of 
isotropic front propagation.  
A detailed discussion of similarities 
and differences of these approaches can be found in \cite{SethVlad3}.
Sethian and collaborators have later extended the Fast Marching
approach to higher-order discretizations on grids and meshes \cite{SethSIAM}, 
more general anisotropic Hamilton-Jacobi-Bellman PDEs \cite{SethVlad2, SethVlad3}, 
and quasi-variational inequalities \cite{SethVladHybrid}.
Similar methods were also introduced for semi-Lagrangian discretizations
\cite{CristianiFalcone}.
The Fast Marching Method for the Eulerian discretization \eqref{eq:Eik_discr} 
is summarized below in Algorithm \ref{alg:FMM}.


\begin{algorithm}[hhhh]
\caption{Fast Marching Method pseudocode.}
\label{alg:FMM}
\algsetup{indent=2em}
\begin{algorithmic}[1]
\STATE Initialization:
\FOR{each gridpoint $\x_{ij} \in X$}
	\IF{$\x_{ij} \in Q$}
		\STATE Label $\x_{ij}$ as $Accepted$ and set $V_{ij} = q(\x_{ij})$.
	\ELSE
		\STATE Label $\x_{ij}$ as $Far$ and set $V_{ij} = \infty$.
	\ENDIF
\ENDFOR
\FOR{each $Far$ neighbor $\x_{ij}$ of each Accepted node}
	\STATE Label $\x_{ij}$ as $Considered$ and put $\x_{ij}$ onto the Considered List $L$.
	\STATE Compute a temporary value $\widetilde{V}_{ij}$ using the upwinding discretization.
	\IF{$\widetilde{V}_{ij} < V_{ij}$}
		\STATE $V_{ij} \gets \widetilde{V}_{ij}$
	\ENDIF
\ENDFOR
\STATE End Initialization
\STATE
\WHILE{$L$ is nonempty}
	\STATE Remove the point $\xbar$ with the smallest value from $L$.
	\FOR{$\x_{ij} \in N(\xbar)$}
	\STATE Compute a temporary value $\widetilde{V}_{ij}$ using the upwinding discretization.
	\IF{$\widetilde{V}_{ij} < V_{ij}$}
		\STATE $V_{ij} \gets \widetilde{V}_{ij}$
	\ENDIF
	\IF{$\x_{ij}$ is $Far$}
		\STATE Label $\x_{ij}$ as $Considered$ and add it to $L$.
	\ENDIF
	\ENDFOR
\ENDWHILE
\end{algorithmic}
\end{algorithm}

As explained in section \ref{ss:intro_graphs}, the label-setting Dijkstra's method can be considered
as a special case of the generic label-correcting algorithm, provided the current smallest node in $L$ 
is always selected for removal.  Of course, in this case it is more efficient to implement $L$
as a binary heap rather than a queue.  The same is also true for the Fast Marching Method, 
and a detailed description of an efficient implementation of the heap-sort data structure can 
be found in \cite{SethBook2}.  The re-sorting of $Considered$ nodes upon each update involves
up to $O(\log M)$ operations, resulting in the overall computational complexity of $O( M \log M)$.



Unfortunately, the discretization \eqref{eq:Eik_discr} is only weakly causal:
there exists no $\delta > 0$ such that $U^{NE} > \delta + \max(U^N, U^E)$
whenever $U^{NE} > \max(U^N, U^E).$  
Thus, no safe ``bucket width'' can be defined and
Dial-like methods are not applicable to the resulting discretized system.  
In \cite{Tsitsiklis} Tsitsiklis introduced a Dial-like method for 
a similar discretization but using an 8-neighbor stencil.  
More recently, another Dial-related method for the Eikonal PDE on a uniform grid
was introduced in \cite{KimGMM}.  
A more general formula for the safe bucket-width to be used in Dial-like methods 
on unstructured acute meshes was derived in \cite{VladMSSP}.
Despite their better computational complexity, Dial-like methods
often perform slower than Dijkstra-like methods at least on single processor architectures.
 
Finally, we note another convenient feature of label-setting methods: 
if the execution of the algorithm is stopped early (before the list $L$ becomes empty),
all gridpoints previously removed from $L$  will already have provably correct values.
This property (unfortunately not shared by the methods in sections 
\ref{ss:FS}-\ref{ss:correcting_Eikonal})
is very useful in a number of applications: e.g., when computing a quickest path
from a single source to a single target or in problems of image segmentation 
\cite{SethBook2}.

\subsection{Fast Sweeping Methods}
\label{ss:FS}

Suppose all gridpoints in $X$ are ordered.  We will slightly abuse the notation
by using a single subscript (e.g., $\x_i$)
to indicate the particular gridpoint's place in that ordering.
The double subscript notation (e.g., $\x_{ij}$) will be still reserved to indicate 
the physical location of a gridpoint in the two-dimensional grid.

Consider discretization \eqref{eq:Eik_discr} and suppose that the solution $U$ is known
for all the gridpoints.  Note that for each $\x_i$, the value $U_i$ will only depend 
on one or two of the neighboring values (depending on which quadrant is used for 
a two-sided update, similar to \eqref{eq:two_sided_update}, and on whether a 
one-sided update is employed, similar to \eqref{eq:one_sided_update}).  
This allows us to define a {\em dependency digraph} $G$ on the vertices
$\x_1, \ldots, \x_M$ with a link from $\x_i$ to $\x_j$ indicating that $U_j$ is needed 
to compute $U_i$.  The causality of the discretization \eqref{eq:Eik_discr} guarantees that $G$ 
will always be acyclic.  Thus, if we were to order the gridpoints respecting
this causality (i.e., with $i>j \; \Longrightarrow \;$ there is no path in $G$ from
$\x_j$ to $\x_i$), then a single Gauss-Seidel iteration would correctly solve 
the full system in $O(M)$ operations.  

However, unless $U$ was already computed,
the dependency digraph $G$ will not be generally known in advance.  
Thus, basing a gridpoint ordering on it is not a practical option.
Instead, one can alternate through a list of several ``likely'' orderings
while performing Gauss-Seidel iterations.  A geometric interpretation of the optimal control problem 
provides a natural list of likely orderings: if all characteristics point from SW to NE,
then ordering the gridpoints bottom-to-top and left-to-right within each row will ensure
the convergence in a single iteration (a ``SW sweep'').  
The ``Fast Sweeping Methods'' 
perform Gauss-Seidel iterations 
on the system \eqref{eq:Eik_discr} in alternating directions (sweeps).   
Let $m$ be the number of gridpoints in the $x$-direction and $n$ be the number in 
the $y$-direction, and $\x_{ij}$ will denote a gridpoint in a uniform Cartesian grid 
on $\domain \subset R^2.$
For simplicity, we will use the Matlab index notation 
to describe the ordering of gridpoints in each sweep.
There are four alternating sweeping directions: from SW, from SE, from NE, 
and from NW.  For the above described southwest sweep, the gridpoints $\x_{ij}$ will be 
processed in the following order:  {\tt i=1:1:m, j=1:1:n}.  
All four orderings are similarly defined in algorithm \ref{alg:FSM_order}.
\begin{algorithm}[hhhh]
\caption{Sweeping Order Selection pseudocode.}
\label{alg:FSM_order}
\algsetup{indent=2em}
\begin{algorithmic}[1]
\STATE $sweepDirection \gets sweepNumber \text{ \tt mod } 4$
\IF{$sweepDirection == 0$}
	\STATE $iOrder \gets \; (1:1:m)$
	\STATE $jOrder \gets \; (1:1:n)$
\ELSIF{$sweepDirection == 1$} 
	\STATE $iOrder \gets \; (1:1:m)$
	\STATE $jOrder \gets \; (n:-1:1)$
\ELSIF{$sweepDirection == 2$} 
	\STATE $iOrder \gets \; (m:-1:1)$
	\STATE $jOrder \gets \; (n:-1:1)$
\ELSE
	\STATE $iOrder \gets \; (m:-1:1)$
	\STATE $jOrder \gets \; (1:1:n)$
\ENDIF 
\end{algorithmic}
\end{algorithm} 

The alternating sweeps are then repeated until convergence.  
The resulting algorithm is summarized in \ref{alg:FSM}. 

\begin{algorithm}[hhhh]
\caption{Fast Sweeping Method pseudocode.}
\label{alg:FSM}
\algsetup{indent=2em}
\begin{algorithmic}[1]
\STATE Initialization:
\FOR{each gridpoint $\x_{ij} \in X $}
	\IF{$\x_{ij} \in Q$}
		\STATE $V_{ij} \gets q(\x_{ij})$.
	\ELSE
		\STATE $V_{ij} \gets \infty$.
	\ENDIF
\ENDFOR
\STATE
\STATE Main Loop:
\STATE sweepNumber $\gets$ 0
\REPEAT
	\STATE changed $\gets$ FALSE
	\STATE Determine iOrder and jOrder based on sweepNumber
	\FOR{$i = iOrder$}
		\FOR{$j = jOrder$}
			\IF{$\x_{ij} \not \in Q$} 
				\STATE Compute a temporary value $\widetilde{V}_{ij}$ using upwinding discretization \eqref{eq:Eik_discr}.
				\IF{$\widetilde{V}_{ij} < V_{ij}$}
					\STATE $V_{ij} \gets \widetilde{V}_{ij}$
					\STATE changed $\gets$ TRUE
				\ENDIF
			\ENDIF
		\ENDFOR
	\ENDFOR
	\STATE sweepNumber $\gets$ sweepNumber + 1
\UNTIL{changed == FALSE}
\end{algorithmic}
\end{algorithm}

\begin{remark}
\label{rem:sweep_history}
The idea that alternating the order of Gauss-Seidel sweeps might
speed up the convergence is a centerpiece of many fast algorithms.
For Euclidean distance computations it was first used by Danielsson in \cite{Danielsson}. 
In the context of general HJB PDEs it was introduced by Boue and Dupuis in \cite{BoueDupuis} 
for a numerical approximation based on controlled Markov chains. 
More recently, a number of papers by Cheng, Kao, Osher, Qian, Tsai, and Zhao
introduced related Fast Sweeping Methods to speed up the iterative solving of 
finite-difference discretizations \cite{TsaiChengOsherZhao, Zhao, KaoOsherQian}.
The key challenge for these methods is to find
a provable and explicit upper bound on the number of iterations. 
As of right now, such a bound is only available for boundary value problems in which characteristics
are straight lines. Experimental evidence suggests that these methods can be also very efficient for
other problems where the characteristics are ``largely'' straight.  
The number of necessary iterations is largely independent of $M$ 
and correlated with the number of times the characteristics 
``switch directions'' (i.e., change from one directional quadrant to another) inside $\domain$.  
However, since
the quadrants are defined relative to the grid orientation, the number of iterations will generally 
be grid-dependent.  

One frequently encountered argument is that, due to its $O(M)$ computational complexity,
the Fast Sweeping is more efficient than the Fast Marching, whose complexity is $O(M \log M)$.
However, this asymptotic complexity notation hides constant factors -- including 
this not-easily-quantifiable (and grid-orientation-dependant) bound on 
the number of iterations needed in Fast Sweeping.  
As a result, whether the $O(M)$ asymptotic complexity actually
leads to any performance advantage on grids of realistic size is 
a highly problem-dependent question.
Extensive experimental comparisons of Marching and Sweeping approaches
can be found in \cite{HysingTurek, GremaudKuster}.  
Even though such a comparison is not the main focus of the current paper,
the performance of both methods is also tabulated for all 
examples in section \ref{s:experiments}.
On the grids we tested, we observe that the Fast Marching Method 
usually performs better than the Fast Sweeping when the domain has a complicated geometry 
(e.g., shortest/quickest path computations in a maze) or 
if the characteristic directions change often 
throughout the domain
-- the latter situation frequently arises in Eikonal problems
when the speed function $F$ is highly inhomogeneous.

We note that the sweeping approach can be in principle useful for a much wider class of problems.
For example, the method introduced in \cite{KaoOsherQian}
is applicable to problems with non-convex Hamiltonians corresponding to differential games; 
however, the amount of required artificial viscosity is strongly problem-dependent and the choice of 
consistently discretized boundary conditions can be complicated.
Sweeping algorithms for discontinuous Galerkin finite element discretizations of the Eikonal PDE
can be found in \cite{Li_1, Li_2}.
\end{remark}

The Fast Sweeping Method performs particularly well on problems where the speed function $F$ is constant,
since in this case the characteristics of the Eikonal PDE will be straight lines regardless of 
the boundary conditions.  (E.g., if $q \equiv 0$, then the quickest path is a straight line to 
the nearest boundary point.)  As a result, the domain consists of 4 subdomains, each with its
own characteristic ``quadrant direction''.
Even though these subdomains are generally not known in advance, 
it is natural to expect Fast Sweeping to converge in at most 4 iterations 
(e.g., if $\x_{ij}$'s characteristic comes from the SE, then the same is true for all points 
immediately to SE from $\x_{ij}$).  However, on the grid, the dependency graph can be more 
complicated -- $\x_{ij}$ will depend on both its southern and eastern neighbors.  
The characteristic directions are changing continuously everywhere, except at the shocks.
So, if $\x_{ij}$ is near a shock line, one of its neighbors might be in another subdomain,
making additional sweeps occasionally necessary even for such simple problems; see
Figure \ref{fig:more_than_4} for an illustration.

\begin{figure}[hhhh]
\center{
\begin{tikzpicture}[scale=1.25]

\draw[thick](0,0) rectangle +(4,4);
\foreach \x in {0.5, 1, 1.5, 2, 2.5, 3, 3.5}
	\draw[dotted](\x,0)--(\x,4);
\foreach \y in {0.5, 1, 1.5, 2, 2.5, 3, 3.5}
	\draw[dotted] (0, \y)--(4, \y);

\draw[black,very thick] (3.75, 0)--(1.75,2) -- (0, 2.25);
\draw[black,very thick] (1.25, 4)--(1.75,2) -- (4, 3.25);

\draw (3,1) circle(.6);

\draw[red,thick,->](4,.5)--(3,1);
\draw[red,thick,->](4, 1.25)--(2.5, 2);
\draw[red,thick,->](4, 2)--(3, 2.5);

\draw[red,thick,->](2,4)--(1.5, 3);
\draw[red,thick,->](3,4)--(2.25, 2.5);
\draw[red,thick,->](4,4)--(3.5, 3);

\draw[red,thick,->](1,4)--(1.25, 3.75);
\draw[red,thick,->](0,4)--(1.5, 2.5);
\draw[red,thick,->](0,3)--(0.75, 2.25);

\draw[red,thick,->](0,1)--(1.5, 2);
\draw[red,thick,->](0,0)--(2.25, 1.5);
\draw[red,thick,->](1.5,0)--(2.75, 0.8333);

\draw (7,2) circle(1.2);
\draw (3,.4) -- (7,.8);
\draw (3,1.6) -- (6.6,3.125);

\draw[dotted](7,0.8)--(7,3.2);
\draw[dotted](5.8,2)--(8.2,2);

\draw[red,thick](7.5, 1.75) -- (7,2);
\draw[red,thick,->](8,1.5)--(7.5, 1.75);

\filldraw(7,2) circle (.05); 
\draw (7.2,2.1) node{$\x_{ij}$};

\filldraw(8,2) circle (.05);
\filldraw(6,2) circle (.05);
\filldraw(7,1) circle (.05);
\filldraw(7,3) circle (.05);

\draw[black, very thick] (7.5, 1)--(6,2.5);

\end{tikzpicture}
}
\caption{
{\footnotesize
Four subdomains with a different update quadrant in each of them.
If the sweeping directions are used in the order $(SE,SW,NW,NE),$ 
then the node labeled $\x_{ij}$ near the shock line will not receive 
its final update until the 5th sweep, since its southern neighbor
lies in the southwest subdomain.
For simplicity, this example uses boundary conditions 
such that the characteristic directions are constant in each subdomain.
As a result, all subdomain boundaries coincide with shock lines,
which need not be the case in general, but the illustrated effect is generic.
}}
\label{fig:more_than_4}
\end{figure}
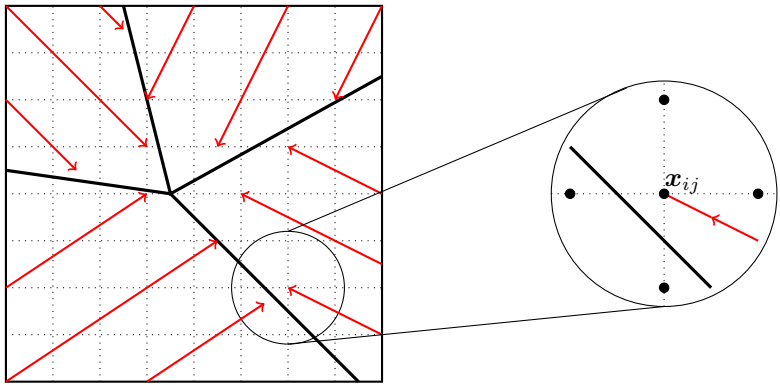

Nevertheless, when $F$ is constant, 
the Fast Sweeping is usually more efficient than the Fast Marching 
regardless of the boundary conditions
-- an observation which 
is the basis for the hybrid methods introduced in the next section.

\begin{remark}
\label{rem:regular_FSM_update}
It might seem that the recomputation of $V_{ij}$ from \eqref{eq:Eik_discr} will generally require 
solving 4 quadratic equations to compare the updates from all 4 quadrants.
However, the monotonicity property noted in Remark \ref{rem:discr_properties} guarantees that 
only one quadrant needs to be considered.
E.g., if $U_S < U_N$ then $U^{SE} \leq U^{NE}$ and the latter is irrelevant 
even if we are currently sweeping from NE.  Thus, the relevant quadrant can be always found by using
$\min(U_S, U_N)$ and $\min(U_E, U_W)$.
We note that this shortcut is not directly 
applicable to discretizations on unstructured meshes nor for more general PDEs.  
Interestingly, Alton and Mitchell showed that the same shortcut can also be used 
with Cartesian grid discretizations of Hamilton-Jacobi PDEs with grid-aligned 
anisotropy \cite{AltonMitchell_TR}. 
\end{remark}

\begin{remark}
\label{rem:locking_FSM}
One of the problems in this basic version of the Fast Sweeping Method is the fact that the CPU time
might be wasted to recompute $V_{ij}$ even if none of $\x_{ij}$'s neighbors have changed 
since the last sweep.
To address this, one natural modification is to introduce ``locking flags'' for individual 
gridpoints and to update the currently unlocked gridpoints only \cite{Renzi}.
Briefly, all gridpoints but those immediately adjacent to $Q$ start out as locked.  
When an unlocked gridpoint $\x_{ij}$ is processed during a sweep, if $U_{ij}$ changes, 
then all of its larger neighbors are unlocked.
The gridpoint $\x_{ij}$  is then itself locked regardless of whether updating $U_{ij}$ resulted in unlocking a neighbor.

The above modification does not change the asymptotic complexity of the method nor
the total number of sweeps needed for convergence. 
Nevertheless, the extra time and memory required to maintain and update the locking flags 
are typically worthwhile since their use allows to decrease the amount of CPU-time wasted on parts of 
the domain, where the iterative process already produced the correct numerical solutions.
In sections \ref{s:our_methods} and \ref{s:experiments} we will refer to this modified version 
as Locking Sweeping Method (LSM) to distinguish it from the standard implementation of the FSM.
\end{remark}



\subsection{Other fast methods for Eikonal equations}
\label{ss:correcting_Eikonal}

Ideas behind many label-correcting algorithms on graphs have also been
applied to discretizations of Eikonal PDEs.  Here we aim to briefly 
highlight some of these connections.

Perhaps the first label-correcting methods developed for the Eikonal PDE were introduced by
Polymenakos, Bertsekas, and Tsitsiklis based on the logic of the discrete SLF/LLL algorithms \cite{PolyBerTsi}.
On the other hand, Bellman-Ford is probably the simplest label-correcting approach and it has been recently
re-invented by several numerical analysts working with Eikonal and more general 
Hamilton-Jacobi-Bellman PDEs \cite{BorRasch}, \cite{Renzi}, including implementations for 
massively parallel computer architectures \cite{JeongWhitaker}.
A recent paper by Bak, McLaughlin, and Renzi \cite{Renzi} also introduces another 
``2-queues method'' essentially mimicking the logic of thresholding label-correcting algorithms on graphs.
While such algorithms clearly have promise and some numerical comparisons of 
them with sweeping and marching techniques are already presented in the above references,
more careful analysis and testing is required to determine the types of examples on which they
are the most efficient.

All of the above methods produce the exact same numerical solutions as FMM and FSM.
In contrast, two of the three new methods introduced in section \ref{s:our_methods} aim 
to gain efficiency even if it results in small additional errors.
We know of only one prior numerical method for Eikonal PDEs with a similar trade-off: 
in \cite{Sapiro_buckets} a Dial-like method is used with buckets of unjustified width 
$\delta$ for a discretization that is not $\delta$-causal.  
This introduces additional errors (analyzed in \cite{RaschSatzger}), 
but decreases the method's running time.
However, the fundamental idea behind our new two-scale methods is quite different,
since we aim to exploit the geometric structure of the speed function.

\section{New hybrid (two-scale) fast methods}
\label{s:our_methods}
We present three new hybrid methods based on splitting the domain into a collection of 
non-overlapping rectangular ``cells''
and running the Fast Sweeping Method on individual cells sequentially.  
The motivation for this 
decomposition is to break the problem into sub-problems, with $F$ nearly constant 
inside each cell.  If the characteristics rarely change their quadrant-directions 
within a single cell, then a small number of sweeps should be sufficient on that cell.
But to compute the value function correctly within each cell, 
the correct boundary conditions 
(coming from the adjacent cells) should be already available.  
In other words, we need to establish a causality-respecting order for processing the cells.  
The Fast Marching Sweeping Method (FMSM) uses the cell-ordering found by running the 
Fast Marching Method on a coarser grid, while the Heap-Cell Methods (HCM and FHCM)
determine the cell-ordering dynamically, based on the value-updates on cell-boundaries.

\subsection{Fast Marching-Sweeping Method (FMSM)}
\label{ss:FMSM}
This algorithm uses a coarse grid and a fine grid.
Each ``coarse gridpoint'' is taken to be the center of a cell 
of ``fine gridpoints''.  (For simplicity, we will assume that the exit-set $Q$ 
is directly representable by coarse gridpoints.)
The Fast Marching is used on the coarse grid, and the Acceptance-order of 
``coarse gridpoints'' is recorded.  The Fast Sweeping is then used on the corresponding 
cells in the same order.  An additional speed-up is obtained, by running 
a fixed number of sweeps on each cell, based on the upwind directions 
determined on the coarse grid.  Before providing the details of our implementation,
we introduce some relevant notation:\\
$\bullet \,$ $X^{c} = \{ \x_1^{c}, ..., \x_J^{c}\} $, the coarse grid.\\
$\bullet \,$ $X^{f} = \{ \x_1^{f}, ..., \x_M^{f}\} $, the fine grid (same as the grid used in FMM or FSM).\\
$\bullet \,$ $Q^{c} \subset X^{c} $, the set of coarse gridpoints discretizing the exit set $Q$.\\
$\bullet \,$ $U^{c}$, the solution of the discretized equations on the coarse grid.\\
$\bullet \,$ $V^{c}$, the temporary label of the coarse gridpoints.\\
$\bullet \,$ $Z = \{c_1, ..., c_J\}$, the set of cells, whose centers correspond to coarse gridpoints.\\
$\bullet \,$ $N^c(c_i)$, the neighbors of cell $c_i$; i.e., the cells that exist to the north, south, east, and west of $c_i$.\\
(The set $N^c(c_i)$ may contain less than four elements if $c_i$ is a boundary cell.)\\ 
$\bullet \,$ $N^f(c_i)$, the fine grid neighbors of $c_i$; i.e., 
$N^f(c_i) = \{ \x_j^f \in X^f \mid \x_j^f \not \in c_i \text{ and } N(\x_j^f) \bigcap c_i \neq \emptyset \}.$\\
$\bullet \,$ $P: \{1,...,J\} \rightarrow \{1,...,J\}$, a permutation on the coarse gridpoint indices.\\
$\bullet \,$ $h_{x}^{c}$, the distance along the $x$-direction between two neighboring coarse gridpoint.\\
Assume for simplicity that $h_{x}^{c} = h_{y}^{c} = h^{c}$.\\

All the obvious analogs hold for the fine grid ($U^{f}, h^{f}$, etc).  
Since Fast Marching will be used on the coarse grid only, the heap $L$ will contain coarse gridpoints only.

\begin{algorithm}[hhhh]
\caption{Fast Marching-Sweeping Method pseudocode.}
\label{alg:FMSM}
\algsetup{indent=2em}
\begin{algorithmic}[1]

\STATE Part I:
	\STATE Run FMM on $X^c$ (see algorithm \ref{alg:FMM}).
	\STATE Build the ordering $P$ to reflect the Acceptance-order on $X^c$.
\STATE	
\STATE Part II:
\STATE Fine grid initialization:
\FOR{each gridpoint $\x_i^f \in X^f$}
	\IF{$\x_i^{f} \in Q^{f}$}
		\STATE $V_i^{f} \gets q_i^{f};$
	\ELSE
		\STATE $V_i^{f} \gets \infty;$
	\ENDIF
\ENDFOR
\STATE
	\FOR{$j = P(1):P(J)$}
		\STATE Define the fine-grid domain $\tilde{c} = c_j \bigcup N^f(c_j).$
		\STATE Define the boundary condition as 
		\STATE $\qquad \qquad \tilde{q}(\x^f_i) = q(\x^f_i)$  on $c_j \bigcap Q^f$ and
		\STATE $\qquad \qquad \tilde{q}(\x^f_i) = V^f_i$  on $N^f(c_j).$
		\STATE Perform Modified Fast Sweeping (see Remark \ref{rem:sweeping_in_FMSM}) on $\tilde{c}$ using boundary conditions $\tilde{q}$.
	\ENDFOR
\STATE
	 
\end{algorithmic}
\end{algorithm}

\begin{remark}
\label{rem:sweeping_in_FMSM}
The ``Modified Fast Sweeping'' procedure applied to individual cells in the algorithm \ref{alg:FMSM} 
follows the same idea as the FSM described in section \ref{ss:FS}.
For all the cells containing parts of $Q$ 
(i.e., the ones whose centers are Accepted {\em in the initialization} of the FMM
on the coarse grid) we use the FSM without any changes.  
For all the remaining cells, our implementation has 3 important distinctions 
from the algorithm \ref{alg:FSM}: 
\begin{enumerate}
\item
No initialization of the fine gridpoints within $\tilde{c}$ is needed
since the entire fine grid is pre-initialized in advance.
\item
Instead of looping through different sweeps until convergence,
we use at most four sweeps and only in the directions found to be ``upwind'' 
on the coarse grid. 
As illustrated by Figure \ref{fig:coarse_sweep_choice}, 
the cells in $N^c(c_i)$ whose centers were accepted prior to $\x_i^c$ 
determine the sweep directions to be used on $c_i$.
\item
When computing $V^f_i$ during the sweeping, we do not employ
the procedure described in Remark \ref{rem:regular_FSM_update} 
to find the relevant quadrant.  
Instead, we use ``sweep-directional updates''; e.g., 
if the current sweeping direction is from the NE, we  
always use the update based on the northern and eastern neighboring fine gridpoints.
The advantage is that we have already processed both of them within the same sweep.
\end{enumerate}
\end{remark}


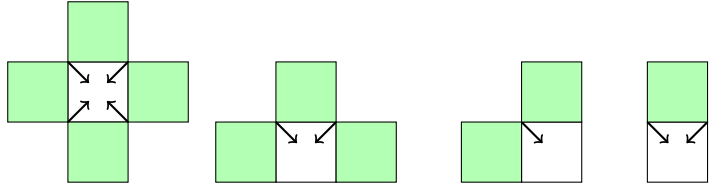
\begin{figure}[hhhh]
\center{
$
\begin{array}{lcrr}

\begin{tikzpicture}[scale=0.4]
\draw (0,0) rectangle +(2,2); 
\draw[thick,->](0,0)--(.7,.7);
\draw[thick,->](2,0)--(1.3,.7);
\draw[thick,->](0,2)--(.7,1.3);
\draw[thick,->](2,2)--(1.3,1.3);

\filldraw[fill = green!30!white, draw = black] (-2,0) rectangle +(2,2) ;
\filldraw[fill = green!30!white, draw = black] (0,2) rectangle +(2,2) ;
\filldraw[fill = green!30!white, draw = black] (2,0) rectangle +(2,2) ;
\filldraw[fill = green!30!white, draw = black] (0,-2) rectangle +(2,2) ;
\end{tikzpicture}

&

\begin{tikzpicture}[scale=0.4]
\draw (0,0) rectangle +(2,2); 
\draw [thick,->] (0,2)--(.7,1.3);
\draw[thick,->](2,2)--(1.3,1.3);

\filldraw[fill = green!30!white, draw = black] (-2,0) rectangle +(2,2) ;
\filldraw[fill = green!30!white, draw = black] (0,2) rectangle +(2,2) ;
\filldraw[fill = green!30!white, draw = black] (2,0) rectangle +(2,2) ;
\end{tikzpicture}

&
\hspace*{5mm}
\begin{tikzpicture}[scale=0.4]
\draw (0,0) rectangle +(2,2); 
\draw[thick,->](0,2)--(.7,1.3);

\filldraw[fill = green!30!white, draw = black] (-2,0) rectangle +(2,2) ;
\filldraw[fill = green!30!white, draw = black] (0,2) rectangle +(2,2) ;
\end{tikzpicture}

&

\hspace*{5mm}
\begin{tikzpicture}[scale=0.4]
\draw (0,0) rectangle +(2,2); 
\draw [thick,->] (0,2)--(.7,1.3);
\draw[thick,->](2,2)--(1.3,1.3);

\filldraw[fill = green!30!white, draw = black] (0,2) rectangle +(2,2) ;
\end{tikzpicture}

\end{array}
$
}
\caption{
{\footnotesize
Sweeping directions on $c_i$ chosen based on the neighboring cells
accepted earlier than $c_i$ (shown in green).  Note that 2 sweeping directions are 
conservatively used in the case of a single accepted neighbor.}}
\label{fig:coarse_sweep_choice}
\end{figure}

\iffullversion
Before discussing the computational cost and accuracy consequences of these 
implementation choices, we illustrate the algorithm on a specific example:
a checkerboard domain with the speed function $F=1$ in white and $F=2$ in black checkers,
and the exit set is a single point in the center of the domain see Figure \ref{fig:checkers_FMSM}).
This example was considered in detail in \cite{ObTaVlad}; 
the numerical results and the performance of our new methods
on the related test problems are described in detail in section \ref{ss:checkers}.
As explained in Remark \ref{rem:sweeping_in_FMSM}.2, 
we do not sweep until convergence on each cell; 
e.g., the sweeps for the cell \# 1 in Figure \ref{fig:checkers_FMSM} will be from
northwest and southwest, while the cell \#14 will be swept from northeast only.

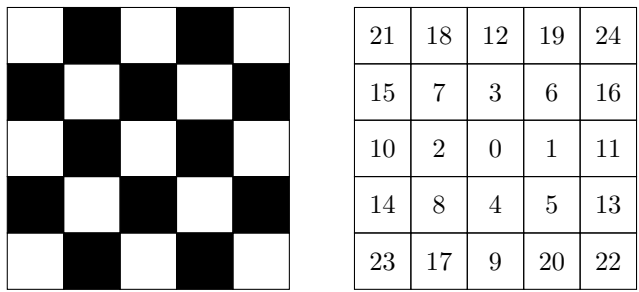
\begin{figure}[hhhh]
\center{
$
\begin{array}{cc}

\begin{tikzpicture} [scale = 1.5]
\draw (0,0) rectangle (.5,.5)  (.5,.5) rectangle +(.5,.5)  (1,0) rectangle +(.5,.5)  (-1,0) rectangle +(.5,.5)  (-.5,.5) rectangle +(.5,.5);
\draw (-.5,.5) rectangle (.5,.5) (-.5,-.5) rectangle +(.5,.5)  (-1,1) rectangle +(.5,.5) (-1,-1) rectangle +(.5,.5)  (0,-1) rectangle +(.5,.5);
\draw (0,1) rectangle +(.5,.5)  (1,1) rectangle +(.5,.5)  (1,-1) rectangle +(.5,.5)  (.5,-.5) rectangle +(.5,.5);

\filldraw[fill=black, draw= black] (.5,0) rectangle +(.5,.5)  (0,.5) rectangle +(.5,.5)  (-.5,0)rectangle +(.5,.5)  (0,-.5) rectangle +(.5,.5);
\filldraw[fill=black, draw= black] (-1,.5) rectangle +(.5,.5)  (-1,-.5) rectangle +(.5,.5)  (.5,1)rectangle +(.5,.5)  (.5,-1) rectangle +(.5,.5);
\filldraw[fill=black, draw= black] (-.5,1) rectangle +(.5,.5)  (-.5,-1) rectangle +(.5,.5)  (1,-.5)rectangle +(.5,.5)  (1,.5) rectangle +(.5,.5);
\end{tikzpicture}

&

\hspace*{5mm}
\begin{tikzpicture}[scale = 0.75]
\foreach \x in {1,2,...,5}
\foreach \y in {1,...,5}
{
\draw (\x,\y) +(-.5,-.5) rectangle ++(.5,.5);
}
\draw(1,1) node{23}   (1,2) node{14}  (1,3) node{10}  (1,4) node{15}  (1,5) node{21};
\draw(2,1) node{17}   (2,2) node{8}  (2,3) node{2}  (2,4) node{7}  (2,5) node{18};
\draw(3,1) node{9}   (3,2) node{4}  (3,3) node{0}  (3,4) node{3}  (3,5) node{12};
\draw(4,1) node{20}   (4,2) node{5}  (4,3) node{1}  (4,4) node{6}  (4,5) node{19};
\draw(5,1) node{22}   (5,2) node{13}  (5,3) node{11}  (5,4) node{16}  (5,5) node{24};
\end{tikzpicture}

\end{array}
$
}
\caption{
{\footnotesize
Left: The $5 \times 5$ checkerboard domain with a source point in the slow checker in the center.
Right: The order of cell-acceptance in Part I of FMSM, assuming that the size of cells and checker is the same.}}
\label{fig:checkers_FMSM}
\end{figure}
\fi

The resulting algorithm clearly introduces additional numerical errors -- 
in all but the simplest examples, the FMSM's output
is not the exact solution of the discretized system \eqref{eq:Eik_discr}
on $X^f$.  We identify three sources of additional errors:
the fact that the coarse grid computation does not capture all cell 
interdependencies, and the two cell-sweeping modifications described in 
Remark \ref{rem:sweeping_in_FMSM}.  Of these, the first one is by far 
the most important.  Focusing on the fine grid, we will say that the
cell $c_i$ {\em depends on} $c_j \in N^c(c_i)$ if there exists a gridpoint
$\x^f_k \in c_i$ such that $U^f_k$ directly depends on 
$U^f_l$ for some gridpoint $\x^f_l \in c_j$.
In the limit, as $h^f \to 0$, this means that $c_i$ depends on $c_j$ if
there is a characteristic going from $c_j$ into $c_i$ (i.e., at least
a part of $c_i$'s boundary shared with $c_j$ is {\em inflow}).
For a specific speed function $F$ and a fixed cell-decomposition $Z$,
a causal ordering of the cells need not exist at all.  
As shown in Figure \ref{fig:cell_interdep}, two cells may easily 
depend on each other.
This situation arises even for problems where $F$ is constant on each 
cell; see Figure \ref{fig:checkers_computed}.
Moreover, if the cell refinement is performed 
uniformly, such non-causal interdependencies will be present even as
the cell size $h^c \to 0$.  This means that every algorithm
processing each cell only once (or even a fixed number of times)
will unavoidably introduce additional errors at least for some speed 
functions $F$. 

\begin{figure}[hhhh]
\center{
\begin{tikzpicture}
[scale = 2]
\draw[thick] (-1,1) rectangle +(2,-2);
\draw (.6,0) circle(.15 cm);

\draw[-] (.6,-.15)--(2.73,-.96);
\draw[-] (.6,.15)--(2.73,.96);
\draw (3,0) circle(1.0cm);
\draw[dotted](3,1)--(3,-1);
\draw [red,very thick,->] (3.27,.96) .. controls +(left:0.8cm) and +(up:1cm)  .. (3.27,-.1);
\draw [red,very thick,->] (3.27,-.1) .. controls +(down:0.5cm) and (2.97, -.55)  .. (2.47,-.75);

\draw(2.6,0) node{$c_i$};
\draw(3.6,0) node{$c_j$};

\foreach \x in {-1,-.6,-.2,.2,.6,1}
	\draw[dotted](\x,-1)--(\x,1);
\foreach \y in {-1,-.6,-.2,.2,.6,1}
	\draw[dotted] (-1,\y)--(1,\y);
\end{tikzpicture}
}
\caption{
{\footnotesize
Two mutually dependent cells.}}
\label{fig:cell_interdep}
\end{figure}
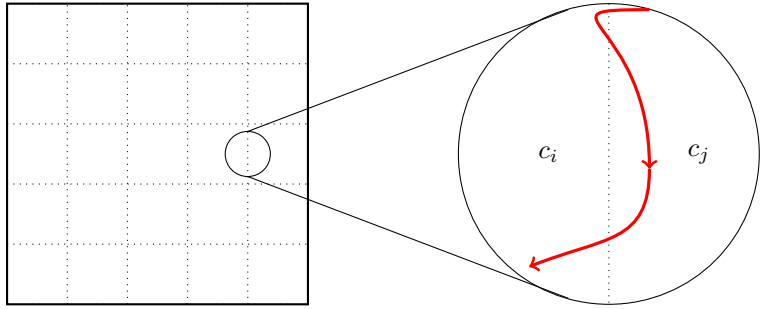
  
One possible way around this problem is to use the characteristic's 
vacillations between $c_i$ to $c_j$ to determine the total number 
of times that these cells should be alternately processed with FSM.
This idea is the basis for heap-cell methods described in the next section. 
However, for the FMSM we simply treat these ``approximate cell-causality'' errors
as a price to pay for the higher computational efficiency.
Our numerical experiments with FMSM showed that,  as $h^c \to 0$, the effects due to
the approximate cell-causality dominate the errors stemming from using 
a finite (coarse-grid determined) number of sweeps. 
I.e., when the cells are sufficiently small, running FSM to convergence does not 
decrease the additional errors significantly, but does noticeably increase 
the computational cost.  The computational savings due to our use of 
``sweep-directional updates'' are more modest
(we simply avoid the necessity to examine/compare all neighbors of the updated node), 
but the numerical evidence indicates that it introduces only small 
additional errors and only near the shock lines, where $\nabla u$ is undefined.  
Since characteristics do not emanate from shocks, the accuracy price of this modification 
is even more limited if the errors are measured in $L_1$ norm.
In section \ref{s:experiments} we show that on most of $X^f$
the cumulative additional errors in FMSM
are typically much smaller than the discretization errors,
provided $h^c$ is sufficiently small.
  
The monotonicity property of the discretization ensures that 
the computed solution $V^{f}$ will always satisfy $V_i^{f} \geq U_i$.
The numerical evidence suggests that $V^{f}$ becomes closer to $U^f$
as $h^c$ decreases, though this process is not always monotone.

The computational cost of Part I is relatively small as long as $J \ll M.$
However, if $h^f$ and $M$ are held constant while $h^c$ decreases,
this results in $J \to M$, and the total computational cost of FMSM
eventually increases.  
As of right now, we do not have any method for predicting the optimal $h^c$
for each specific example.   Such a criterion would be obviously useful for realistic 
applications of our hybrid methods, and we hope to address it in the future.

\subsection{Label-correcting methods on cells}
\label{ss:FHCM}

The methods presented in this section also rely on
the cell-decomposition $Z=\{c_1, \ldots, c_J\}$, but do not use any coarse-level grid.  
Thus, $X=X^f$ and we will omit the superscripts $f$ and $c$
with the exception of $N^c(c_i)$ and $N^f(c_i)$.  
We will also use $h^c$ to denote the distance between the centers of two adjacent square cells.  
In what follows, we will also define ``cell values" to represent coarse-level information about cell 
dependencies.  Unlike in finite volume literature, here a ``cell value" is not necessarily synonymous with the average of a 
function over a cell.

\subsubsection{A generic cell-level convergent method}

To highlight the fundamental idea, we start with a simple ``generic'' version
of a label-correcting method on cells.
We maintain a list of cells to be updated, starting with the cells containing 
the exit set $Q$.  While the list is non-empty, we choose a cell to remove from it,
``process'' that cell (by any convergent Eikonal-solver), and use the new grid values near 
the cell boundary to determine which neighboring cells should be added to the list.
The criterion for adding cells to the list is illustrated in Figure \ref{fig:cell_inflow_boundary}.
All other implementation details are summarized in Algorithm \ref{alg:generic_LC_cells}.

\begin{figure}[hhhh]
\center{
\begin{tikzpicture}
[scale = 0.7]
\draw[-] (0,0) -- (0,-4);
\draw[-] (-4,0) -- (4,0);

\draw(2,-2) node{cell $c_k$};
\draw(-2,-2) node{cell $c_l$};
\filldraw (.15,-.7) circle (.1 cm);
\filldraw (.15,-1.4) circle (.1 cm);
\filldraw (.15,-2.1) circle (.1 cm);
\filldraw (.15,-2.8) circle (.1 cm);

\filldraw (-.15,-.7) circle (.1 cm);
\filldraw (-.15,-1.4) circle (.1 cm);
\filldraw (-.15,-2.1) circle (.1 cm);
\filldraw (-.15,-2.8) circle (.1 cm);

\draw(.55,-2.8) node{$\x_j$};
\draw(-.55,-2.8) node {$\x_i$};

\end{tikzpicture}
}
\caption{
{\footnotesize
Suppose that, as a result of processing the cell $c_l$ an eastern border value $V_i$ becomes updated.  
If $V_i < V_j$ and $\x_j \not \in Q$, the cell $c_k$ will be added to $L$ (unless it is already on the list).
}}
\label{fig:cell_inflow_boundary}
\end{figure}
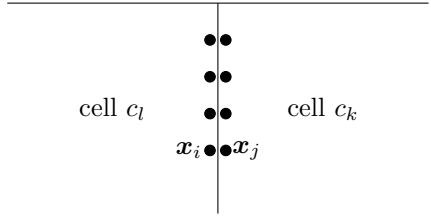

\begin{algorithm}[h]
\caption{Generic Label-Correcting on Cells pseudocode.}
\label{alg:generic_LC_cells}

\algsetup{indent=2em}
\begin{algorithmic}[1]

\STATE Cell Initialization:
\FOR{each cell $c_k$}
	\IF{$c_k \cap Q \neq \emptyset$}
		\STATE add $c_k$ to the list $L$
	\ENDIF
\ENDFOR

\STATE

\STATE Fine Grid Initialization:
\FOR{each gridpoint $\x_i$}
	\IF{$\x_i \in Q$}
		\STATE $V_i \gets q(\x_i)$
	\ELSE
		\STATE $V_i \gets \infty$
	\ENDIF
\ENDFOR
\STATE

\STATE Main Loop:
\WHILE{$L$ is nonempty}
	\STATE Remove a cell $c$ from the list $L$.
	\STATE Define a domain $\tilde{c} = c \cup N^f(c)$.
	\STATE Define the boundary condition as 
	\STATE $\qquad \qquad \tilde{q}(\x_i) = q(\x_i)$  on $c \cap Q$ and
	\STATE $\qquad \qquad \tilde{q}(\x_i) = V_i$  on $N^f(c).$
	\STATE Process $c$ by solving the Eikonal on $\tilde{c}$ using boundary conditions $\tilde{q}$.
	\FOR{each cell $c_k \in N^c(c) \backslash L$}
		\IF{$\quad \exists \, \x_i \in \left(c \cap N^f(c_k)\right) \;$ 
		    AND $\; \x_j \in \left(c_k \cap N(\x_i) \backslash Q \right)\;$ such that\\ 
		    $\qquad \quad( \; V_i$ has changed $\;\text{OR}\;$ 
					($\x_i \in Q \;$ AND $\; c$ is removed from $L$ for the first time) \, )\\
			$\qquad \quad$ AND $ \; (V_i < V_j) \quad$}
			\STATE $\qquad \qquad \qquad \qquad \qquad \qquad \qquad$ Add $c_k$ to the list $L$.
		\ENDIF
	\ENDFOR
\ENDWHILE

\end{algorithmic}
\end{algorithm}

It is easy to prove by induction that this method terminates in a finite number of steps;
in Theorem \ref{thm:convergence} we show that upon its termination $V=U$ 
on the entire grid $X$, regardless of the specific Eikonal-solver employed 
to process individual cells 
(e.g., FMM, FSM, LSM or any other method producing the exact solution to \eqref{eq:Eik_discr} will do).
We emphasize that the fact of convergence also does not depend on the specific
selection criteria for the next cell to be removed from $L$.
However, even for a fixed cell-decomposition $Z$, the above choices will 
significantly influence the total number of 
list removals and the overall computational cost of the algorithm.
One simple strategy is to implement $L$ as a queue, adding cells
at the bottom and always removing from the top, thus mirroring the logic
of Bellman-Ford algorithm.  In practice, we found the version described in 
the next subsection to be more efficient.

\begin{theorem}  
\label{thm:convergence}
The generic cell-based label-correcting method converges 
to the exact solution of system \eqref{eq:Eik_discr}.
\end{theorem}

\begin{proof}

\noindent
First we describe notation and recall from section \ref{ss:FS} the dependency digraph $G$.

$\bullet \,$ We say $\x_j$ \emph{depends on} $\x_i$ if $U_i$ is used to compute $U_j$ 
(see discussion of formulas \eqref{eq:one_sided_update}  and \eqref{eq:two_sided_update}).

$\bullet \,$ $\Gamma_{\blds{x}} = \{$nodes in $G$ on which $\x$ depends directly$\}$.  
For each node $\x$,  the set $\Gamma_{\blds{x}}$ will have 0, 1, or 2 elements.
If $\x \in Q$, then $\Gamma_{\blds{x}}$ is empty.  If  a one-sided update was used to compute $U(\x)$ 
(see formula \eqref{eq:one_sided_update}), then there is only one element in $\Gamma_{\blds{x}}$.

$\bullet \,$ $G_{\blds{x}}$ denotes the subgraph of $G$ that is reachable from the node $\x$.

$\bullet \,$ We define the
cell transition distance 
$d(\x) = max_{\x_i \in \Gamma_{\blds{x}}} \{d(\x_i) +$ cell$\_$dist$(\x,\x_i)  \}$,\\ 
where 
cell$\_$dist$(\x,\x_i) = 0$ if both $\x$ and $\x_i$ are in the same cell and $1$ otherwise.  
Note that in general $d(\x) < M$, but in practice $\max d(\x)$ is typically much smaller.  
In the continuous limit 
$d(\x)$ is related to the number of times a characteristic that reaches $\x$ crosses cell boundaries.

$\bullet \,$ $D_s = \{ \x \in G \mid d(\x) = s \}$.  See Figure \ref{fig:dependency_digraph} 
for an illustration of $G_{\blds{x}}$ split into $D_0, D_1, \ldots , D_{d(\x)}$.

$\bullet \,$ $\widetilde{D}_{s} = 
\{\x_j \in D_s \mid \exists \x_i \in D_{s-1}$ such that $\x_j$ depends on $\x_i \, \}$, 
i.e., the set of gridpoints in $D_s$ that depend on a gridpoint in a neighboring cell.
Note that $\widetilde{D}_{0} = \emptyset$.

$\bullet \,$ $\widehat{D}_{s} = 
\{\x_i \in D_s \mid \exists \x_j \in D_{s+1}$ such that $\x_j$ depends on $\x_i \, \}$, 
i.e., the set of gridpoints in $D_s$ that influence a gridpoint in a neighboring cell.

$\bullet \,$ $\star$ denotes any method that exactly solves the Eikonal on $\tilde{c}$ 
(see line 20 of algorithm \ref{alg:generic_LC_cells}).

Recall that by the monotonicity property of the discretization \eqref{eq:Eik_discr}, 
the temporary labels $V_j$ will always be 
greater than or equal to $U_j$ throughout algorithm \ref{alg:generic_LC_cells}. 
Moreover, once $V_j$ becomes equal to $U_j$, this temporary label will not change 
in any subsequent applications of $\star$ to the cell $c$ containing $\x_j$.
The goal is to show that $V_j = U_j$ for all $\x_j \in X$ upon the termination of 
Algorithm \ref{alg:generic_LC_cells}.

\begin{figure}

\begin{tikzpicture}[scale = .7,>=stealth]
 \tikzstyle{every state}=[draw, shape=circle, inner sep=0mm, minimum size = 6mm]

\node[label=left:$\x$] (a) at (0,3) {};

\node (b) at (1,4.2) {};
\node (c) at (2, 3.6) {};
\node (d) at (2,4.8) {};
\node (e) at (3,4.2) {}; 

\node(f) at (5,5) {}; 
\node (g) at (6,2) {}; 
\node (h) at (6.5,2.8) {}; 
\node (i) at (6.5,1) {}; 

\node (j) at (10,4) {}; 
\node (k) at (9.5,2) {};

\node[label=right:$\x_q \in Q$] (q) at (18.5,2.5) {};
\node (l) at (17.5,4) {}; 
\node (m) at (17,1.5) {};

\foreach \p in {a,b,c,d,e,f,g,h,i,j,k,l,m,q}
\fill  (\p) circle (.15cm);

\path[->,thick](a) edge (b)
	(a) edge (g)
	(b) edge (c)
	(b) edge (d)
	(d) edge (e)
	(c) edge (e)
	(d) edge (f)
	(e) edge (f)
	(g) edge (h)
	(g) edge (i)
	(c) edge (j)
	(f) edge (j)
	(f) edge (h)
	(h) edge (k)
	(i) edge (k);

\path[->,thick] (j) edge (12.5,5) 
	(j) edge (12.5,3)
	(k) edge (12.5,1)

	(15.3,5) edge (l)
	(15.3,3) edge (l)
	(15.3,1) edge (m)
	(l) edge (q)
	(m) edge (q);

\draw(2,6) node{$D_{s+1}$};
\draw(6,6) node{$D_{s}$};
\draw(10,6) node{$D_{s-1}$};
\draw(18,6) node{$D_0$};

\draw[-](4,6)--(4,0);
\draw[-](8,6)--(8,0);
\draw[-](12,6)--(12,0);
\draw[-](16.5,6)--(16.5,0);

\filldraw (13.3,3) circle(.05 cm);
\filldraw (13.8,3) circle(.05 cm);
\filldraw (14.3,3) circle(.05 cm);

\filldraw (13.3,5) circle(.05 cm);
\filldraw (13.8,5) circle(.05 cm);
\filldraw (14.3,5) circle(.05 cm);

\filldraw (13.3,1) circle(.05 cm);
\filldraw (13.8,1) circle(.05 cm);
\filldraw (14.3,1) circle(.05 cm);

\end{tikzpicture}
\caption{A schematic view of dependency digraph $G_{\blds{x}}$.}
\label{fig:dependency_digraph}
\end{figure}

To prove convergence we will use induction on $s$.  
First, consider $s=0$ and note that every cell $c$ containing some part of $D_0$ is put in $L$ 
at the time of the cell initialization step 
of the algorithm.  When $c$ is removed from $L$ and $\star$ is applied to it, 
every $\x \in D_0 \cap c$ will obtain its final value $V(\x) = U(\x)$ because 
$G_{\blds{x}}$ contains no gridpoints in other cells by the definition of $D_0$.

Now suppose all $\x \in D_k$ already have $V(\x) = U(\x)$ for all $k \leq s$.
We claim that:

\noindent
1) If a cell $c$ contains any $\x \in D_{s+1}$ such that $V(\x) > U(\x)$, then this cell is guaranteed 
to be in $L$ at the point in the algorithm when the last $\x_i \in D_s \cap N^f(c)$ receives its final update.

\noindent
2)  The next time 
$\star$ is applied to $c$, $V(\x)$ will become equal to $U(\x)$ for all $\x \in D_{s+1} \cap c$.

To prove 1), suppose $D_{s+1} \cap c \neq \emptyset$ and note that there exist 
$\x_j \in \widetilde{D}_{s+1} \cap c$ and $\x_i \in \Gamma_{\blds{x}_j}$ with 
$\x_i \in \widehat{D_s} \cap \hat{c}$ for some neighboring cell $\hat{c}$.  
Indeed, if each gridpoint $\x$ $\in$ $D_{s+1} \cap c$ 
were to depend only on those in $D_{s+1}$ (gridpoints within the same cell) 
and/or those in $D_{k}$ for $k < s$, this would contradict $\x \in D_{s+1}$
(it is not possible for $ \Gamma_{\blds{x}} \subset \cup_{k<s} D_k$; 
see Figure \ref{fig:dependency_digraph}).  
At the time the \emph{last such} $\x_i$ receives its final update,
we will have $V_j \geq U_j > U_i = V_i$ since $\x_i \in \Gamma_{\blds{x}_j}$.
Thus, $c$ is added to $L$ (if not already there) as a result of the add criterion 
in Algorithm \ref{alg:generic_LC_cells}.  

To prove 2), we simply note that all nodes in 
$\left( G_{\blds{x}} \backslash c  \right) \subset \left( \bigcup_{k=0}^{s} D_k \right)$ 
will already have correct values at this point.

\end{proof}

\begin{remark}
\label{rem:subgraph_LC_algorithm}
We note that the same ideas are certainly applicable to finding shortest paths on graphs.
The Algorithm \ref{alg:generic_lc} can be similarly modified using a collection of non-overlapping subgraphs 
instead of cells, but so far we were unable to find any description of this approach
in the literature.
\end{remark}

\subsubsection{Heap-Cell Method (HCM)}
To ensure the efficiency of cell-level label-correcting algorithms, 
it is important to have the ``influential'' cells (on which most others depend)
processed as early as possible.
Once the algorithm produces correct solution $V=U$ on those cells, 
they will never enter the list again, and their neighboring cells will 
have correct boundary conditions at least on a part of their boundary.
The same logic can be applied repeatedly by always selecting for removal 
the most ``influential'' cells currently on the list.  
We introduce the concept of ``cell values'' $V^c_k = V^c(c_k)$
to estimate the likelihood of that cell influencing others 
(the smaller is $V^c_k$, the more likely is $c_k$ to influence
subsequent computations in other cells, 
and the higher is its priority of removal from the list).
In Fast Marching-Sweeping Method of section \ref{ss:FMSM}, 
the cell values were essentially defined 
by running FMM on the coarse grid.  That approach is not very
suitable here, since each cell $c_k$ might enter the list more than once
and it is important to re-evaluate $V^c_k$ each time this happens.
Instead, we define and update $V^c_k$ using the boundary values in the adjacent cells, 
and modify Algorithm \ref{alg:generic_LC_cells} to use the cell values as follows:
\begin{enumerate}
\item
Amend the cell initialization to set 
$$
V^c_k \gets \max\limits_{\x_i \in (c_k \cap Q)} q(\x_i)
\quad \text{ or } \quad 
V^c_k \gets \infty \text{ if } c_k \cap Q = \emptyset.
$$

\item 
Always remove and process the cell with the smallest value
currently on the list.  Efficient implementation 
requires maintaining $L$ as a heap-sort data structure --
hence the name of ``Heap-Cell Method'' (HCM) for the resulting 
algorithm.

\item
After solving the Eikonal on $c$, 
update the cell values for all $c_k \in N^c(c)$ 
(including those already in $L$).  Let $\bb_k$
be a unit vector pointing from the center of $c$ 
in the direction of $c_k$'s center and 
suppose that $\x_i$ has the largest current value among the 
gridpoints inside $c$ but adjacent to $c_k$; i.e.,
$\x_i = \argmax\limits_{\x_j \in (c \cap N^f(c_k))} V_j$.
Define $\y_i =\x_i + \frac{h + h^c}{2} \bb_k$.
Then 
\begin{eqnarray}
\label{eq:cell_value}
\widetilde{V}^c_k &\gets& V_i + \frac{(h+h^c)/2}{F\left( \y_i \right)};\\
\nonumber
V^c_k &\gets& \min \left( V^c_k, \, \widetilde{V}^c_k \right).
\end{eqnarray}

\end{enumerate}

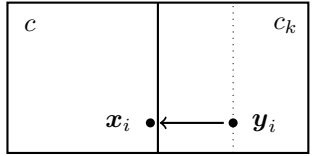
\begin{figure}[h]
\center
{

\begin{tikzpicture}[scale=2]

 \tikzstyle{every state}=[draw, shape=circle, inner sep=0mm, minimum size = 1mm]

\node[label=left:$\x_i$] (a) at (.95,.2) {};
\node[label=right:$\y_i$] (b) at (1.5,.2) {};

\draw(.15,.85) node{$c$};
\draw(1.85,.85) node{$c_k$};

\foreach \p in {a,b}
\fill  (\p) circle (.03cm);

\draw [thick] (0,0) rectangle (2,1);
\draw [thick,-] (1,0)--(1,1);
\draw[dotted] (1.5,0)--(1.5,1);
\path [thick,->] (b) edge (a);

\end{tikzpicture}
}
\caption{An illustration corresponding to equation \ref{eq:cell_value} 
(the estimate for a cell value) 
with $\bb_k = (1,0)$.}

\end{figure}


\begin{remark}
\label{rem:cell_value_updates}
We note that, in the original Dijkstra's and Bellman-Ford methods on graphs, 
a neighboring node's temporary label is updated {\em before} that node 
is added to $L$.   In the Heap-Cell Method,
the cell value is also updated before adding that cell to the list,
but the grid values within that cell are updated {\em after} it is removed 
from $L$.

Regardless of the method used to compute cell values, they can only 
provide an estimate of the dependency structure.  
As explained in section \ref{ss:FMSM}, a causal cell-ordering need not exist
for a fixed $Z$ and a general speed functions $F$.
Thus, $V^c_k < V^c_i$ does not exclude the possibility of 
$c_k$ depending on $c_i$, and we do not use cell values to decide 
which neighboring cells to add to the list -- 
this is still done based on
the cell boundary data; see Algorithm \ref{alg:generic_LC_cells}.  
As a result, the fact of convergence of such cell-level methods 
does not depend on the particular heuristic used to define cell values.
There are certainly many reasonable alternatives to formula 
\eqref{eq:cell_value} (e.g., a more aggressive/optimistic version can 
instead select $\x_i = \argmin V_j$ on the boundary; 
an average value of $F$ on $c_k$ could also be used here;
or the distance to travel could be measured from $\x_i$ to 
the center of $c_k$, etc).  Empirically, formula 
\eqref{eq:cell_value} results in smaller computational cost
than the mentioned alternatives and it was therefore used 
in our implementation.
\end{remark}

\begin{remark}
\label{rem:SLF_cells}
The cell-values are useful even 
if $L$ is implemented as a queue and the cells are always removed from the top.
Indeed, $V^c_k$ can still be used to decide whether $c_k$ should be added
at the top or at the bottom of $L$. 
This is the SLF/LLL strategy previously used to solve the Eikonal PDE
on the grid-level (i.e., without any cells) by 
Polymenakos, Bertsekas, and Tsitsiklis \cite{PolyBerTsi}.
We have also implemented this strategy and found it to be fairly good,
but on average less efficient than the HCM described above.  
(The performance comparison is omitted to save space.)
The intuitive reason is that the SLF/LLL is based on mimicking the logic of 
Dijkstra's method, but without the expensive heap-sort data structures.
However, when $J \ll M$, the cost of maintaining the heap is much smaller
than the cost of occasionally removing/processing less influential cells 
from $L$.
\end{remark}

To complete our description of HCM, we need to specify how the Eikonal PDE
is solved on individual cells.  Since the key idea behind our hybrid methods
is to take advantage of the good performance of sweeping methods when
the speed is more or less constant, we follow the same idea as 
the FSM described in section \ref{ss:FS}, 
but with the following important distinctions from 
the basic version of algorithm \ref{alg:FSM}:  
\begin{enumerate}
\item
No initialization of gridpoint values $V_i$ is needed within $\tilde{c}$ --
indeed, the initialization is carried out on the full grid 
at the very beginning and if $c$ is removed from $L$ more than once,
the availability of previously computed $V_i$'s might only speed up 
the convergence on $c$.  Here we take advantage of the comparison 
principle for the Eikonal PDE: 
the viscosity solution cannot increase anywhere inside the cell
in response to decreasing the cell-boundary values.
\item
We use the Locking Sweeping version described in Remark \ref{rem:locking_FSM}.
\item
The standard FSM and LSM loop through the four sweep directions 
always in the same order.
In our implementation of HCM, 
we choose a different order for the first four sweeps to ensure that 
the ``preferred sweep directions'' (determined for each cell individually) 
are used before all others.
For all other sweeps after the first four, we 
revert to the standard loop defined in Algorithm \ref{alg:FSM_order}.
Of course, as in the standard FSM, the sweeps only continue as long
as grid values keep changing somewhere inside the cell.
The procedure for determining preferred sweep directions
is explained in Remark \ref{rem:HCM_preferred_sweeps}.
\end{enumerate}

\begin{remark}
\label{rem:HCM_preferred_sweeps}
Recall that in FMSM, the coarse grid information was used to determine
the sweep directions to use on each cell; 
see Remark \ref{rem:sweeping_in_FMSM} and Figure \ref{fig:coarse_sweep_choice}.
Similarly, in HCM we use the neighboring cells of $c_k$ that
were found to have newly changed $c_k$-inflow boundary 
since the last time $c_k$ was added to $L$.
We maintain four ``directional flags'' -- 
boolean variables initialized to {\tt FALSE}
and representing all possible preferred sweeping directions 
-- for each cell $c_k$ currently in $L$.
When a neighboring cell $c_l$ is processed/updated and is found to influence $c_k$, 
this causes two of $c_k$'s directional flags to be set to {\tt TRUE}.
To illustrate, supposing that $c_l$ is a currently-processed-western-neighbor of
$c_k$ (as in Figure \ref{fig:cell_inflow_boundary}). 
If the value of $\x_i \in c_l \cap N^f(c_k)$ has 
just changed and $V_i < V_j$, then both relevant preferred direction flags
in $c_k$ (i.e., both NW and SW) will be raised.
Once $c_k$ is removed from $L$ and processed, its directional flags are reset 
to {\tt FALSE}.

As explained in section \ref{sss:FHCM},
a better procedure for setting these directional flags could be built 
based on fine-grid information on the cell-boundary. 
However, we emphasize that the procedure for determining preferred directions will
not influence the ultimate output of HCM (since we will sweep on $c_k$
until convergence every time we remove it from $L$), though such preferred directions 
are usually useful in reducing the number of sweeps needed for convergence. 
\end{remark}

The performance and accuracy data in section \ref{s:experiments} shows that, 
for sufficiently small $h$ and $h^c$, 
HCM often outperforms both FMM and FSM on a variety of examples, including those with piecewise
continuous speed function $F$.
This is largely due to the fact that the average number of times a cell 
enters the heap tends to 1 
as $h^c \to 0$. 

\subsubsection{Fast Heap-Cell Method (FHCM)}
\label{sss:FHCM}

We also implement an accelerated version of HCM by using the following modifications:
\begin{enumerate}
\item
Each newly removed cell is processed using at most four iterations -- 
i.e., it is only swept once in each of the preferred directions instead 
of continuing to iterate until convergence.
\item
Directional flags in all cells containing parts of $Q$ are initialized to {\tt TRUE}.
\item
To further speed up the process, we use a ``Monotonicity Check'' 
on cell-boundary data to further restrict the preferred sweeping directions.
For concreteness, assume that $c_l$ and $c_k$ are related as in Figure \ref{fig:cell_inflow_boundary}.
If the grid values in 
$N^f(c_k) \cap c_l$  are monotone non-decreasing from north to south,
we set $c_k$'s NW preferred direction flag to {\tt TRUE}; if those grid values are 
monotone non-increasing we flag SW; otherwise we flag both NW and SW.
(In contrast, both HCM and FMSM are always using two sweeps in this situation;
see Figure \ref{fig:coarse_sweep_choice} and Remark \ref{rem:HCM_preferred_sweeps}.)
We note that the set $c \cap N^f(c_k)$  already had to be examined 
to compute an update to $V^c_k$ and the above 
Monotonicity Check can be performed simultaneously. 
\end{enumerate}

\iffullversion
The resulting Fast Heap-Cell Method (FHCM) 
is significantly faster than HCM, 
but at the cost of introducing additional errors (see section \ref{s:experiments}).  

The Monotonicity Checks result in a considerable 
increase in performance since, for small enough $h^c$, most cell boundaries become monotone.
However, generalizing this procedure to higher dimensional cells 
is less straightforward.  
For this reason we decided against using Monotonicity Checks 
in our implementation of HCM.
FHCM is summarized in Algorithm \ref{alg:FHCM}.
\begin{algorithm}[hhhh]
\caption{Fast Heap-Cell Method pseudocode.}
\label{alg:FHCM}
\algsetup{indent=2em}
\begin{algorithmic}[1]
\STATE Cell Initialization:
	\IF{cell $c_k \ni \x$ for $\x \in Q^{f}$}
		\STATE Add $c_k$ to the list $L$;
		\STATE Tag all four sweeping directions of $c_k$ as \emph{true};
		\STATE Assign a cell value $V_k^{cell} := 0$;		
	\ELSE
		\STATE Assign a cell value $V_k^{cell} := \infty$;
	\ENDIF

\STATE Fine Grid Initialization:
	\IF{$\x_i^{f} \in Q^{f}$}
		\STATE $V_i^{f} := q_i^{f};$
	\ELSE
		\STATE $V_i^{f} := \infty;$
	\ENDIF

\STATE
\WHILE{$L$ is nonempty}
	\STATE Remove cell at the top of $L$;
	\STATE Perform Non-Directional Fast Sweeping within the cell according to its directions marked \emph{true}, then set all directions to \emph{false} and:
	\FOR{Each cell border N,S,E,W}
		\IF{the value of a gridpoint $\x_i^{f}$ along a border changes and $V_i^{f} < V_j^{f}$ for $\x_j^{f}$ a neighboring gridpoint across the border}
			\STATE Add the cell $c_k$ containing $x_j^{f}$ onto $L$ if not already there.
			\STATE Update the planned sweeping directions 
for $c_k$ based on the location of the cell containing $\x_i^{f}$ (more about this later). 
		\ENDIF
		\STATE{Compute a value $v$ for the neighbor cell $c_k$ (more about this later)}
		\IF{$v < V^{cell}_k$}
			\STATE($V_k^{cell}$) $\gets v$
		\ENDIF
	\ENDFOR
\ENDWHILE

\end{algorithmic}
\end{algorithm}

As an illustration, we consider another $5 \times 5$ checkerboard example 
(this time with a fast checker in the center) and show the contents of the heap
in Figure \ref{fig:FHCM_progress}.
\begin{figure}[H]
\center{
$
\begin{array}{cc}
\includegraphics[scale = .46] {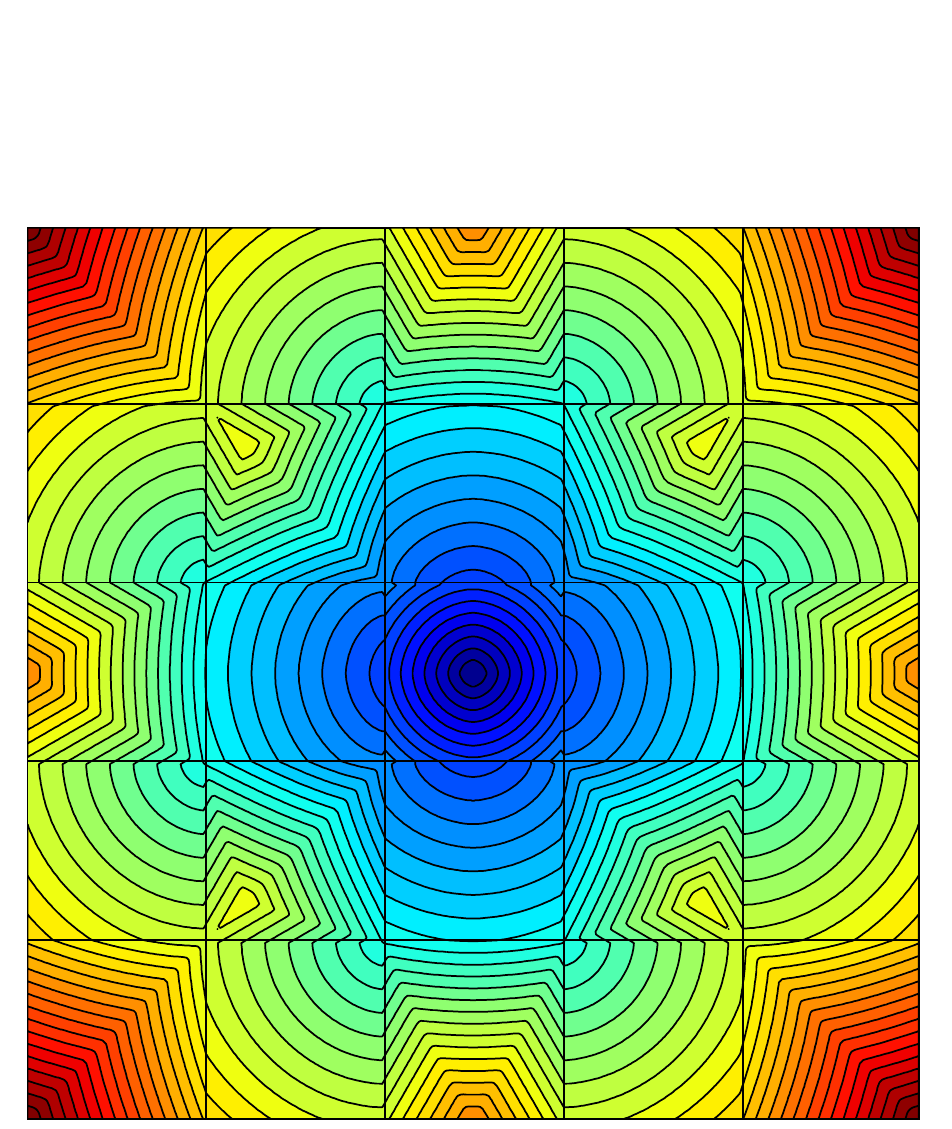}&
\hspace*{10mm}
\includegraphics[scale = .4] {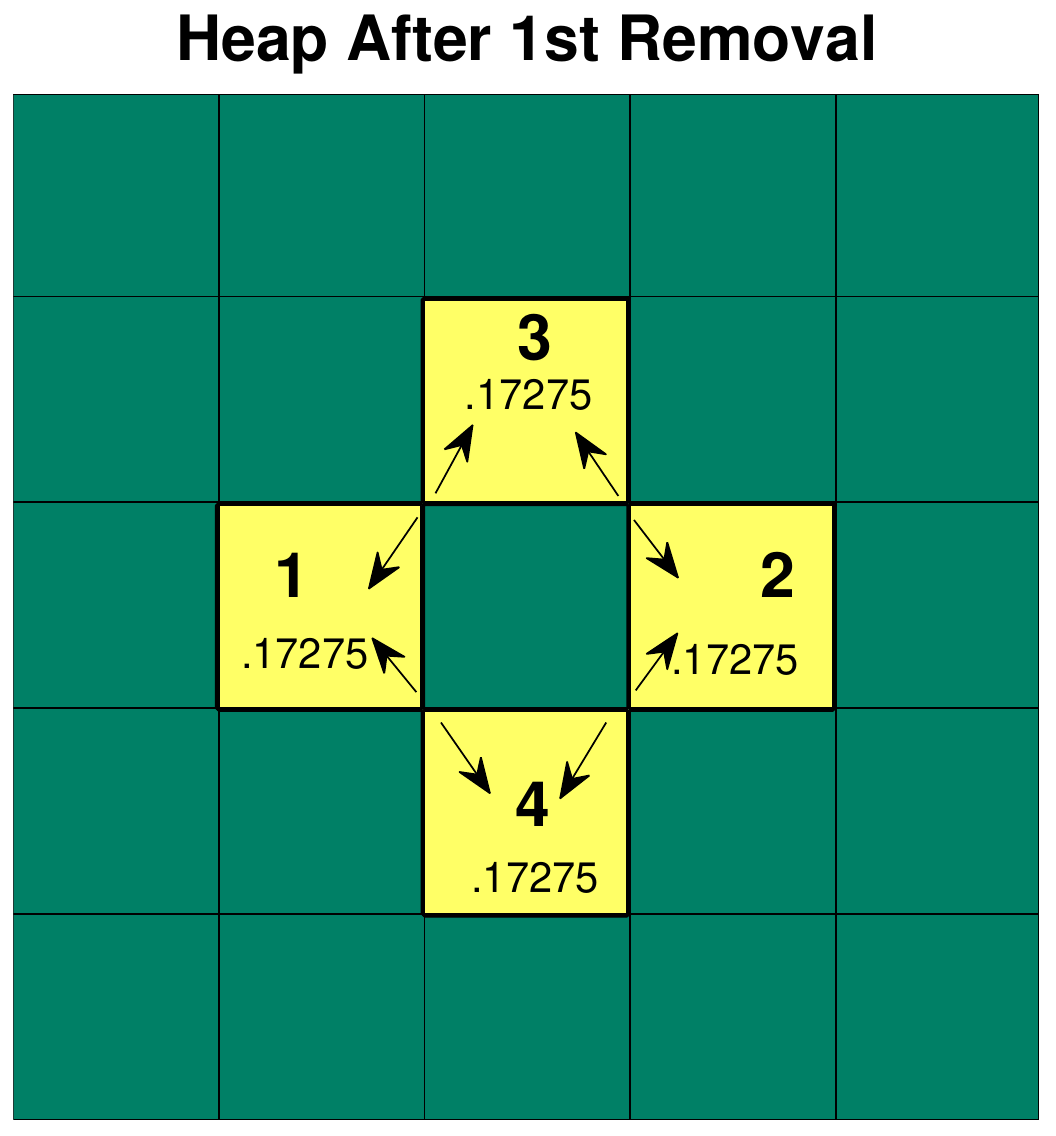}\\
A&B\\
\\
\includegraphics[scale = .4] {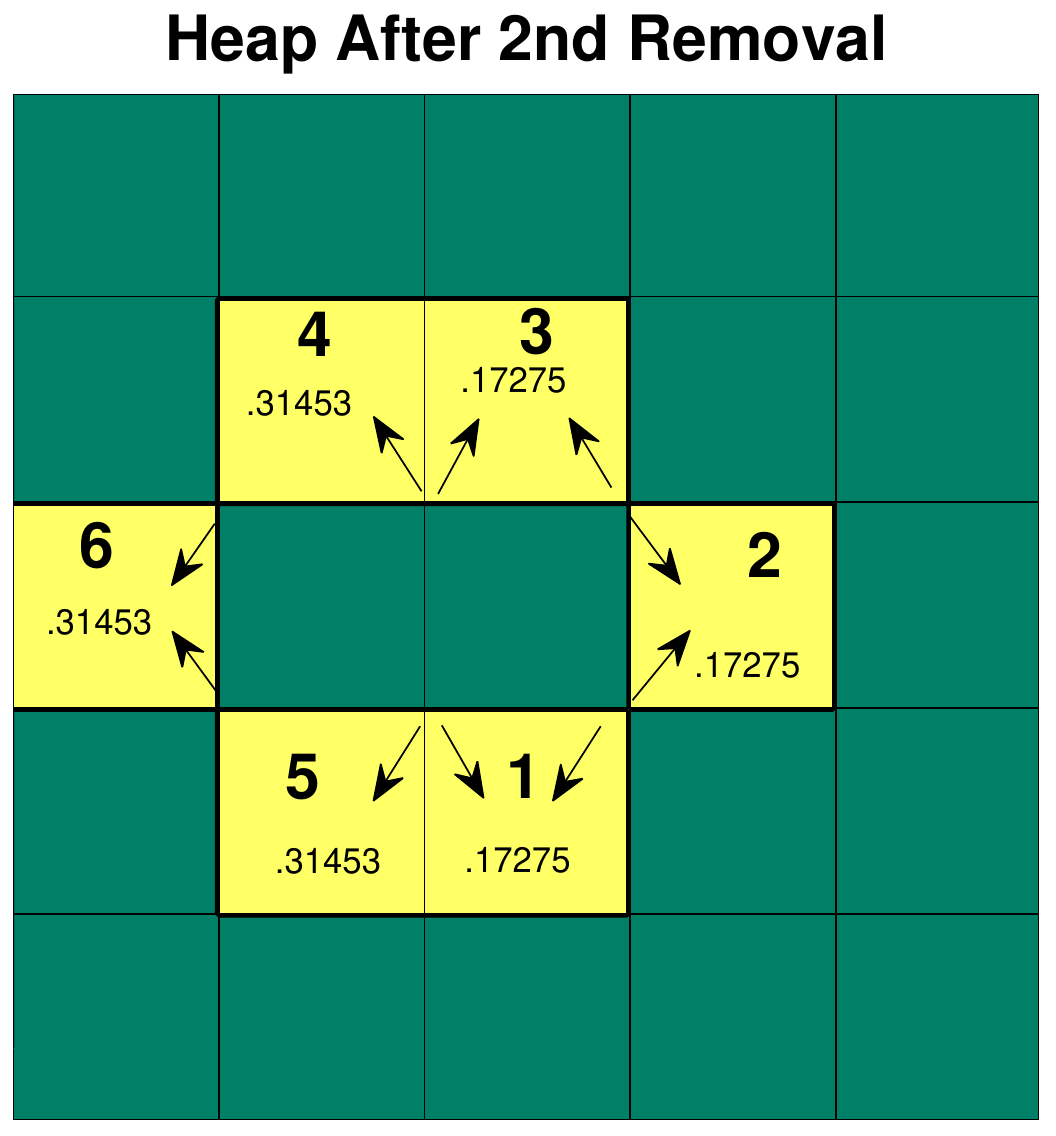}&
\hspace*{10mm}
\includegraphics[scale = .4]{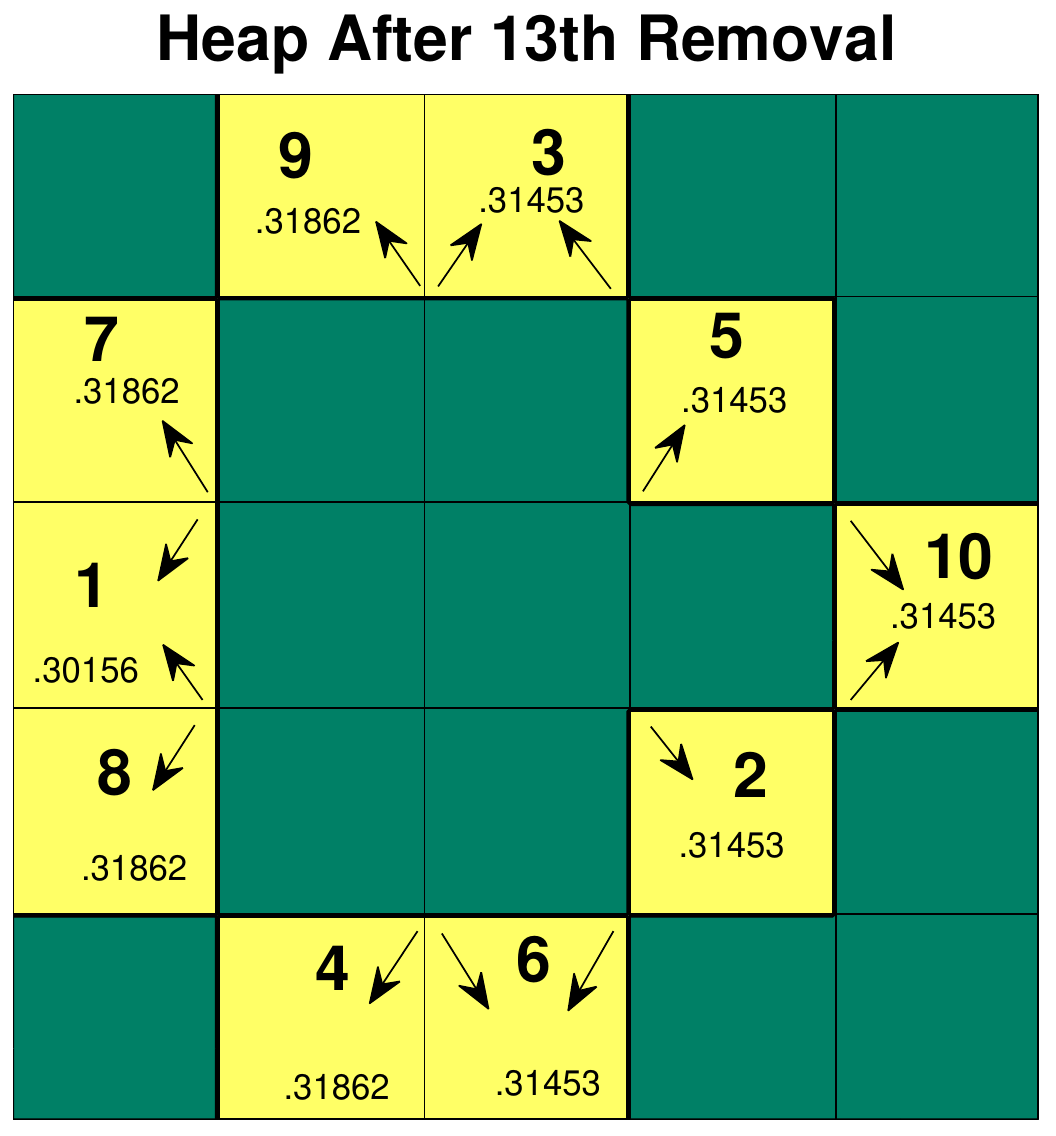}\\
C&D
\end{array}
$
}
\caption{
{\footnotesize
FHCM on a $5 \times 5$ checkerboard example. The level sets of the solution are shown in subfigure A.
The state of the cell-heap, current cell values and tagged preferred sweeping directions 
are shown after 1, 2, and 13 cell removals in subfigures B, C, and D.}}
\label{fig:FHCM_progress}
\end{figure}
Here we take the cells coinciding with checkers; 
finer cell-decompositions are numerically tested in section \ref{ss:checkers}.
The arrows indicate flagged sweeping directions for each cell,
and the smaller font is used to show the current cell values.  
Similarly to Dijkstra's method and FMM, the heap data structure is implemented 
as an array; the bold numbers represent each cell's index in this array.
In the beginning the central cell is the only one in $L$; once it is removed,
it adds to $L$ all four of its neighbors, all of them with the same cell value.
Once the first of these (to the west of the center) is removed,
it adds three more neighbors\footnote{We note that the cell indexed 4 after 
the first removal is indexed 2 immediately after the second.
This is the consequence of the performing {\em remove\_the\_smallest} using 
the {\em down\_heap} procedure in the standard implementation of the heap;
see \cite{SethSIAM}.}  (but not the central cell since there are 
no characteristics passing through the former into the latter).
This is similar to the execution path of FMSM, however, with heap-cell methods
the cells may generally enter the heap more than once. 
Thus, additional errors introduced by FHCM are usually smaller 
than those in FMSM.

\fi

\begin{remark}
\label{rem:WeberDevirBronsBronsKimmelHCM}
To conclude the discussion of our heap-cell methods, we briefly describe
a recent algorithm with many similar features, but very different goals
and implementation details.
The ``Raster scan algorithm on a multi-chart geometry image''
was introduced in \cite{WeberDevirBronsBronsKimmel}
for geodesic distance computations on parametric surfaces.
Such surfaces are frequently represented by an atlas of overlapping charts,
where each chart has its own parametric representation and grid resolution
(depending on the detail level of the underlying surface).  The computational 
subdomains corresponding to charts are typically large and the ``raster
scan algorithm'' (similar to the traditional FSM with a fixed ordering of 
sweep directions) is used to parallelize the computations within each chart.
The heuristically defined chart values are employed to decide which chart 
will be raster-scanned next. 

Aside from the difference in heuristic formulas used to compute chart 
values, in \cite{WeberDevirBronsBronsKimmel} the emphasis is on providing 
the most efficient implementation of raster scans on each chart
(particularly for massively parallel architectures).
The use of several large, parametrization/resolution-defined charts, 
typically results in complicated chart interdependencies since most 
chart boundaries are generally both inflow and outflow.  
Moreover, if this method is applied to any
Eikonal problems beyond the geodesic distance computations,
the monotonicity of characteristic directions will generally not hold
and a high number of sweeps may be needed on each chart.
In contrast, our focus is on reducing the cell interdependencies and on
the most efficient cell ordering:
when $h^c$ is sufficiently small, most cell boundaries 
are either completely inflow or outflow, defining a causal relationship
among the cells.  Relatively small cell sizes also ensure that $F$ 
is approximately constant, the characteristics are approximately straight lines,
and only a small number of sweeps is needed on each cell. Finally, the
cell orderings are also useful to accelerate the convergence within each cell
by altering the sweep-ordering
based on the location of upwinding cells (as in FMSM and HCM) or based on
fine-grid cell-boundary data (as in FHCM).  The hybrid methods 
introduced here show that causality-respecting domain decompositions
can accelerate even serial algorithms on single processor machines.
%
%
%
\end{remark}

\section{Numerical Experiments}
\label{s:experiments}

All examples were computed on a unit square $[0,1] \times [0,1]$ domain
with zero boundary conditions $q=0$ on the exit set $Q$ 
(defined separately in each case).
In each example that follows we have fixed the grid size $h = h^f$, and only the cell size $h^c$ is varied.
Since analytic formulas for viscosity solutions are typically unavailable,
we have used the Fast Marching Method on a much finer grid (of size $h/4$)
to obtain the ``ground truth'' used to evaluate the errors in all the other methods.

Suppose $e_i$ is the absolute value of the error-due-to-discretization at gridpoint $\x_i$
(i.e., the error produced by FSM or FMM when directly executed on the 
fine grid),
 and suppose $E_i$ is the absolute value of 
 the error committed by one of the new hybrid methods at the same $\x_i$.
Define the set $X_+ = \{ \x_i \in X \mid e_i \neq 0 \}$ and let $M_+ = |X_+|$ 
be the number of elements in it.  (We verified that 
$\x_i \not \in X_+ \, \Rightarrow E_i = 0$ 
in all computational experiments.)
To analyze the ``additional errors'' introduced by FMSM and FHCM,  we report\\
$\bullet \,$ the \emph{Maximum Error Ratio} defined as 
$\mathcal{R}=\max_i (E_i/e_i)$, where the maximum is taken over $\x_i \in X_+$;\\
$\bullet \,$ the \emph{Average Error Ratio} defined as 
$\mathcal{\rho}=\frac{\sum (E_i/e_i)}{M_+}$, where the sum is taken over $\x_i \in X_+$;\\
$\bullet \,$ the \emph{Ratio of Maximum Errors} defined as $R = \frac{\max_i(E_i)}{\max_i(e_i)}$.\\  
$R$ is relevant since on parts of the domain where $e_i$'s are very small,
additional errors might result in large $\mathcal{R}$ even if $E_i$'s are quite
small compared to the $L_\infty$ norm of discretization errors.
In the ideal scenario, with no additional errors, $\mathcal{R} = \mathcal{\rho} = R = 1.$

\noindent
For the Heap-Cell algorithms we also report\\ 
$\bullet \,$ {\em AvHR}, the average number of heap removals per cell,\\
$\bullet \,$ {\em AvS}, the average number of sweeps per cell, and\\ 
$\bullet \,$ {\em Mon \%}, the percentage of times that the ``cell-boundary monotonicity'' check was successful.\\
\noindent
Finally, we report the number of sweeps needed in FSM and LSM for each problem. 
 
\begin{remark}
\label{rem:fair_comparison}
Performance analysis of competing numerical methods is an obviously
delicate undertaking since the implementation details as well as the choice 
of test problems might affect the outcome.  
We have made every effort to select representative examples highlighting
advantages and disadvantages of all approaches.  All tests were performed on 
an AMD Turion 2GHz dual-core processor with 3GB RAM.  
Only one core was used to perform all tests.
Our C++ implementations were carefully checked for the efficiency of data 
structures and algorithms, but we did not 
conduct any additional performance tuning or Assembly-level optimizations.  
Our code was compiled using the {\tt g++} compiler version 3.4.2 with compiler
options {\tt -O0 -finline}.
We have also preformed all tests with the full compiler optimization (i.e., with {\tt -O3});
the results were qualitatively similar, but we opted to report the performance data for the
unoptimized version to make the comparison as compiler-independent as possible.
For each method, all memory allocations (for grids and heap data structures)
were not timed; the reported CPU times include the time needed to 
initialize the relevant data structures and run the corresponding algorithm.
We also note that the speed function $F(\x)$ was computed by a separate
function call whenever needed,
rather than precomputed and stored for every gridpoint during initialization. 
All CPU-times are reported in seconds
for the Fast Marching (FMM), the standard Fast Sweeping (FSM), the Locking Sweeping (LSM),
and the three new hybrid methods (HCM, FHCM, and FMSM).  
\end{remark}

\subsection{Comb Mazes}
\label{ss:combmazes}
The following examples model optimal motion through a maze with slowly permeable barriers.
Speed function $F(x,y)$ is defined by a ``comb maze'': $F=1$ outside and $0.01$ inside the barriers; 
see Figure \ref{fig:comb_maze}.
The exit set consists of the origin: $Q = \{(0,0)\}$. 
The computational cost of sweeping methods is roughly proportional to the number of barriers,
while FMM is only minimally influenced by this.
The same good property is inherited by the hybrid methods introduced in this paper.
The first example with 4 barriers uses barrier walls aligned with cell boundaries
and all hybrid methods easily outperform the fastest of the previous methods (LSM); 
see Table \ref{tab:4_comb}.


\begin{figure}[h]
\center{
$
\begin{array}{cc}
\includegraphics[scale = .3] {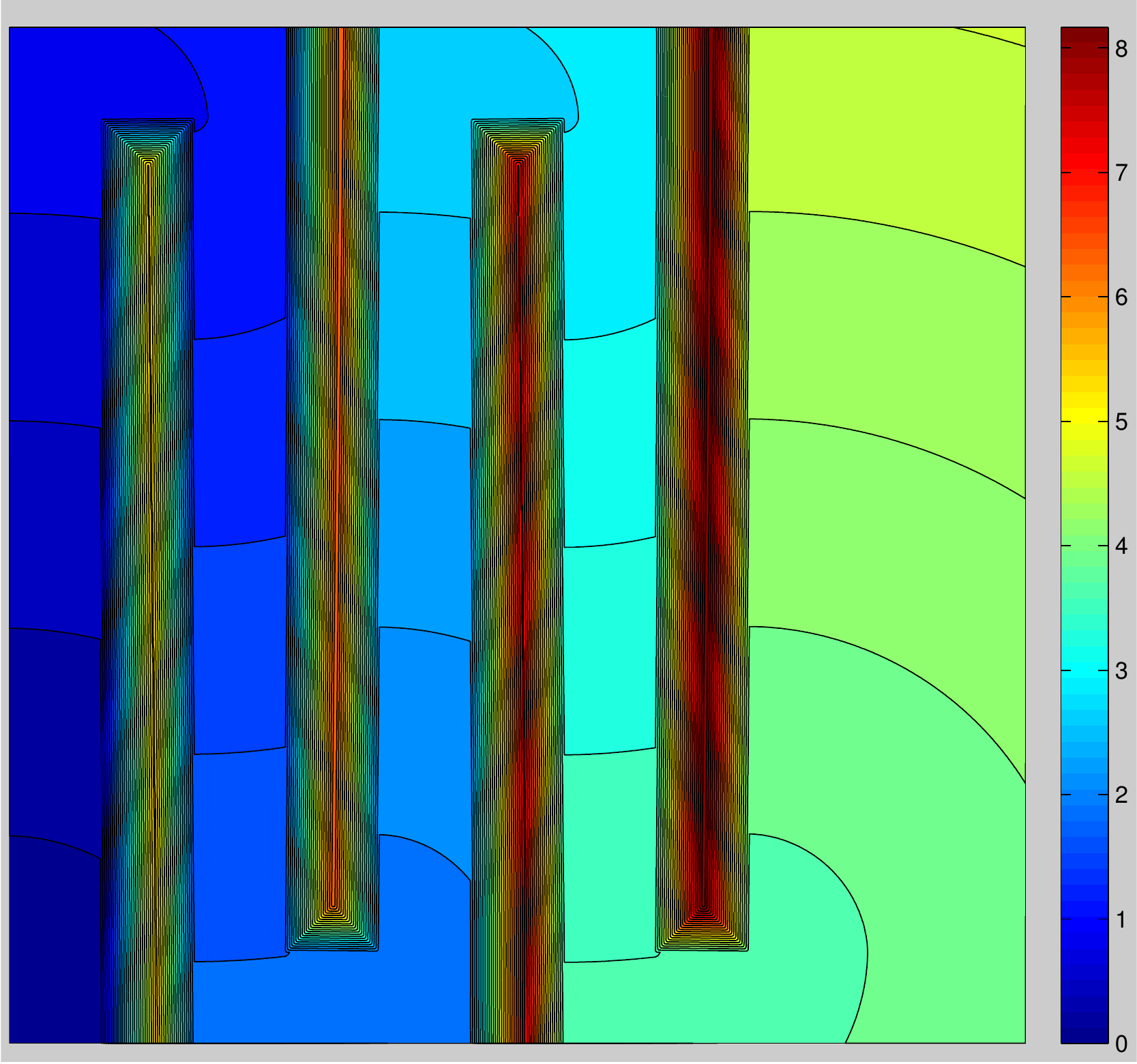}

&
\hspace*{10mm}
\includegraphics[scale = .5] {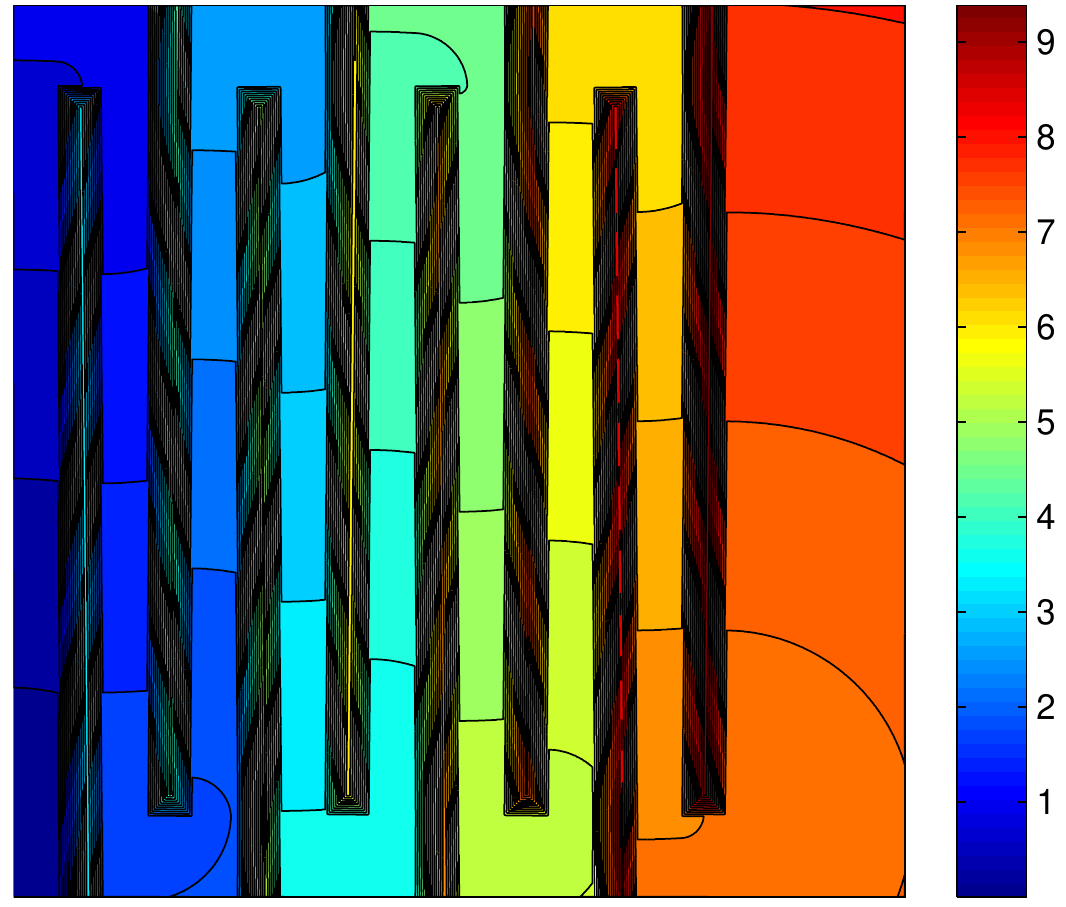}\\
A & B
\end{array}
$
}
\caption{
{\footnotesize
Min time to the point $(0,0)$ on comb maze domains:  4 barriers (A),
and 8 barriers (B).}}
\label{fig:comb_maze}
\end{figure}

\begin{table}[h]\footnotesize
\caption{Performance/convergence results for a 4 wall comb maze example.}
 \vspace*{2mm}
\begin{tabular}{|c|c|c|c|c|c|c|}
\hline
\textbf{Grid Size} & $\bm{L_\infty}$ \textbf{Error} & $\bm{L_1}$ \textbf{Error} & \textbf{FMM Time} & \textbf{FSM Time} & \textbf{LSM Time} & \textbf{\# Sweeps}\\
\hline

1408 $\times$ 1408  & 5.9449e-002  & 1.4210e-002  & 2.45   & 6.41  & 2.05  & 12 \\ 
\hline
\end{tabular}
  
\vspace*{.5cm}

\begin{tabular}{|l|c|c|c|c|c|c|c|}
\hline 
\textbf{METHOD} & \textbf{TIME} & $\bm{\mathcal{R}}$ & $\bm{\mathcal{\rho}}$ & \textbf{R} & \textbf{AvHR} & \textbf{AvS} & \textbf{Mon \%} \\
\hline
\vspace*{-.4cm}
&&&&&&&\\
\hline 	

HCM $22\times22$ cells &1.08&&&& 1.151& 3.971 & \\ \hline
HCM $44\times44$ cells &1.10&&&& 1.078& 3.724 & \\ \hline
HCM $88\times88$ cells &1.08&&&& 1.040& 3.593 & \\ \hline
HCM $176\times176$ cells &1.10&&&& 1.020& 3.518 & \\ \hline
HCM $352\times352$ cells &1.24&&&& 1.015& 3.496 & \\ \hline
HCM $704\times704$ cells &1.63&&&& 1.008& 3.468 & \\ \hline

\hline
\vspace*{-.4cm}
&&&&&&&\\
\hline

FHCM $22\times22$ cells &0.79& 1.0460& 1.0000& 1.0000 & 1.151& 1.618 & 85.5\\ \hline
FHCM $44\times44$ cells &0.74& 1.0191& 1.0000& 1.0000 & 1.078& 1.310 & 92.6\\ \hline
FHCM $88\times88$ cells &0.74& 1.0085& 1.0000& 1.0000 & 1.040& 1.156 & 96.2\\ \hline
FHCM $176\times176$ cells &0.78& 1.0073& 1.0000& 1.0000 & 1.020& 1.080 & 98.4\\ \hline
FHCM $352\times352$ cells &0.95& 1.0002& 1.0000& 1.0000 & 1.015& 1.049 & 99.3\\ \hline
FHCM $704\times704$ cells &1.41& 1.0000& 1.0000& 1.0000 & 1.008& 1.022 & 100.0\\ \hline

\hline
\vspace*{-.4cm}
&&&&&&&\\
\hline 	

FMSM $22\times22$ cells &0.58& 1.1659& 1.0000& 1.0000 & & 1.436 & \\ \hline
FMSM $44\times44$ cells &0.54& 1.0706& 1.0000& 1.0018 & & 1.218 & \\ \hline
FMSM $88\times88$ cells &0.53& 1.0821& 1.0000& 1.0018 & & 1.110 & \\ \hline
FMSM $176\times176$ cells &0.57& 1.0468& 1.0000& 1.0008 & & 1.055 & \\ \hline
FMSM $352\times352$ cells &0.71& 1.0378& 1.0000& 1.0004 & & 1.028 & \\ \hline
FMSM $704\times704$ cells &1.24& 1.0064& 1.0000& 1.0001 & & 1.014 & \\ \hline
\hline
\end{tabular}
\label{tab:4_comb}
\end{table}

We note that even the slowest of the HCM trials outperforms FMM, FSM, and LSM on this example.
Despite the special alignment of cell boundaries, 
this example is typical in the following ways:

1.  In both Heap-Cell algorithms, as the number of cells increases, the average number of heap removals per cell decreases.

2.  In FHCM the average number of sweeps per cell decreases to 1 as $h^c$ decreases.


3.  In FHCM the percentage of monotonicity check successes increases as $h^c$ decreases.

4.  For timing performance in both HCM and FHCM, 
the optimal choice of $h^c$ is somewhere in the middle of the tested range.

The reason for \#2 is that, as the number of cells $J$ increases, most cells will pass the Monotonicity Check.  When the 
monotonicity percentage is high and each cell has on average 2 ``upwinding'' neighboring cells,
each cell on the heap will have one sweeping direction tagged.  This observation combined with
\#1 explains \#2.

Combining \#1 and \#2 and the fact that the length of the heap also increases with $J$
there is a complexity trade-off that explains \#4.  
As $J$ tends to $M$, the complexity of both Heap-Cell 
algorithms is similar to that of Fast Marching.  As $J$ tends to $1$, the complexity of 
HCM is similar to that of Locking Sweeping.

In the second example we use 8 barriers
and the boundaries of the cells are {\bf not aligned} with 
the discontinuities of the speed function.  This example was chosen specifically because it is difficult for our 
new hybrid methods when using the same cell-decompositions as in the previous example.
The performance data is summarized in Table \ref{tab:8_comb}.
\begin{table}[h]\footnotesize
\caption{Performance/convergence results for an 8 wall comb maze example.}
 \vspace*{2mm}
\begin{tabular}{|c|c|c|c|c|c|c|}
\hline
\textbf{Grid Size} & $\bm{L_\infty}$ \textbf{Error} & $\bm{L_1}$ \textbf{Error} & \textbf{FMM Time} & \textbf{FSM Time} & \textbf{LSM Time} & \textbf{\# Sweeps}\\
\hline
1408 $\times$ 1408  & 6.5644e-002  & 1.6865e-002  & 2.50   & 11.1  & 3.20  & 20 \\ 
\hline
\end{tabular}

\vspace*{.5cm}

\begin{tabular}{|l|c|c|c|c|c|c|c|}
\hline 
\textbf{METHOD} & \textbf{TIME} & $\bm{\mathcal{R}}$ & $\bm{\mathcal{\rho}}$ & \textbf{R} & \textbf{AvHR} & \textbf{AvS} & \textbf{Mon \%} \\
\hline
\vspace*{-.4cm}
&&&&&&&\\
\hline 	
HCM $22\times22$ cells &2.13&&&& 2.795& 9.293 & \\ \hline
HCM $44\times44$ cells &7.68&&&& 8.738& 28.046 & \\ \hline
HCM $88\times88$ cells &6.68&&&& 6.798& 22.804 & \\ \hline
HCM $176\times176$ cells &5.86&&&& 5.655& 18.872 & \\ \hline
HCM $352\times352$ cells &2.95&&&& 2.456& 8.314 & \\ \hline
HCM $704\times704$ cells &1.74&&&& 1.037& 3.587 & \\ \hline

\hline
\vspace*{-.4cm}
&&&&&&&\\
\hline 	

FHCM $22\times22$ cells &1.75& 1.4247& 1.0000& 1.0000 & 2.946& 4.087 & 84.7\\ \hline
FHCM $44\times44$ cells &5.86& 1.4250& 1.0000& 1.0000 & 8.991& 10.209 & 94.0\\ \hline
FHCM $88\times88$ cells &4.54& 1.3083& 1.0000& 1.0000 & 6.976& 7.329 & 98.1\\ \hline
FHCM $176\times176$ cells &3.96& 1.2633& 1.0000& 1.0000 & 5.754& 5.910 & 99.1\\ \hline
FHCM $352\times352$ cells &2.13& 1.8922& 1.0000& 1.0000 & 2.468& 2.549 & 99.1\\ \hline
FHCM $704\times704$ cells &1.48& 1.5700& 1.0000& 1.0000 & 1.037& 1.066 & 100.0\\ \hline

\hline
\vspace*{-.4cm}
&&&&&&&\\
\hline
FMSM $22\times22$ cells &0.68& 604.49& 6.6555& 21.036 & & 1.783 & \\ \hline
FMSM $44\times44$ cells &0.59& 228.29& 3.1529& 19.442 & & 1.385 & \\ \hline
FMSM $88\times88$ cells &0.56& 313.01& 2.7666& 6.4608 & & 1.195 & \\ \hline
FMSM $176\times176$ cells &0.58& 381.98& 1.7374& 5.5944 & & 1.097 & \\ \hline
FMSM $352\times352$ cells &0.74& 45.397& 1.1718& 2.0506 & & 1.049 & \\ \hline
FMSM $704\times704$ cells &1.26& 23.303& 1.1738& 1.3536 & & 1.024 & \\ \hline

\hline
\end{tabular}
\label{tab:8_comb}
\end{table}

Notice that since the edges of cells do not coincide with the edges of barriers, the 
performance of the hybrid methods is not as good as in the previous 4-barrier case, where the edges do coincide.
In this example the cells that contain
a discontinuity of the speed function may not receive an accurate cell value (for either the Heap-Cell
algorithms or FMSM) and may often have poor
choices of planned sweeping directions (for FHCM \& FMSM).   For FHCM, since the error is small in most trials, 
this effect appears to be rectified at the expense of the same cells being added to the heap many times.
For FMSM, since each cell is processed 
only once, large error remains.  
The non-monotonic behavior of $\mathcal{R}$ in FMSM and FHCM 
appears to be due to changes in positions of cell centers
relative to barrier edges as $h^c$ decreases.

These comb maze examples illustrate the importance of choosing 
cell placement and cell sizes so that the speed is roughly constant in each cell.
This is necessary both for a small number of sweeps to be effective and for choosing cell values accurately.

\subsection{Checkerboards}
\label{ss:checkers}

\iffullversion
We return to the checkerboard example already described in section \ref{ss:FMSM}.  
For both the $11 \times 11$ and $41 \times 41$ checkerboard speed functions
the center checker is slow.  The speed is 1 in the slow checkers and 2 in the fast checkers.  
The exit set is the single point $Q= \{ (0.5, \, 0.5) \}$.
\fi

\begin{figure}[h]
\center{
$
\begin{array}{cc}
\includegraphics[scale = .3] {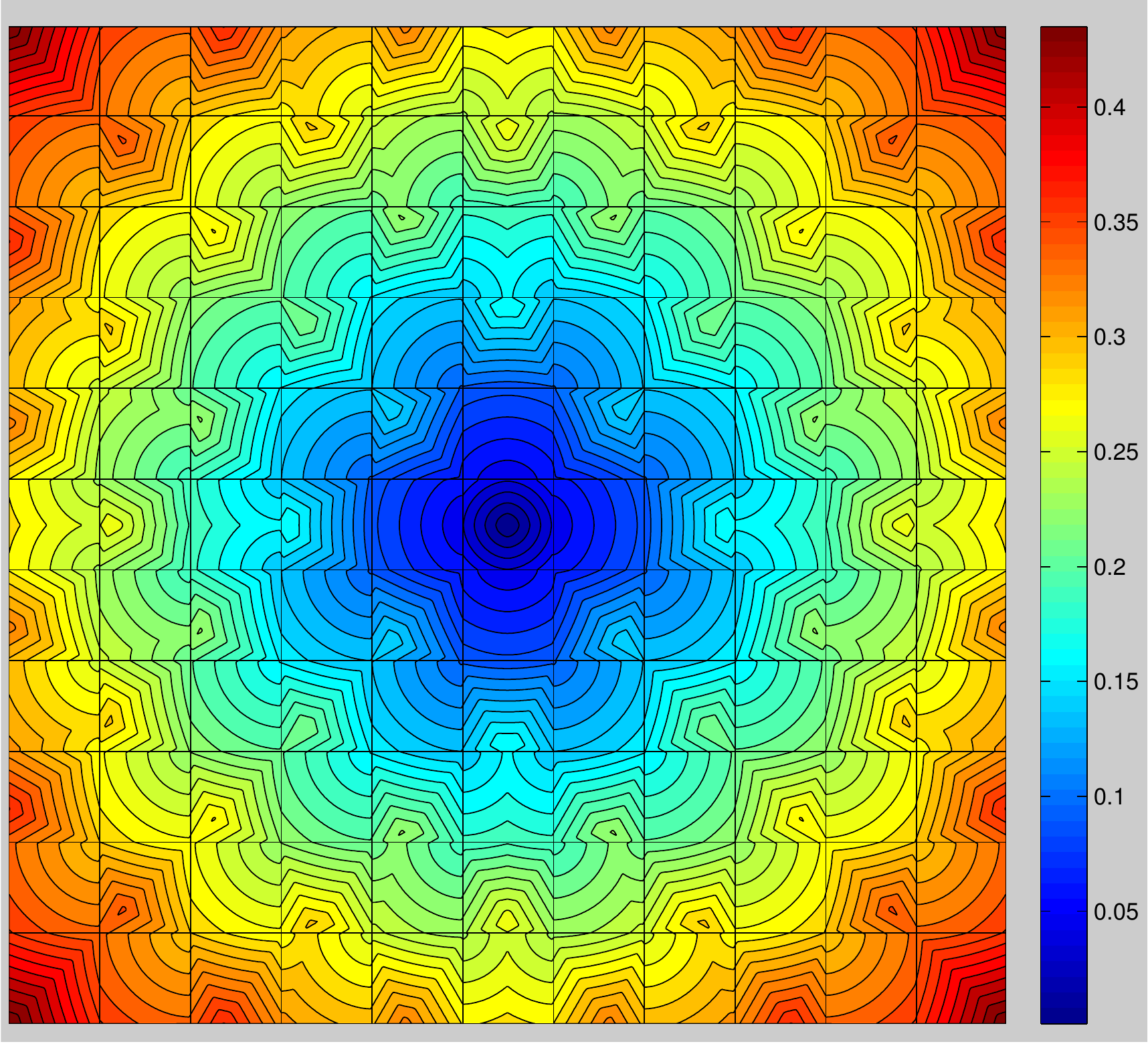}

&
\hspace*{5mm}
\includegraphics[scale = .3] {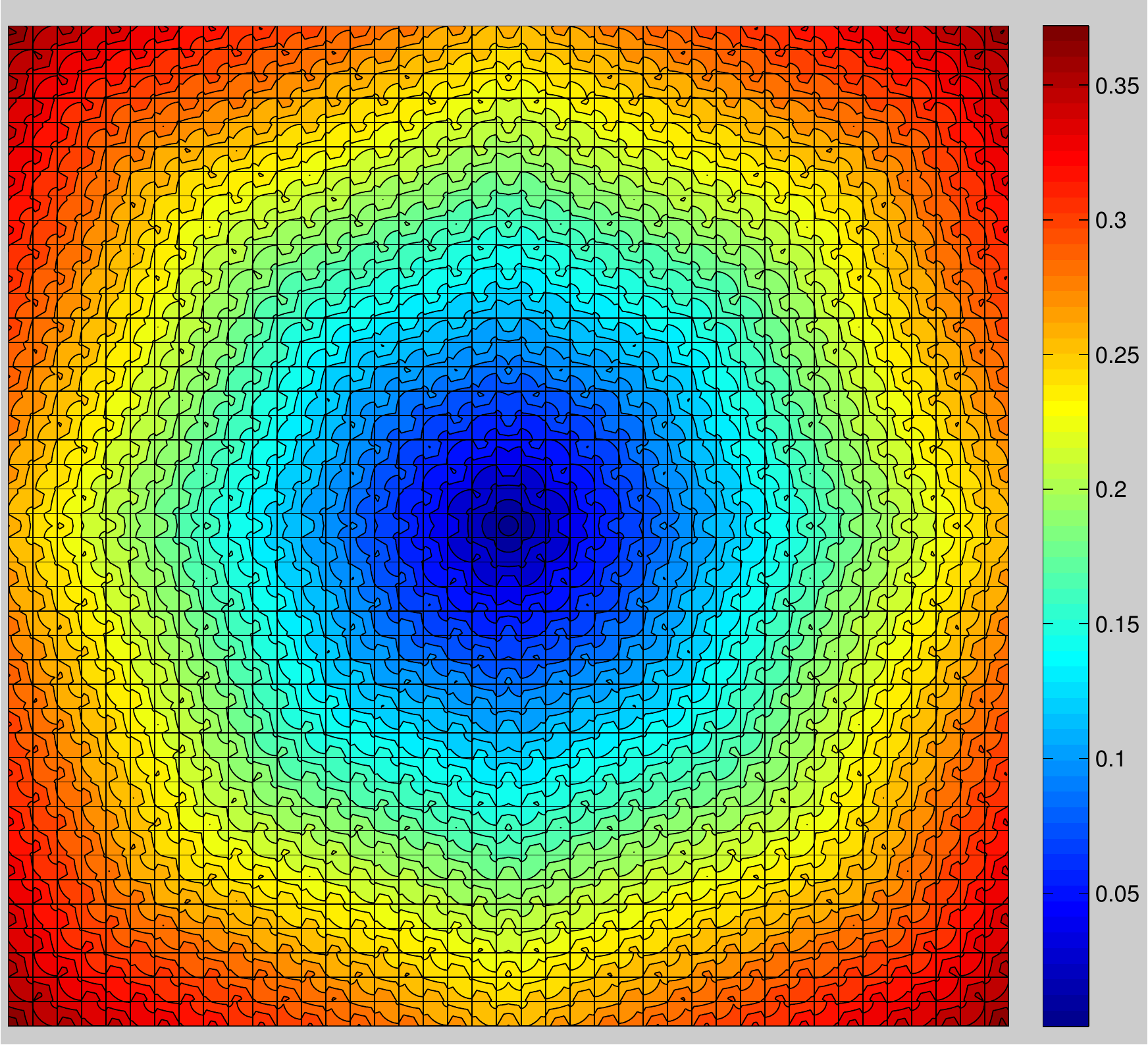}\\
A & B
\end{array}
$
}
\caption{
{\footnotesize
Min time to the center on checkerboard domains:  $11 \times 11$ checkers (A),
and $41 \times 41$ checkers (B).}}
\label{fig:checkers_computed}
\end{figure}

\begin{remark}
\label{rem:checkers_FSM}
Such checkerboard examples arise naturally in the context of front propagation through 
composite medium, consisting of a periodic mix of isotropic constituent materials
with different speed function $F$.  The idea of {\em homogenization} is to
derive a homogeneous but anisotropic speed function $\widebar{F}(\bn)$, describing 
the large-scale properties of the composite material.  After $\widebar{F}(\bn)$ is 
computed, the boundary value problems can be solved on a coarser grid.
A new efficient method for this homogenization was introduced in \cite{ObTaVlad}, 
using FMM on the fine scale grid since the characteristics are highly oscillatory
and the original implementation of sweeping was inefficient.
The same test problems were later attacked in \cite{LuoYuZhao} using 
a version of FSM with gridpoint locking (see Remark \ref{rem:locking_FSM}).  The results in 
Table \ref{tab:checker_41} shows that even the Locking-Sweeping Method becomes
significantly less efficient than FMM with the increase in the number of checkers.
\end{remark}

\vspace*{-1mm}
\begin{table}[H]\footnotesize
\caption{Performance/convergence results for $11 \times 11$ checkerboard example.}
 \vspace*{2mm}
\begin{tabular}{|c|c|c|c|c|c|c|}
\hline
\textbf{Grid Size} & $\bm{L_\infty}$ \textbf{Error} & $\bm{L_1}$ \textbf{Error} & \textbf{FMM Time} & \textbf{FSM Time} & \textbf{LSM Time} & \textbf{\# Sweeps}\\
\hline
1408 $\times$ 1408  & 3.2639e-003  & 1.7738e-003  & 3.44   & 12.3  & 2.28  & 16 \\ 
\hline
\end{tabular}

\vspace*{.5cm}

\begin{tabular}{|l|c|c|c|c|c|c|c|}
\hline 
\textbf{METHOD} & \textbf{TIME} & $\bm{\mathcal{R}}$ & $\bm{\mathcal{\rho}}$ & \textbf{R} & \textbf{AvHR} & \textbf{AvS} & \textbf{Mon \%} \\
\hline
\vspace*{-.4cm}
&&&&&&&\\
\hline 	
HCM $22\times22$ cells &1.84&&&& 1.397& 5.254 & \\ \hline
HCM $44\times44$ cells &1.73&&&& 1.209& 4.613 & \\ \hline
HCM $88\times88$ cells &1.69&&&& 1.083& 4.117 & \\ \hline
HCM $176\times176$ cells &1.72&&&& 1.029& 3.864 & \\ \hline
HCM $352\times352$ cells &1.87&&&& 1.009& 3.768 & \\ \hline
HCM $704\times704$ cells &2.51&&&& 1.003& 3.746 & \\ \hline

\hline
\vspace*{-.4cm}
&&&&&&&\\
\hline 	

FHCM $22\times22$ cells &1.17& 1.0122& 1.0000& 1.0000 & 1.399& 1.779 & 86.3\\ \hline
FHCM $44\times44$ cells &1.11& 1.0208& 1.0000& 1.0000 & 1.227& 1.535 & 90.6\\ \hline
FHCM $88\times88$ cells &1.08& 1.0111& 1.0000& 1.0000 & 1.091& 1.247 & 95.1\\ \hline
FHCM $176\times176$ cells &1.14& 1.0050& 1.0000& 1.0000 & 1.029& 1.103 & 97.8\\ \hline
FHCM $352\times352$ cells &1.33& 1.0006& 1.0000& 1.0000 & 1.009& 1.043 & 99.4\\ \hline
FHCM $704\times704$ cells &2.08& 1.0000& 1.0000& 1.0000 & 1.003& 1.020 & 100.0\\ \hline

\hline
\vspace*{-.4cm}
&&&&&&&\\
\hline
FMSM $22\times22$ cells &0.87& 40.312& 1.5725& 13.016 & & 1.269 & \\ \hline
FMSM $44\times44$ cells &0.91& 18.167& 1.0875& 7.4581 & & 1.334 & \\ \hline
FMSM $88\times88$ cells &0.89& 7.6692& 1.0113& 3.1400 & & 1.222 & \\ \hline
FMSM $176\times176$ cells &0.91& 5.4947& 1.0025& 2.4813 & & 1.127 & \\ \hline
FMSM $352\times352$ cells &1.07& 2.4557& 1.0004& 1.3888 & & 1.067 & \\ \hline
FMSM $704\times704$ cells &1.84& 1.5267& 1.0000& 1.0032 & & 1.035 & \\ \hline

\hline
\end{tabular}
\label{tab:checker_11}
\end{table}

\vspace*{-4mm}
\begin{table}[h]\footnotesize
\caption{Performance/convergence results for $41 \times 41$ checkerboard example.}
 \vspace*{2mm}
\begin{tabular}{|c|c|c|c|c|c|c|}
\hline
\textbf{Grid Size} & $\bm{L_\infty}$ \textbf{Error} & $\bm{L_1}$ \textbf{Error} & \textbf{FMM Time} & \textbf{FSM Time} & \textbf{LSM Time} & \textbf{\# Sweeps}\\
\hline
1312 $\times$ 1312  & 1.2452e-002  & 6.6827e-003  & 4.13   & 58.9  & 11.7  & 45 \\ 
\hline
\end{tabular}

\vspace*{.5cm}

\begin{tabular}{|l|c|c|c|c|c|c|c|}
\hline 
\textbf{METHOD} & \textbf{TIME} & $\bm{\mathcal{R}}$ & $\bm{\mathcal{\rho}}$ & \textbf{R} & \textbf{AvHR} & \textbf{AvS} & \textbf{Mon \%} \\

\hline
\vspace*{-.4cm}
&&&&&&&\\
\hline
HCM $41\times41$ cells &4.18&&&& 3.261& 11.926 & \\ \hline
HCM $82\times82$ cells &3.05&&&& 1.571& 5.939 & \\ \hline
HCM $164\times164$ cells &2.84&&&& 1.314& 4.831 & \\ \hline
HCM $328\times328$ cells &2.81&&&& 1.080& 3.972 & \\ \hline
HCM $656\times656$ cells &3.36&&&& 1.026& 3.768 & \\ \hline

\hline
\vspace*{-.4cm}
&&&&&&&\\
\hline 

FHCM $41\times41$ cells &2.83& 1.7506& 1.0041& 1.7123 & 3.261& 4.600 & 75.5\\ \hline
FHCM $82\times82$ cells &2.09& 1.0299& 1.0006& 1.0128 & 1.584& 2.147 & 78.8\\ \hline
FHCM $164\times164$ cells &1.95& 1.0103& 1.0001& 1.0000 & 1.321& 1.670 & 90.4\\ \hline
FHCM $328\times328$ cells &2.01& 1.0173& 1.0000& 1.0000 & 1.080& 1.236 & 96.9\\ \hline
FHCM $656\times656$ cells &2.79& 1.0075& 1.0000& 1.0000 & 1.026& 1.106 & 100.0\\ \hline
\hline
\vspace*{-.4cm}
&&&&&&&\\
\hline
FMSM $41\times41$ cells &1.46& 12.398& 3.4110& 3.3991 & & 1.164 & \\ \hline
FMSM $82\times82$ cells &1.54& 10.551& 1.0975& 1.7662 & & 1.211 & \\ \hline
FMSM $164\times164$ cells &1.70& 4.7036& 1.0142& 1.7123 & & 1.281 & \\ \hline
FMSM $328\times328$ cells &1.88& 2.0192& 1.0020& 1.7123 & & 1.242 & \\ \hline
FMSM $656\times656$ cells &2.65& 1.7506& 1.0004& 1.7123 & & 1.147 & \\ \hline
\hline
\end{tabular}
\label{tab:checker_41}
\end{table}

In both examples the cell sizes were chosen to align with the edges of the checkers 
(i.e., the discontinuities of the speed function).
On the $11 \times 11$ checkerboard, 
almost all of the HCM trials outperforms FMM and LSM, and most of 
the FHCM trials are more than twice as fast as LSM and three times faster than FMM
while the additional errors are negligible; see Table \ref{tab:checker_11}.

The $41 \times 41$ example is much more difficult for the sweeping algorithms 
because the number of times the characteristics changes direction
increases with the number of checkers.  
We note that the performance of FMM is only moderately worse here
(mostly due to a larger length of level curves and the resulting growth of the ``Considered List'').  
Again, almost all hybrid methods outperform all other methods.  The difference
is less striking than in the $11 \times 11$ example when compared with FMM, but FHCM and FMSM
are $4$ to $6$ times faster than LSM; see Table \ref{tab:checker_41}.

\iffullversion

\vfill

\pagebreak
\fi

\subsection{Continuous speed functions with a point source}
\label{ss:continuous_F_examples}

Suppose the speed function is $F \equiv 1$ and the exit set consists of a single point 
$Q = \{(0.5, \, 0.5)\}.$ 
In this case the viscosity solution is simply the distance to the center of the unit square.
We also note that the causal ordering of cells is clearly available here;
as a result, FHCM and FMSM do not introduce any additional errors. 
The performance data is summarized in Table \ref{tab:const}.
For constant speed functions LSM performs significantly better than FMM on  
fine meshes (such as this one).  The reason
why FMSM and FHCM are faster than LSM in some trials is that LSM checks all parts of the domain in each sweep,
including non-downwinding or already-computed parts.
Additionally LSM must perform a final sweep to check that all gridpoints 
are locked.  All of the hybrid algorithms slow down monotonically as $J$ 
 increases because of the cost of sorting the heap.

\begin{table}[h]\footnotesize
\caption{Performance/convergence results for constant speed function.}
 \vspace*{2mm}
\begin{tabular}{|c|c|c|c|c|c|c|}
\hline
\textbf{Grid Size} & $\bm{L_\infty}$ \textbf{Error} & $\bm{L_1}$ \textbf{Error} & \textbf{FMM Time} & \textbf{FSM Time} & \textbf{LSM Time} & \textbf{\# Sweeps}\\
\hline
1408 $\times$ 1408  & 1.0956e-003  & 6.8382e-004  & 2.72   & 2.07  & 0.83  & 5 \\
\hline
\end{tabular}

\vspace*{.5cm}

\begin{tabular}{|l|c|c|c|c|c|c|c|}
\hline
\textbf{METHOD} & \textbf{TIME} & $\bm{\mathcal{R}}$ & $\bm{\mathcal{\rho}}$ & \textbf{R} & \textbf{AvHR} & \textbf{AvS} & \textbf{Mon \%} \\
\hline
\vspace*{-.4cm}
&&&&&&&\\
\hline

HCM $22\times22$ cells &1.05&&&& 1.000& 3.692 & \\ \hline
HCM $44\times44$ cells &1.12&&&& 1.000& 3.718 & \\ \hline
HCM $88\times88$ cells &1.10&&&& 1.000& 3.733 & \\ \hline
HCM $176\times176$ cells &1.14&&&& 1.000& 3.742 & \\ \hline
HCM $352\times352$ cells &1.29&&&& 1.000& 3.746 & \\ \hline
HCM $704\times704$ cells &1.76&&&& 1.000& 3.748 & \\ \hline

\hline
\vspace*{-.4cm}
&&&&&&&\\
\hline 	

FHCM $22\times22$ cells &0.66& 1.0000& 1.0000& 1.0000 & 1.000& 1.025 & 100.0\\ \hline
FHCM $44\times44$ cells &0.67& 1.0000& 1.0000& 1.0000 & 1.000& 1.006 & 100.0\\ \hline
FHCM $88\times88$ cells &0.69& 1.0000& 1.0000& 1.0000 & 1.000& 1.002 & 100.0\\ \hline
FHCM $176\times176$ cells &0.75& 1.0000& 1.0000& 1.0000 & 1.000& 1.000 & 100.0\\ \hline
FHCM $352\times352$ cells &0.92& 1.0000& 1.0000& 1.0000 & 1.000& 1.000 & 100.0\\ \hline
FHCM $704\times704$ cells &1.47& 1.0000& 1.0000& 1.0000 & 1.000& 1.000 & 100.0\\ \hline

\hline
\vspace*{-.4cm}
&&&&&&&\\
\hline
FMSM $22\times22$ cells &0.47& 1.0000& 1.0000& 1.0000 & & 1.103 & \\ \hline
FMSM $44\times44$ cells &0.47& 1.0000& 1.0000& 1.0000 & & 1.049 & \\ \hline
FMSM $88\times88$ cells &0.49& 1.0000& 1.0000& 1.0000 & & 1.024 & \\ \hline
FMSM $176\times176$ cells &0.53& 1.0000& 1.0000& 1.0000 & & 1.012 & \\ \hline
FMSM $352\times352$ cells &0.67& 1.0000& 1.0000& 1.0000 & & 1.006 & \\ \hline
FMSM $704\times704$ cells &1.23& 1.0000& 1.0000& 1.0000 & & 1.003 & \\ \hline

\hline
\end{tabular}
\label{tab:const}
\end{table}

\iffullversion
\fi
Next we consider examples of 
min-time to the center under two different
oscillatory continuous speed functions.
For $F(x,y) = 1 + \frac{1}{2} \sin(20 \pi x) \sin(20 \pi y)$ 
the level sets of the value function are shown in Figure \ref{fig:sinusoidal}A
and the performance data is summarized in Table \ref{tab:sinu_1}.
For $F(x,y) = 1 + 0.99 \sin(2 \pi x) \sin(2 \pi y)$
the level sets of the value function are shown in Figure \ref{fig:sinusoidal}B
and the performance data is summarized in Table \ref{tab:sinu_2}.

\iffullversion
\fi
\begin{figure}[h]
\center{
$
\begin{array}{cc}
\includegraphics[scale = .3] {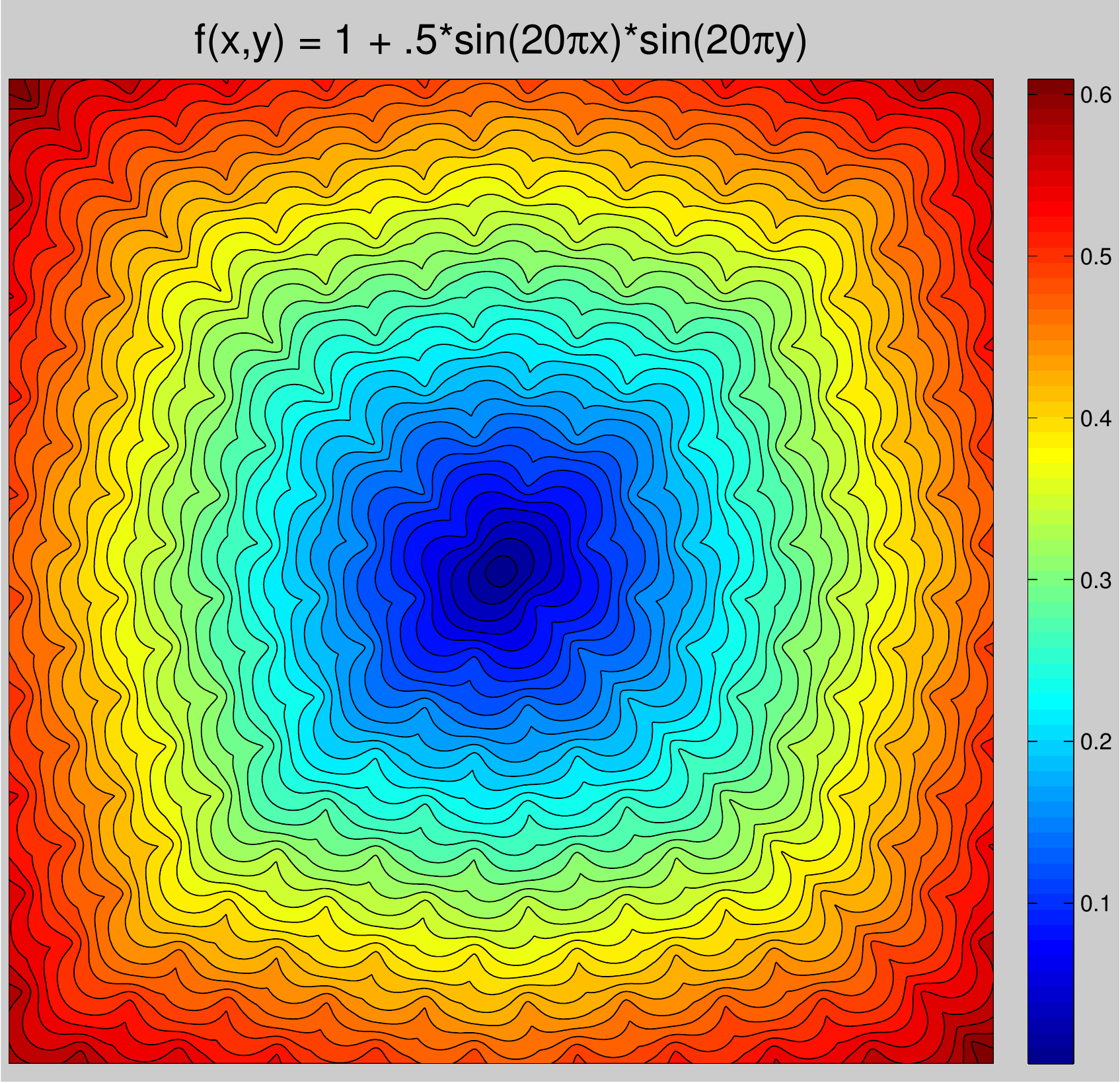}

&
\hspace*{5mm}
\includegraphics[scale = .3] {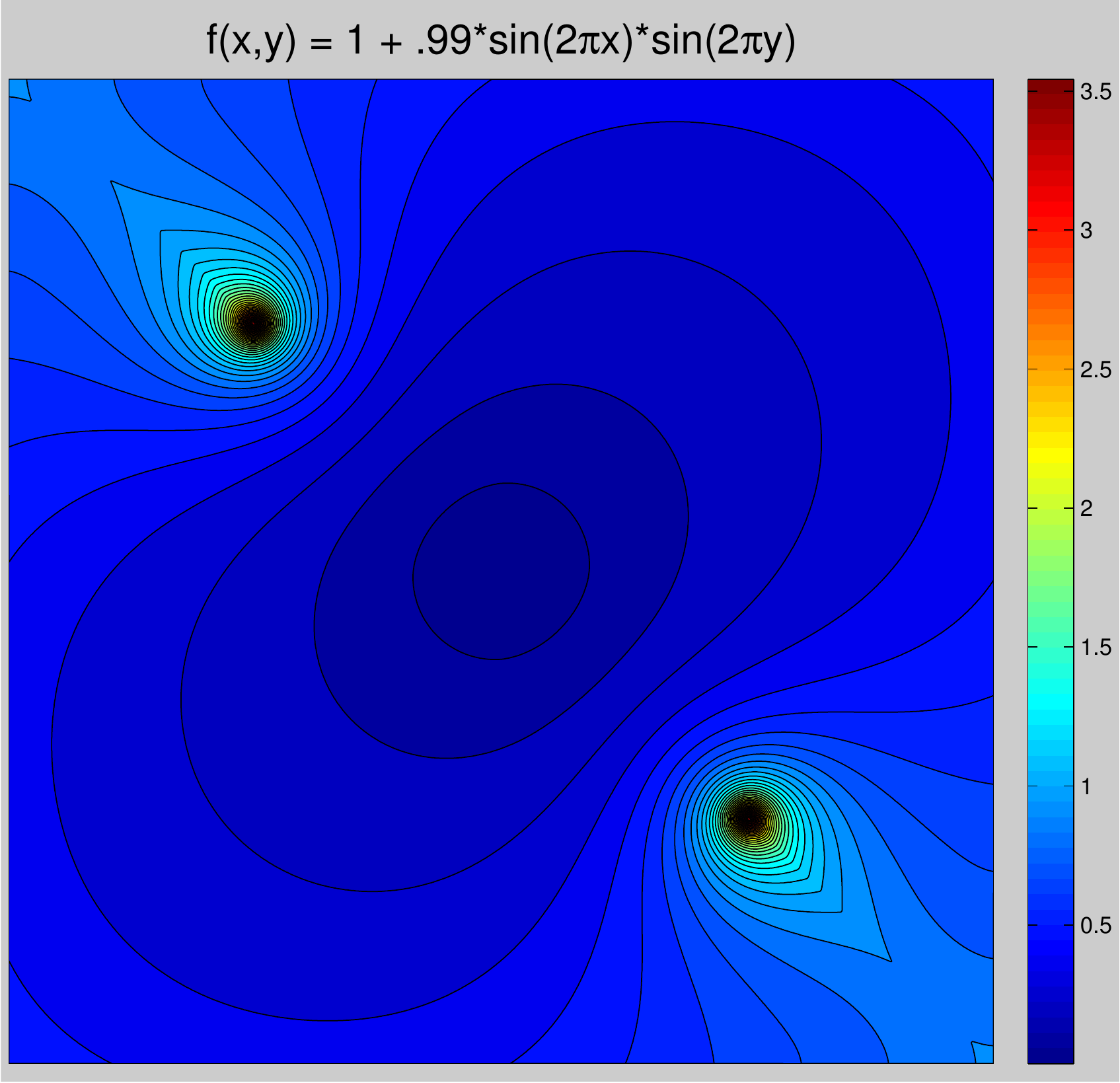}\\
A & B
\end{array}
$
}
\caption{
{\footnotesize
Min time to the center under sinusoidal speed functions.}}
\label{fig:sinusoidal}
\end{figure}

\iffullversion
\fi
\begin{table}[h]\footnotesize
\caption{Performance/convergence results for $F(x,y) = 1 + \frac{1}{2} \sin(20 \pi x) \sin(20 \pi y) $.}
 \vspace*{2mm}
\begin{tabular}{|c|c|c|c|c|c|c|}
\hline
\textbf{Grid Size} & $\bm{L_\infty}$ \textbf{Error} & $\bm{L_1}$ \textbf{Error} & \textbf{FMM Time} & \textbf{FSM Time} & \textbf{LSM Time} & \textbf{\# Sweeps}\\
\hline
1408 $\times$ 1408  & 4.7569e-003  & 1.9724e-003  & 3.74   & 23.7  & 6.39  & 24 \\ 
\hline
\end{tabular}

\vspace*{.5cm}

\begin{tabular}{|l|c|c|c|c|c|c|c|}
\hline 
\textbf{METHOD} & \textbf{TIME} & $\bm{\mathcal{R}}$ & $\bm{\mathcal{\rho}}$ & \textbf{R} & \textbf{AvHR} & \textbf{AvS} & \textbf{Mon \%} \\
\hline
\vspace*{-.4cm}
&&&&&&&\\
\hline 	

HCM $22\times22$ cells &3.61&&&& 1.913& 10.785 & \\ \hline
HCM $44\times44$ cells &2.97&&&& 1.446& 6.811 & \\ \hline
HCM $88\times88$ cells &2.60&&&& 1.245& 5.201 & \\ \hline
HCM $176\times176$ cells &2.40&&&& 1.117& 4.350 & \\ \hline
HCM $352\times352$ cells &2.40&&&& 1.047& 3.945 & \\ \hline
HCM $704\times704$ cells &2.92&&&& 1.016& 3.788 & \\ \hline

\hline
\vspace*{-.4cm}
&&&&&&&\\
\hline

FHCM $22\times22$ cells &2.72& 5.6062& 1.1358& 2.0960 & 4.413& 5.310 & 67.3\\ \hline
FHCM $44\times44$ cells &1.82& 3.1094& 1.1480& 1.0000 & 1.555& 2.132 & 78.7\\ \hline
FHCM $88\times88$ cells &1.61& 1.4025& 1.0122& 1.0000 & 1.277& 1.575 & 88.2\\ \hline
FHCM $176\times176$ cells &1.53& 1.0560& 1.0022& 1.0000 & 1.125& 1.262 & 94.5\\ \hline
FHCM $352\times352$ cells &1.65& 1.0226& 1.0004& 1.0000 & 1.048& 1.106 & 98.1\\ \hline
FHCM $704\times704$ cells &2.40& 1.0037& 1.0001& 1.0000 & 1.016& 1.035 & 100.0\\ \hline

\hline
\vspace*{-.4cm}
&&&&&&&\\
\hline

FMSM $22\times22$ cells &1.14& 10.497& 2.4811& 2.9653 & & 1.262 & \\ \hline
FMSM $44\times44$ cells &1.10& 6.0892& 1.3657& 2.2889 & & 1.200 & \\ \hline
FMSM $88\times88$ cells &1.16& 4.6801& 1.0515& 1.9504 & & 1.213 & \\ \hline
FMSM $176\times176$ cells &1.18& 3.4828& 1.0074& 1.3705 & & 1.126 & \\ \hline
FMSM $352\times352$ cells &1.34& 1.5987& 1.0007& 1.0000 & & 1.067 & \\ \hline
FMSM $704\times704$ cells &2.14& 1.1262& 1.0001& 1.0000 & & 1.035 & \\ \hline
\hline
\end{tabular}
\label{tab:sinu_1}
\end{table}

\begin{table}[h]\footnotesize
\caption{Performance/convergence results for $F(x,y) = 1 + 0.99 \sin(2\pi x) \sin(2\pi y) $ .}
 \vspace*{2mm}
\begin{tabular}{|c|c|c|c|c|c|c|}
\hline
\textbf{Grid Size} & $\bm{L_\infty}$ \textbf{Error} & $\bm{L_1}$ \textbf{Error} & \textbf{FMM Time} & \textbf{FSM Time} & \textbf{LSM Time} & \textbf{\# Sweeps}\\
\hline
1408 $\times$ 1408  & 2.1793e-002  & 9.8506e-004  & 3.69   & 12.7  & 2.73  & 13 \\ 
\hline
\end{tabular}

\vspace*{.5cm}

\begin{tabular}{|l|c|c|c|c|c|c|c|}
\hline 
\textbf{METHOD} & \textbf{TIME} & $\bm{\mathcal{R}}$ & $\bm{\mathcal{\rho}}$ & \textbf{R} & \textbf{AvHR} & \textbf{AvS} & \textbf{Mon \%} \\
\hline
\vspace*{-.4cm}
&&&&&&&\\
\hline
HCM $22\times22$ cells &2.29&&&& 1.165& 4.651 & \\ \hline
HCM $44\times44$ cells &2.15&&&& 1.070& 4.132 & \\ \hline
HCM $88\times88$ cells &2.11&&&& 1.034& 3.920 & \\ \hline
HCM $176\times176$ cells &2.13&&&& 1.015& 3.811 & \\ \hline
HCM $352\times352$ cells &2.26&&&& 1.008& 3.763 & \\ \hline
HCM $704\times704$ cells &2.80&&&& 1.002& 3.741 & \\ \hline

\hline
\vspace*{-.4cm}
&&&&&&&\\
\hline 
FHCM $22\times22$ cells &1.37& 60.848& 1.0020& 1.0014 & 1.174& 1.409 & 92.7\\ \hline
FHCM $44\times44$ cells &1.28& 4.5786& 1.0002& 1.0001 & 1.078& 1.185 & 96.1\\ \hline
FHCM $88\times88$ cells &1.28& 1.0224& 1.0000& 1.0000 & 1.039& 1.086 & 98.2\\ \hline
FHCM $176\times176$ cells &1.35& 1.0019& 1.0000& 1.0000 & 1.017& 1.039 & 99.3\\ \hline
FHCM $352\times352$ cells &1.55& 1.0003& 1.0000& 1.0000 & 1.008& 1.018 & 99.7\\ \hline
FHCM $704\times704$ cells &2.27& 1.0001& 1.0000& 1.0000 & 1.002& 1.006 & 100.0\\ \hline

\hline
\vspace*{-.4cm}
&&&&&&&\\
\hline 
FMSM $22\times22$ cells &1.13& 1362.4& 1.0270& 1.0053 & & 1.231 & \\ \hline
FMSM $44\times44$ cells &1.06& 174.62& 1.0054& 1.0053 & & 1.116 & \\ \hline
FMSM $88\times88$ cells &1.05& 38.545& 1.0021& 1.0046 & & 1.057 & \\ \hline
FMSM $176\times176$ cells &1.09& 7.1581& 1.0006& 1.0046 & & 1.029 & \\ \hline
FMSM $352\times352$ cells &1.28& 1.1687& 1.0001& 1.0028 & & 1.014 & \\ \hline
FMSM $704\times704$ cells &2.08& 1.0724& 1.0000& 1.0000 & & 1.007 & \\ \hline
\hline
\end{tabular}
\label{tab:sinu_2}
\end{table}

Note that HCM outperforms Fast Marching on all trials, and outperforms the sweeping methods 
significantly on the first example (Table \ref{tab:sinu_1}) despite the fact that no special selection of cell boundaries 
was made.  
Small changes in the frequency of the speed function did not significantly alter the performance of the hybrid algorithms.  
In the second example (Table \ref{tab:sinu_2}) most HCM trials were again faster than LSM and
FMM.  Note that for some cell sizes, both FMSM and FHCM have $R \ll \mathcal{R} = \max_j (E_j/e_j)$.
Whenever $R$ is close to 1, the rate of convergence of hybrid methods (based on $L_{\infty}$ errors)
is the same as that of FMM and FSM.  
%

\iffullversion

\subsection{Performance on coarser grids}
\label{ss:coarser}

Our hybrid methods exploit the fact that there exists $h^c$ small enough so that
most cell-boundaries will be either fully inflow or fully outflow and most pairs of cells
will not be mutually dependent. But if the original grid $X$ is sufficiently coarse,
this may not be possible to achieve since we also need $h^c \geq 2 h$ 
(otherwise FMM is clearly more efficient).
In this subsection we return to some of the previous examples but on significantly coarser grids, 
to test whether the hybrid methods remain competitive with FMM and LSM. 
The performance data is summarized in Tables \ref{tab:checker_11_c1}-\ref{tab:sinu_2_c2};
to improve the accuracy of timing on coarser grids, all CPU times are reported for 
20 executions of each algorithm.

Since $M$ is much smaller 
here, the $\log M$ term
in the complexity of Fast Marching plays less of a role.  
On most of the examples in this subsection, HCM and FHCM are not much faster than Fast Marching or Locking Sweeping.
For example, in Table \ref{tab:checker_41_c1} even though the cell boundaries are perfectly aligned with 
the checker boundaries, 
both Heap-Cell methods are merely on par with Fast Marching. 
Note that when $h$ is sufficiently small, their advantage over FMM and LSM
is clear (see Table \ref{tab:checker_41}).
  FMSM, however, is about twice as fast as the faster
of FMM and LSM.  In addition, FMSM's error ratios 
($R$, $\mathcal{R}$, and $\mathcal{\rho}$) 
are smaller here than for the same examples on finer grids 
in subsections \ref{ss:checkers}-\ref{ss:continuous_F_examples}.

\vspace*{10mm}
\begin{remark}
\label{rm:coarse_fmm_and_fewer_sweeps}
Since two of the hybrid methods introduce additional errors, 
an important question is, ``Given the total errors resulting from FHCM 
and FMSM at a given resolution ($h$, $h^c$), for which $\bar{h} > h$ would
FMM commit similar errors, and how well would FMM perform on that new coarser grid?''
For simplicity, assume in the following discussion that the CPU time required 
by FMM is roughly linear in $M = O(h^{-2})$ and that the resulting $L_\infty$ error is $O(h)$.
These are reasonable assumptions for coarse grids; e.g., 
see Tables \ref{tab:checker_11_c1}-\ref{tab:sinu_2_c2}.
For example, if we want to decrease time by a factor of $p^2$,  then $M \to M/p^2$,  
$ h \to p*h$ (in 2-d), and errors would increase by a factor of $p$.
Such estimates allow for a more accurate performance comparison between FMM and FMSM (or FHCM)
based on the ratio $\textbf{R}$.  
Dividing the reported FMM time by the value $\textbf{R}^2$, we will arrive at an estimate 
for the new FMM time computed on a coarser $\bar{h}$-grid with errors similar 
to those committed by FMSM on an($h$,$h^c$)-grid.

Among Tables \ref{tab:checker_11_c1}-\ref{tab:sinu_2_c2}, 
the overall worst-case scenario for FMSM under this analysis 
is the $11 \times 11$ checkerboard example.  Using the data in Table 
\ref{tab:checker_11_c1} with $M = 176^2$ and comparing FMM with FMSM at $22^2$, 
$44^2$, and $88^2$ cells, the new estimated FMM times
would be $0.82/(2.392^2) = .343$, $0.82/(1.3489^2) = .608$, and $0.82/(1.004^2) = .817$.
Comparing this to $0.29$, $0.35$, $0.53$ reported for FMSM, we see that
 each of the cell trials still outperforms the corresponding improved time of FMM.
Similar conclusions are reached when this analysis is performed using error ratios in $L_1$ norms. 
\end{remark}

\pagebreak

\begin{table}[H]\footnotesize
\caption{Performance/convergence results for 20 trials of $11 \times 11$ checkerboard example on a coarse grid.}
 \vspace*{2mm}
\begin{tabular}{|c|c|c|c|c|c|c|}
\hline
\textbf{Grid Size} & $\bm{L_\infty}$ \textbf{Error} & $\bm{L_1}$ \textbf{Error} & \textbf{FMM Time}  & \textbf{FSM Time} & \textbf{LSM Time} & \textbf{\# Sweeps}\\
\hline
176 $\times$ 176  & 2.0986e-002  & 1.1087e-002  & 0.82   & 3.91  & 0.81  & 16 \\ 
\hline
\end{tabular}

\vspace*{.2cm}

\begin{tabular}{|l|c|c|c|c|c|c|c|}
\hline 
\textbf{METHOD} & \textbf{TIME} & $\bm{\mathcal{R}}$ & $\bm{\mathcal{\rho}}$ & \textbf{R} & \textbf{AvHR} & \textbf{AvS} & \textbf{Mon \%} \\
\hline
\vspace*{-.4cm}
&&&&&&&\\
\hline 	
HCM $22\times22$ cells &0.59&&&& 1.438& 5.134 & \\ \hline
HCM $44\times44$ cells &0.59&&&& 1.171& 4.199 & \\ \hline
HCM $88\times88$ cells &0.72&&&& 1.041& 3.779 & \\ \hline

\hline
\vspace*{-.4cm}
&&&&&&&\\
\hline 	

FHCM $22\times22$ cells &0.41& 1.0017& 1.0000& 1.0000 & 1.440& 1.804 & 88.2\\ \hline
FHCM $44\times44$ cells &0.43& 1.0015& 1.0000& 1.0000 & 1.171& 1.374 & 97.0\\ \hline
FHCM $88\times88$ cells &0.59& 1.0000& 1.0000& 1.0000 & 1.041& 1.158 & 100.0\\ \hline

\hline
\vspace*{-.4cm}
&&&&&&&\\
\hline
FMSM $22\times22$ cells &0.29& 5.1670& 1.0770& 2.3920 & & 1.269 & \\ \hline
FMSM $44\times44$ cells &0.35& 2.2742& 1.0066& 1.3489 & & 1.334 & \\ \hline
FMSM $88\times88$ cells &0.53& 1.2309& 1.0004& 1.0040 & & 1.221 & \\ \hline

\hline
\end{tabular}

\vspace*{6mm}
\begin{tabular}{|c|c|c|c|c|c|c|}
\hline
\textbf{Grid Size} & $\bm{L_\infty}$ \textbf{Error} & $\bm{L_1}$ \textbf{Error} & \textbf{FMM Time}& \textbf{FSM Time}  & \textbf{LSM Time} & \textbf{\# Sweeps}\\
\hline
352 $\times$ 352  & 1.1470e-002  & 6.0787e-003  & 3.52   & 15.4  & 3.16  & 16 \\ 
\hline
\end{tabular}

\vspace*{.2cm}

\begin{tabular}{|l|c|c|c|c|c|c|c|}
\hline 
\textbf{METHOD} & \textbf{TIME} & $\bm{\mathcal{R}}$ & $\bm{\mathcal{\rho}}$ & \textbf{R} & \textbf{AvHR} & \textbf{AvS} & \textbf{Mon \%} \\
\hline
\vspace*{-.4cm}
&&&&&&&\\
\hline 	
HCM $22\times22$ cells &2.40&&&& 1.438& 5.302 & \\ \hline
HCM $44\times44$ cells &2.25&&&& 1.208& 4.465 & \\ \hline
HCM $88\times88$ cells &2.32&&&& 1.059& 3.904 & \\ \hline
HCM $176\times176$ cells &2.91&&&& 1.018& 3.757 & \\ \hline

\hline
\vspace*{-.4cm}
&&&&&&&\\
\hline 	

FHCM $22\times22$ cells &1.61& 1.1194& 1.0002& 1.0725 & 1.490& 1.936 & 84.9\\ \hline
FHCM $44\times44$ cells &1.53& 1.0434& 1.0000& 1.0000 & 1.228& 1.508 & 92.2\\ \hline
FHCM $88\times88$ cells &1.69& 1.0745& 1.0000& 1.0000 & 1.059& 1.190 & 97.5\\ \hline
FHCM $176\times176$ cells &2.40& 1.0273& 1.0000& 1.0000 & 1.018& 1.086 & 100.0\\ \hline

\hline
\vspace*{-.4cm}
&&&&&&&\\
\hline
FMSM $22\times22$ cells &1.12& 10.551& 1.1593& 4.0315 & & 1.269 & \\ \hline
FMSM $44\times44$ cells &1.21& 4.7036& 1.0252& 3.9089 & & 1.334 & \\ \hline
FMSM $88\times88$ cells &1.38& 4.1945& 1.0093& 3.9089 & & 1.222 & \\ \hline
FMSM $176\times176$ cells &2.12& 4.1945& 1.0074& 3.9089 & & 1.127 & \\ \hline

\hline
\end{tabular}
\label{tab:checker_11_c1}
\end{table}

\begin{table}[H]\footnotesize
\caption{Performance/convergence results for 20 trials of $41 \times 41$ checkerboard on a coarse grid.}
 \vspace*{2mm}
\begin{tabular}{|c|c|c|c|c|c|c|}
\hline
\textbf{Grid Size} & $\bm{L_\infty}$ \textbf{Error} & $\bm{L_1}$ \textbf{Error} & \textbf{FMM Time}& \textbf{FSM Time} & \textbf{LSM Time} & \textbf{\# Sweeps}\\
\hline
164 $\times$ 164  & 7.1112e-002  & 3.8397e-002  & 1.08   & 17.9  & 4.01  & 44 \\
\hline
\end{tabular}

\vspace*{.2cm}

\begin{tabular}{|l|c|c|c|c|c|c|c|}
\hline 
\textbf{METHOD} & \textbf{TIME} & $\bm{\mathcal{R}}$ & $\bm{\mathcal{\rho}}$ & \textbf{R} & \textbf{AvHR} & \textbf{AvS} & \textbf{Mon \%} \\

\hline
\vspace*{-.4cm}
&&&&&&&\\
\hline
HCM $41\times41$ cells &1.13&&&& 2.204& 7.041 & \\ \hline
HCM $82\times82$ cells &1.05&&&& 1.261& 4.215 & \\ \hline

\hline
\vspace*{-.4cm}
&&&&&&&\\
\hline 

FHCM $41\times41$ cells &0.85& 1.0000& 1.0000& 1.0000 & 2.204& 2.449 & 92.2\\ \hline
FHCM $82\times82$ cells &0.90& 1.0000& 1.0000& 1.0000 & 1.261& 1.474 & 100.0\\ \hline

\hline
\vspace*{-.4cm}
&&&&&&&\\
\hline
FMSM $41\times41$ cells &0.53& 1.4878& 1.0850& 1.0197 & & 1.163 & \\ \hline
FMSM $82\times82$ cells &0.77& 1.1277& 1.0162& 1.0193 & & 1.210 & \\ \hline

\hline
\end{tabular}

 \vspace*{6mm}
\begin{tabular}{|c|c|c|c|c|c|c|}
\hline
\textbf{Grid Size} & $\bm{L_\infty}$ \textbf{Error} & $\bm{L_1}$ \textbf{Error} & \textbf{FMM Time}& \textbf{FSM Time} & \textbf{LSM Time} & \textbf{\# Sweeps}\\
\hline
328 $\times$ 328  & 4.0403e-002  & 2.3205e-002  & 4.44   & 73.3  & 16.6  & 45 \\ 
\hline
\end{tabular}

\vspace*{.2cm}

\begin{tabular}{|l|c|c|c|c|c|c|c|}
\hline 
\textbf{METHOD} & \textbf{TIME} & $\bm{\mathcal{R}}$ & $\bm{\mathcal{\rho}}$ & \textbf{R} & \textbf{AvHR} & \textbf{AvS} & \textbf{Mon \%} \\

\hline
\vspace*{-.4cm}
&&&&&&&\\
\hline

HCM $41\times41$ cells &5.42&&&& 2.873& 9.970 & \\ \hline
HCM $82\times82$ cells &4.02&&&& 1.500& 5.104 & \\ \hline
HCM $164\times164$ cells &4.19&&&& 1.181& 4.105 & \\ \hline

\hline
\vspace*{-.4cm}
&&&&&&&\\
\hline 

FHCM $41\times41$ cells &3.65& 1.0988& 1.0008& 1.0679 & 2.873& 3.802 & 81.6\\ \hline
FHCM $82\times82$ cells &2.90& 1.0236& 1.0000& 1.0000 & 1.501& 1.923 & 88.0\\ \hline
FHCM $164\times164$ cells &3.55& 1.0000& 1.0000& 1.0000 & 1.181& 1.384 & 100.0\\ \hline
	
\hline
\vspace*{-.4cm}
&&&&&&&\\
\hline
FMSM $41\times41$ cells &1.88& 2.9459& 1.4364& 1.4668 & & 1.164 & \\ \hline
FMSM $82\times82$ cells &2.22& 2.3040& 1.0533& 1.1457 & & 1.211 & \\ \hline
FMSM $164\times164$ cells &3.27& 1.1540& 1.0009& 1.0679 & & 1.281 & \\ \hline
\hline
\end{tabular}
\label{tab:checker_41_c1}
\end{table}

\vspace*{-10mm}
\begin{table}[H]\footnotesize
\caption{Performance/convergence results for 20 trials of $F(x,y) = 1 + \frac{1}{2} \sin(20 \pi x) \sin(20 \pi y) $ on a coarse grid.}
\begin{tabular}{|c|c|c|c|c|c|c|}
\hline
\textbf{Grid Size} & $\bm{L_\infty}$ \textbf{Error} & $\bm{L_1}$ \textbf{Error} & \textbf{FMM Time}& \textbf{FSM Time} & \textbf{LSM Time} & \textbf{\# Sweeps}\\
\hline
176 $\times$ 176  & 3.6535e-002  & 1.3374e-002  & 0.94   & 8.77  & 3.08  & 28 \\
\hline
\end{tabular}

\vspace*{1mm}

\begin{tabular}{|l|c|c|c|c|c|c|c|}
\hline 
\textbf{METHOD} & \textbf{TIME} & $\bm{\mathcal{R}}$ & $\bm{\mathcal{\rho}}$ & \textbf{R} & \textbf{AvHR} & \textbf{AvS} & \textbf{Mon \%} \\
\hline
\vspace*{-.4cm}
&&&&&&&\\
\hline 	

HCM $22\times22$ cells &0.97&&&& 1.773& 8.233 & \\ \hline
HCM $44\times44$ cells &0.88&&&& 1.280& 4.992 & \\ \hline
HCM $88\times88$ cells &0.87&&&& 1.100& 3.975 & \\ \hline

\hline
\vspace*{-.4cm}
&&&&&&&\\
\hline

FHCM $22\times22$ cells &0.66& 1.3736& 1.0209& 1.0000 & 2.153& 2.814 & 69.3\\ \hline
FHCM $44\times44$ cells &0.60& 1.1703& 1.0186& 1.0000 & 1.285& 1.684 & 87.7\\ \hline
FHCM $88\times88$ cells &0.71& 1.1170& 1.0072& 1.0000 & 1.100& 1.234 & 100.0\\ \hline

\hline
\vspace*{-.4cm}
&&&&&&&\\
\hline

FMSM $22\times22$ cells &0.38& 7.0809& 1.2945& 1.0359 & & 1.244 & \\ \hline
FMSM $44\times44$ cells &0.42& 2.2023& 1.0402& 1.0100 & & 1.197 & \\ \hline
FMSM $88\times88$ cells &0.64& 1.0945& 1.0024& 1.0000 & & 1.213 & \\ \hline
\hline
\end{tabular}

 \vspace*{2mm}
\begin{tabular}{|c|c|c|c|c|c|c|}
\hline
\textbf{Grid Size} & $\bm{L_\infty}$ \textbf{Error} & $\bm{L_1}$ \textbf{Error} & \textbf{FMM Time}& \textbf{FSM Time} & \textbf{LSM Time} & \textbf{\# Sweeps}\\
\hline
352 $\times$ 352  & 1.8414e-002  & 7.0584e-003  & 3.92   & 33.7  & 11.1  & 27 \\ 
\hline
\end{tabular}

\vspace*{1mm}

\begin{tabular}{|l|c|c|c|c|c|c|c|}
\hline 
\textbf{METHOD} & \textbf{TIME} & $\bm{\mathcal{R}}$ & $\bm{\mathcal{\rho}}$ & \textbf{R} & \textbf{AvHR} & \textbf{AvS} & \textbf{Mon \%} \\
\hline
\vspace*{-.4cm}
&&&&&&&\\
\hline 	

HCM $22\times22$ cells &4.43&&&& 1.909& 9.864 & \\ \hline
HCM $44\times44$ cells &3.57&&&& 1.403& 5.969 & \\ \hline
HCM $88\times88$ cells &3.18&&&& 1.178& 4.493 & \\ \hline
HCM $176\times176$ cells &3.45&&&& 1.060& 3.891 & \\ \hline
\hline
\vspace*{-.4cm}
&&&&&&&\\
\hline

FHCM $22\times22$ cells &2.89& 1.8770& 1.0300& 1.0202 & 2.905& 3.630 & 66.2\\ \hline
FHCM $44\times44$ cells &2.29& 1.8064& 1.0712& 1.0000 & 1.425& 1.918 & 82.2\\ \hline
FHCM $88\times88$ cells &2.23& 1.2724& 1.0108& 1.0000 & 1.182& 1.394 & 93.7\\ \hline
FHCM $176\times176$ cells &2.84& 1.0500& 1.0016& 1.0000 & 1.060& 1.130 & 100.0\\ \hline

\hline
\vspace*{-.4cm}
&&&&&&&\\
\hline

FMSM $22\times22$ cells &1.44& 4.3257& 1.4890& 1.1939 & & 1.246 & \\ \hline
FMSM $44\times44$ cells &1.46& 2.2958& 1.0975& 1.1932 & & 1.197 & \\ \hline
FMSM $88\times88$ cells &1.78& 1.7082& 1.0110& 1.0806 & & 1.213 & \\ \hline
FMSM $176\times176$ cells &2.57& 1.0845& 1.0010& 1.0000 & & 1.126 & \\ \hline

\hline
\end{tabular}
\label{tab:sinu_1_c1}
\end{table}


\vspace*{-3mm}

\begin{table}[H]\footnotesize
\caption{Performance/convergence results for 20 trials $F(x,y) = 1 + 0.99 \sin(2\pi x) \sin(2\pi y) $ on a coarse grid.}
\begin{tabular}{|c|c|c|c|c|c|c|}
\hline
\textbf{Grid Size} & $\bm{L_\infty}$ \textbf{Error} & $\bm{L_1}$ \textbf{Error} & \textbf{FMM Time}& \textbf{FSM Time}  & \textbf{LSM Time} & \textbf{\# Sweeps}\\
\hline
176 $\times$ 176  & 1.0533e-001  & 5.6430e-003  & 0.93   & 4.00  & 0.93  & 13 \\ 
\hline
\end{tabular}

\vspace*{1mm}

\begin{tabular}{|l|c|c|c|c|c|c|c|}
\hline 
\textbf{METHOD} & \textbf{TIME} & $\bm{\mathcal{R}}$ & $\bm{\mathcal{\rho}}$ & \textbf{R} & \textbf{AvHR} & \textbf{AvS} & \textbf{Mon \%} \\
\hline
\vspace*{-.4cm}
&&&&&&&\\
\hline

HCM $22\times22$ cells &0.74&&&& 1.165& 4.496 & \\ \hline
HCM $44\times44$ cells &0.73&&&& 1.085& 4.040 & \\ \hline
HCM $88\times88$ cells &0.83&&&& 1.026& 3.790 & \\ \hline
	
\hline
\vspace*{-.4cm}
&&&&&&&\\
\hline 
FHCM $22\times22$ cells &0.47& 1.0952& 1.0020& 1.0004 & 1.169& 1.388 & 94.2\\ \hline
FHCM $44\times44$ cells &0.50& 1.0200& 1.0005& 1.0000 & 1.087& 1.173 & 97.9\\ \hline
FHCM $88\times88$ cells &0.66& 1.0045& 1.0001& 1.0000 & 1.027& 1.051 & 100.0\\ \hline

\hline
\vspace*{-.4cm}
&&&&&&&\\
\hline 
FMSM $22\times22$ cells &0.37& 1.2819& 1.0044& 1.0164 & & 1.231 & \\ \hline
FMSM $44\times44$ cells &0.41& 1.1839& 1.0007& 1.0053 & & 1.116 & \\ \hline
FMSM $88\times88$ cells &0.59& 1.0979& 1.0001& 1.0000 & & 1.057 & \\ \hline

\hline
\end{tabular}

\vspace*{2mm}
\begin{tabular}{|c|c|c|c|c|c|c|}
\hline
\textbf{Grid Size} & $\bm{L_\infty}$ \textbf{Error} & $\bm{L_1}$ \textbf{Error} & \textbf{FMM Time}& \textbf{FSM Time}  & \textbf{LSM Time} & \textbf{\# Sweeps}\\
\hline
352 $\times$ 352  & 6.8813e-002  & 3.1818e-003  & 3.84   & 15.9  & 3.64  & 13 \\ 
\hline
\end{tabular}

\vspace*{1mm}

\begin{tabular}{|l|c|c|c|c|c|c|c|}
\hline 
\textbf{METHOD} & \textbf{TIME} & $\bm{\mathcal{R}}$ & $\bm{\mathcal{\rho}}$ & \textbf{R} & \textbf{AvHR} & \textbf{AvS} & \textbf{Mon \%} \\
\hline
\vspace*{-.4cm}
&&&&&&&\\
\hline
HCM $22\times22$ cells &3.00&&&& 1.178& 4.624 & \\ \hline
HCM $44\times44$ cells &2.76&&&& 1.076& 4.082 & \\ \hline
HCM $88\times88$ cells &2.83&&&& 1.033& 3.853 & \\ \hline
HCM $176\times176$ cells &3.29&&&& 1.008& 3.747 & \\ \hline
	
\hline
\vspace*{-.4cm}
&&&&&&&\\
\hline 

FHCM $22\times22$ cells &1.82& 1.1364& 1.0040& 1.0004 & 1.178& 1.405 & 93.2\\ \hline
FHCM $44\times44$ cells &1.71& 1.0204& 1.0005& 1.0000 & 1.080& 1.170 & 97.7\\ \hline
FHCM $88\times88$ cells &1.98& 1.0034& 1.0001& 1.0000 & 1.034& 1.071 & 99.2\\ \hline
FHCM $176\times176$ cells &2.69& 1.0006& 1.0000& 1.0000 & 1.008& 1.022 & 100.0\\ \hline

\hline
\vspace*{-.4cm}
&&&&&&&\\
\hline 
FMSM $22\times22$ cells &1.44& 2.3482& 1.0080& 1.0074 & & 1.231 & \\ \hline
FMSM $44\times44$ cells &1.42& 1.5167& 1.0014& 1.0037 & & 1.116 & \\ \hline
FMSM $88\times88$ cells &1.61& 1.1989& 1.0004& 1.0034 & & 1.057 & \\ \hline
FMSM $176\times176$ cells &2.44& 1.0953& 1.0001& 1.0015 & & 1.028 & \\ \hline

\hline
\end{tabular}
\label{tab:sinu_2_c2}
\end{table}

\pagebreak

\begin{remark}
We could perform a similar comparison between FMSM and sweeping methods, but the latter 
allow for yet another speed up technique: the sweeping can be stopped before the full
convergence to the solution of system \eqref{eq:Eik_discr}.
In fact, in many implementations of Fast Sweeping, the method terminates when the 
changes in grid values due to the most recent sweep fall below
some positive threshold $t^*$; e.g., see \cite{KaoOsherTsai}.
Similarly to FHCM and FMSM, this results in additional errors, and it is useful
to consider both these errors and the corresponding savings in computational time.
To the best of our knowledge, this issue has not been analyzed so far.
The practical implementations of FSM and LSM typically select $t^*$ heuristically
or make it proportional to the grid-size $h$.  It is usually claimed that the number
of sweeps necessary for convergence is $h$-independent \cite{TsaiChengOsherZhao, Zhao}.  
Tables \ref{tab:checker_41_c1} and \ref{tab:sinu_1_c1} show that the
number of sweeps-to-convergence (i.e., for $t^* = 0$) may in fact depend on $h$.
We believe that Figure \ref{fig:more_than_4} provides one possible explanation 
for this phenomenon (since the location of gridpoints relative to shocklines
is $h$-dependent).

For $t^* > 0$, the more relevant questions are:
\begin{enumerate}
\item
How well do the changes in the most recent sweep represent the additional errors,
which would result if we were to stop the sweeping?
\item
Is the number of sweeps (needed for a fixed $t^* > 0$) really $h$-independent?
\item
Supposing the additional (``early-termination'') errors could be estimated,
would the number of required sweeps be $h$-independent?
\item
Supposing FSM or LSM were run for as many sweeps as necessary to make 
the additional errors approximately the same as those introduced by FMSM or FHCM,
would the resulting computational costs be less than those of hybrid methods?
\end{enumerate}
To answer these questions for one specific ($41 \times 41$ checkerboard) example,
we have run both sweeping methods on $164 ^2$ and $1312^2$ grids. 
In table \ref{tab:check41_sweepMaxChange}
we report the $L_\infty$ change in grid values, the percentage of gridpoints changing, 
and potential early-termination errors ($\bm{\mathcal{R}}$, $\bm{\rho}$, and $\textbf{R}$) 
after each sweep.  At least for this particular example:

\begin{enumerate}
\item
The answer to Question 1 is inconclusive,
though the max changes are clearly correlated with $\textbf{R}$ and $\rho$.
\item
The answer to Question 2 is negative; moreover, after the same number of sweeps,
the max changes on the $1312^2$ grid are clearly larger than on the $164^2$ grid.
\item
The answer to Question 3 is negative; e.g., $\textbf{R}$ reduces below $1.1$ after
only $12$ sweeps on the $164^2$ grid, but the same reduction on the $1312^2$ grid
requires $42$ sweeps.
\item
To answer the last question, we note that for this example FHCM produces very 
small additional errors, while FMSM  results in  
$\textbf{R} = 1.0197$ and $\textbf{R} = 1.0193$
(on the $164^2$ grid with $21^2$ and $42^2$ cells, respectively; 
see Table \ref{tab:checker_41_c1}).
As Table \ref{tab:check41_sweepMaxChange}A shows,
$16$ sweeps would be needed for FSM or LSM to produce the same $\textbf{R}$ values
on this grid. 
Our computational experiment shows that FSM and LSM times for these 16 sweeps 
are 6.62 and 2.91  seconds respectively (note that this is the total time for 20 trials,
similar to the times reported in Table \ref{tab:checker_41_c1}).
Thus, FMSM is still more than 3.5 times faster than the early-terminated LSM 
and more than 8 times faster than the early-terminated FSM.
For the $1312^2$ example, we see that the error ratios take longer to converge to 
$1$ for the sweeping methods (Table \ref{tab:check41_sweepMaxChange}B).  
The FMSM's $\textbf{R}$ values of $\{3.3991, 1.7662, 1.7123\}$ (from 
Table \ref{tab:checker_41}, for different cell sizes) 
correspond to $\{28, 37, 37 \}$ sweeps in Table \ref{tab:check41_sweepMaxChange}B.  
The experimentally measured early-terminated execution times for FSM  and LSM are 
$\{36.85, 48.77, 48.77\}$ seconds and $\{7.40, 11.68, 11.68\}$ seconds respectively.
Again, FMSM still holds a large advantage (more than 4 times faster than LSM and more than 
18 times faster than FSM).  We note that for both the $164^2$ and 
$1312^2$ cases, the early-terminated FSM time was linear in the number of sweeps, 
while LSM did not receive as much of a speed boost; this is natural
since the percentage of gridpoints changing in the omitted ``later iterations'' is low,
and the LSM's computational cost is largely dependent on the number of unlocked gridpoints 
in each sweep.
\end{enumerate}
\end{remark}

\begin{center}
\begin{table}[H]\footnotesize
$
\begin{array}{cc}

\hspace{-1.7cm}

\begin{tabular}{|c|c|c|c|c|c|}
\hline
\textbf{Sweep}&\textbf{Max}&\textbf{\% GPs} & $\bm{\mathcal{R}}$ & $\bm{\mathcal{\rho}}$ 
& \textbf{R}\\
\textbf{\#}&\textbf{Change}&\textbf{changing} & & & \\ \hline \hline

$1$ &1.00e+008&26.22 & -& -& -\\ \hline
$2$ &1.000e+008&31.856 & -& -& -\\ \hline
$3$ &1.000e+008&58.247 & 44.595& 1.7709& 4.8179\\ \hline
$4$ &2.7622e-001&44.4527 & 1.4685& 1.1027& 1.2445\\ \hline
$5$ &6.3846e-003&41.5341 & 1.4224& 1.0888& 1.1995\\ \hline
$6$ &5.9641e-003&41.1957 & 1.4195& 1.0759& 1.1995\\ \hline
$7$ &5.9641e-003&41.0730 & 1.3832& 1.0631& 1.1951\\ \hline
$8$ &5.4993e-003&40.1919 & 1.3331& 1.0509& 1.1562\\ \hline
$9$ &4.9918e-003&37.0650 & 1.3243& 1.0440& 1.1205\\ \hline
$10$ &4.9918e-003&36.6337 & 1.3230& 1.0377& 1.1205\\ \hline
$11$ &4.9918e-003&36.3995 & 1.2881& 1.0314& 1.1191\\ \hline
$12$ &4.7740e-003&34.6743 & 1.2492& 1.0255& 1.0854\\ \hline
$13$ &4.5076e-003&31.3318 & 1.2403& 1.0218& 1.0532\\ \hline
$14$ &4.5076e-003&30.8150 & 1.2400& 1.0185& 1.0520\\ \hline
$15$ &4.5076e-003&30.4767 & 1.2098& 1.0152& 1.0511\\ \hline
$16$ &4.1600e-003&28.0934 & 1.1820& 1.0121& 1.0270\\ \hline
$17$ &3.6304e-003&22.6502 & 1.1646& 1.0102& 1.0004\\ \hline
$18$ &3.6304e-003&21.8062 & 1.1644& 1.0085& 1.0000\\ \hline
$19$ &3.6304e-003&21.2374 & 1.1467& 1.0068& 1.0000\\ \hline
$20$ &3.2984e-003&19.1404 & 1.1268& 1.0052& 1.0000\\ \hline
$21$ &2.7917e-003&14.3367 & 1.1079& 1.0043& 1.0000\\ \hline
$22$ &2.7142e-003&13.6340 & 1.1079& 1.0035& 1.0000\\ \hline
$23$ &2.7142e-003&13.2250 & 1.0951& 1.0027& 1.0000\\ \hline
$24$ &2.4311e-003&11.6746 & 1.0812& 1.0020& 1.0000\\ \hline
$25$ &2.1725e-003&8.4659 & 1.0637& 1.0015& 1.0000\\ \hline
$26$ &1.8533e-003&7.9677 & 1.0630& 1.0012& 1.0000\\ \hline
$27$ &1.8533e-003&7.6852 & 1.0546& 1.0009& 1.0000\\ \hline
$28$ &1.7075e-003&6.6664 & 1.0457& 1.0006& 1.0000\\ \hline
$29$ &1.5049e-003&4.8000 & 1.0365& 1.0004& 1.0000\\ \hline
$30$ &1.1216e-003&4.4653 & 1.0303& 1.0003& 1.0000\\ \hline
$31$ &1.1216e-003&4.2646 & 1.0257& 1.0002& 1.0000\\ \hline
$32$ &1.0109e-003&3.5656 & 1.0209& 1.0001& 1.0000\\ \hline
$33$ &8.5675e-004&2.2754 & 1.0153& 1.0001& 1.0000\\ \hline
$34$ &4.8751e-004&2.0300 & 1.0110& 1.0001& 1.0000\\ \hline
$35$ &4.8751e-004&1.8813 & 1.0087& 1.0000& 1.0000\\ \hline
$36$ &4.2582e-004&1.4314 & 1.0068& 1.0000& 1.0000\\ \hline
$37$ &3.4338e-004&0.7064 & 1.0043& 1.0000& 1.0000\\ \hline
$38$ &1.1188e-004&0.5689 & 1.0025& 1.0000& 1.0000\\ \hline
$39$ &1.1188e-004&0.4871 & 1.0015& 1.0000& 1.0000\\ \hline
$40$ &8.9968e-005&0.2863 & 1.0011& 1.0000& 1.0000\\ \hline
$41$ &6.8284e-005&0.0632 & 1.0006& 1.0000& 1.0000\\ \hline
$42$ &2.4066e-005&0.0297 & 1.0002& 1.0000& 1.0000\\ \hline
$43$ &1.0931e-005&0.0112 & 1.0000& 1.0000& 1.0000\\ \hline
$44$ &0.0000e+000&0.0000 & 1.0000& 1.0000& 1.0000\\ \hline
\hline

\end{tabular}
&
\hspace{.2cm}
\begin{tabular}{|c|c|c|c|c|c|}
\hline
\textbf{Sweep}&\textbf{Max}&\textbf{\% GPs} & $\bm{\mathcal{R}}$ & $\bm{\mathcal{\rho}}$ 
& \textbf{R}\\
\textbf{\#}&\textbf{Change}&\textbf{changing} & & & \\ \hline \hline

$1$ &1.0e+008&25.2 & -& -& -\\ \hline
$2$ &1.000e+008&34.249 & -& -& -\\ \hline
$3$ &1.000e+008&62.372 & 48.051& 9.2339& 30.026\\ \hline
$4$ &3.621e-001&49.221 & 12.002& 4.1935& 7.7797\\ \hline
$5$ &1.0709e-002&43.0590 & 11.194& 3.9168& 7.7098\\ \hline
$6$ &1.0252e-002&42.1586 & 10.474& 3.6528& 7.2822\\ \hline
$7$ &1.0252e-002&42.0001 & 10.269& 3.3925& 7.2771\\ \hline
$8$ &1.0229e-002&39.8694 & 9.6885& 3.1431& 6.9538\\ \hline
$9$ &1.0207e-002&34.3652 & 9.6713& 2.9458& 6.9386\\ \hline
$10$ &1.0207e-002&33.2951 & 9.4074& 2.7562& 6.5653\\ \hline
$11$ &1.0207e-002&33.1453 & 9.0914& 2.5686& 6.5623\\ \hline
$12$ &1.0185e-002&31.2973 & 8.6218& 2.3892& 6.1648\\ \hline
$13$ &1.0165e-002&26.4975 & 8.5771& 2.2483& 6.1607\\ \hline
$14$ &1.0165e-002&25.5465 & 8.3781& 2.1135& 5.8497\\ \hline
$15$ &1.0165e-002&25.3961 & 8.0972& 1.9804& 5.8487\\ \hline
$16$ &1.0145e-002&23.7761 & 7.5600& 1.8536& 5.4607\\ \hline
$17$ &1.0127e-002&19.6460 & 7.5550& 1.7561& 5.4571\\ \hline
$18$ &1.0127e-002&18.8095 & 7.2488& 1.6635& 5.0667\\ \hline
$19$ &1.0127e-002&18.6819 & 7.0133& 1.5722& 5.0658\\ \hline
$20$ &1.0108e-002&17.2760 & 6.5857& 1.4861& 4.7569\\ \hline
$21$ &1.0092e-002&13.8028 & 6.5691& 1.4221& 4.7449\\ \hline
$22$ &1.0092e-002&13.0992 & 6.2438& 1.3619& 4.3682\\ \hline
$23$ &1.0092e-002&12.9746 & 6.0465& 1.3027& 4.3674\\ \hline
$24$ &1.0075e-002&11.8224 & 5.5031& 1.2475& 3.9749\\ \hline
$25$ &1.0060e-002&9.0073 & 5.4990& 1.2088& 3.9720\\ \hline
$26$ &1.0060e-002&8.4400 & 5.2424& 1.1730& 3.6712\\ \hline
$27$ &1.0060e-002&8.3202 & 5.0816& 1.1380& 3.6705\\ \hline
$28$ &1.0045e-002&7.3932 & 4.5433& 1.1059& 3.2817\\ \hline
$29$ &1.0031e-002&5.2311 & 4.5402& 1.0854& 3.2794\\ \hline
$30$ &1.0031e-002&4.8007 & 4.1140& 1.0670& 2.8875\\ \hline
$31$ &1.0031e-002&4.7054 & 3.9971& 1.0491& 2.8871\\ \hline
$32$ &1.0018e-002&4.0108 & 3.5893& 1.0334& 2.5926\\ \hline
$33$ &1.0005e-002&2.5072 & 3.5711& 1.0249& 2.5795\\ \hline
$34$ &1.0005e-002&2.2144 & 3.1264& 1.0177& 2.2008\\ \hline
$35$ &1.0005e-002&2.1276 & 3.0465& 1.0109& 2.2005\\ \hline
$36$ &9.9928e-003&1.6782 & 2.4976& 1.0053& 1.8040\\ \hline
$37$ &9.9809e-003&0.8296 & 2.4942& 1.0034& 1.8016\\ \hline
$38$ &9.9809e-003&0.6789 & 2.1526& 1.0020& 1.5223\\ \hline
$39$ &9.9809e-003&0.6047 & 2.1076& 1.0009& 1.5223\\ \hline
$40$ &9.9619e-003&0.3888 & 1.5638& 1.0002& 1.1295\\ \hline
$41$ &5.0711e-003&0.1151 & 1.5638& 1.0001& 1.1295\\ \hline
$42$ &4.9343e-003&0.0698 & 1.1631& 1.0000& 1.0000\\ \hline
$43$ &1.4668e-003&0.0338 & 1.0000& 1.0000& 1.0000\\ \hline
$44$ &0.0000e+000&0.0000 & 1.0000& 1.0000& 1.0000\\ \hline

\hline
\end{tabular}

\\A&B

\end{array}
$
\caption{Maximum change of $V$ for the sweeping methods for the
$41 \times 41$ checkerboard example on the $164 \times 164$ grid (A) and $1312 \times 1312$ grid (B). }

\vspace*{.2cm}
\label{tab:check41_sweepMaxChange}

\end{table}
\end{center}

\fi

\iffullversion
\subsection{Continuous speed functions with general boundary conditions}
\label{ss:general_bc}


Next we return to speed functions $F(x,y) = 1 + 0.99 \sin(2 \pi x) \sin(2 \pi y)$
and $F(x,y) = 1 + \frac{1}{2} \sin(20\pi x) \sin(20 \pi y)$,
but this time with zero boundary conditions on the entire boundary of the square.
The performance data is summarized in Tables \ref{tab:sinu_1_gbc} and \ref{tab:sinu_2_gbc}.

\begin{remark}
\label{rem:FMSM_boundary_conditions}
Our current implementation of FMSM treats the coarse gridpoints nearest to the boundary 
as $Accepted$ in the initialization.  If there is more than one coarse gridpoint in the exit set, as in the following examples,
 care must be taken when 
ranking the ``acceptance order" of these coarse gridpoints.
While in the case of single-point exit sets it is safe to assign a zero value to these
coarse gridpoints,
for general boundary conditions we compute the values by a one-sided update 
from the cell center to the nearest
point on the boundary.  
In addition, our FMSM implementation iterates FSM to convergence on all cells
containing parts of $Q$.
\end{remark}

\begin{figure}[H]
\center{
$
\begin{array}{cc}
\includegraphics[scale = .28] {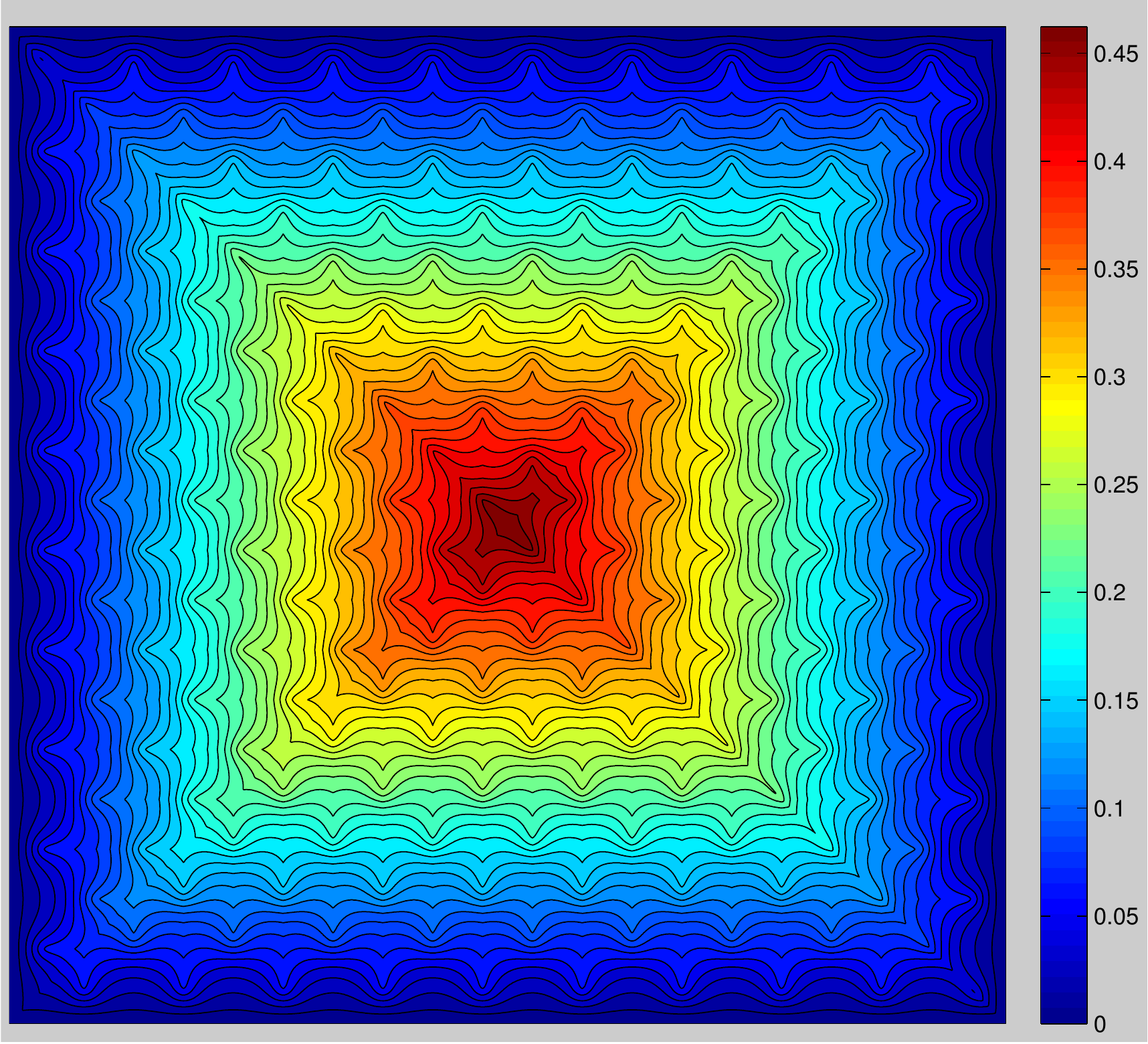}

&
\hspace*{5mm}
\includegraphics[scale = .28] {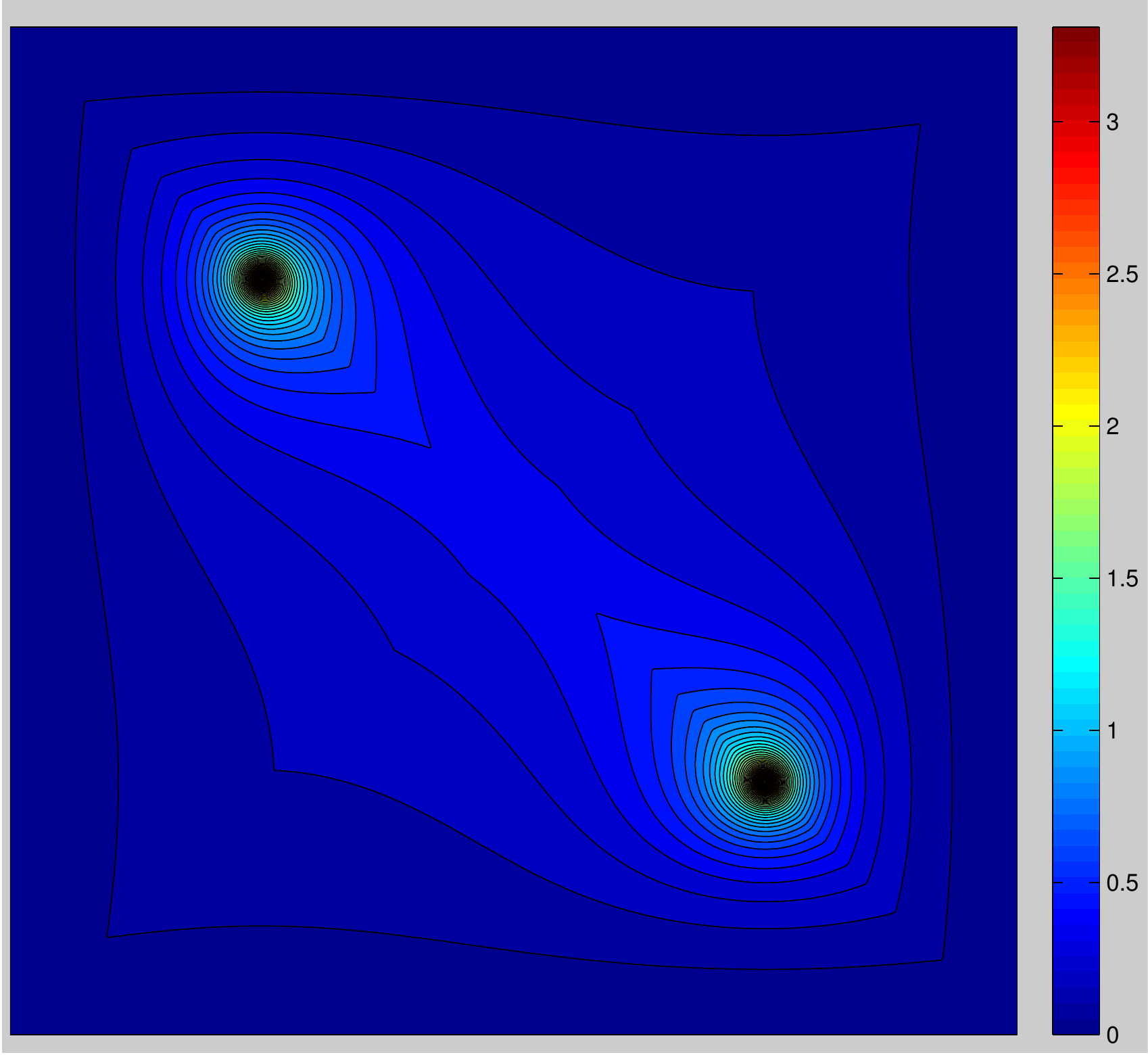}\\
A & B
\end{array}
$
}
\caption{
{\footnotesize
Min time to $\boundary$ under two sinusoidal speed functions.}}
\label{fig:sinusoids_gbc_computed}
\end{figure}

\vspace*{-8mm}
\begin{table}[h]\footnotesize
\caption{Performance/convergence results for $F(x,y) = 1 + \frac{1}{2} \sin(20\pi x) \sin(20\pi y) $ with $Q = \boundary$.}
 \vspace*{2mm}
\begin{tabular}{|c|c|c|c|c|c|c|}
\hline
\textbf{Grid Size} & $\bm{L_\infty}$ \textbf{Error} & $\bm{L_1}$ \textbf{Error} & \textbf{FMM Time} & \textbf{FSM Time} & \textbf{LSM Time} & \textbf{\# Sweeps}\\
\hline
1408 $\times$ 1408  & 1.3670e-003  & 3.7171e-004  & 3.89   & 24.3  & 6.62  & 24 \\ 
\hline
\end{tabular}

\vspace*{.2cm}

\begin{tabular}{|l|c|c|c|c|c|c|c|}
\hline 
\textbf{METHOD} & \textbf{TIME} & $\bm{\mathcal{R}}$ & $\bm{\mathcal{\rho}}$ & \textbf{R} & \textbf{AvHR} & \textbf{AvS} & \textbf{Mon \%} \\
\hline
\vspace*{-.4cm}
&&&&&&&\\
\hline
HCM $22\times22$ cells &3.48&&&& 1.853& 10.273 & \\ \hline
HCM $44\times44$ cells &2.92&&&& 1.470& 6.811 & \\ \hline
HCM $88\times88$ cells &2.53&&&& 1.195& 4.987 & \\ \hline
HCM $176\times176$ cells &2.35&&&& 1.098& 4.301 & \\ \hline
HCM $352\times352$ cells &2.37&&&& 1.046& 3.951 & \\ \hline
HCM $704\times704$ cells &2.91&&&& 1.018& 3.785 & \\ \hline

\hline
\vspace*{-.4cm}
&&&&&&&\\
\hline 

FHCM $22\times22$ cells &2.60& 20660& 1.5321& 3.2150 & 2.915& 4.498 & 54.5\\ \hline
FHCM $44\times44$ cells &1.95& 62.164& 1.2465& 1.5447 & 1.539& 2.502 & 68.0\\ \hline
FHCM $88\times88$ cells &1.66& 64.719& 1.0187& 1.0128 & 1.223& 1.749 & 83.9\\ \hline
FHCM $176\times176$ cells &1.55& 5.7122& 1.0032& 1.0063 & 1.102& 1.361 & 92.4\\ \hline
FHCM $352\times352$ cells &1.66& 1.1083& 1.0007& 1.0011 & 1.047& 1.165 & 97.5\\ \hline
FHCM $704\times704$ cells &2.40& 1.0192& 1.0001& 1.0001 & 1.018& 1.064 & 100.0\\ \hline

\hline
\vspace*{-.4cm}
&&&&&&&\\
\hline 

FMSM $22\times22$ cells &1.97& 1.6383e+5& 7.8665& 12.339 & & 2.184 & \\ \hline
FMSM $44\times44$ cells &1.67& 1.1325e+6& 2.6113& 4.2370 & & 1.892 & \\ \hline
FMSM $88\times88$ cells &1.42& 5506.21& 1.0388& 1.8072 & & 1.527 & \\ \hline
FMSM $176\times176$ cells &1.29& 859.45& 1.0044& 1.2609 & & 1.265 & \\ \hline
FMSM $352\times352$ cells &1.40& 253.58& 1.0009& 1.0270 & & 1.134 & \\ \hline
FMSM $704\times704$ cells &2.17& 6.6107& 1.0001& 1.0000 & & 1.062 & \\ \hline

\hline
\end{tabular}
\label{tab:sinu_1_gbc}
\end{table}

\begin{table}[h]\footnotesize
\caption{Performance/convergence results for $F(x,y) = 1 + 0.99 \sin(2\pi x) \sin(2\pi y) $  with $Q = \boundary$.}
 \vspace*{2mm}
\begin{tabular}{|c|c|c|c|c|c|c|}
\hline
\textbf{Grid Size} & $\bm{L_\infty}$ \textbf{Error} & $\bm{L_1}$ \textbf{Error} & \textbf{FMM Time} & \textbf{FSM Time} & \textbf{LSM Time} & \textbf{\# Sweeps}\\
\hline

1408 $\times$ 1408  & 2.2246e-002  & 2.7572e-004  & 3.66   & 8.06  & 2.58  & 8 \\ 
\hline
\end{tabular}

\vspace*{.2cm}

\begin{tabular}{|l|c|c|c|c|c|c|c|}
\hline 
\textbf{METHOD} & \textbf{TIME} & $\bm{\mathcal{R}}$ & $\bm{\mathcal{\rho}}$ & \textbf{R} & \textbf{AvHR} & \textbf{AvS} & \textbf{Mon \%} \\
\hline
\vspace*{-.4cm}
&&&&&&&\\
\hline

HCM $22\times22$ cells &2.03&&&& 1.176& 4.448 & \\ \hline
HCM $44\times44$ cells &1.97&&&& 1.089& 4.021 & \\ \hline
HCM $88\times88$ cells &1.93&&&& 1.047& 3.830 & \\ \hline
HCM $176\times176$ cells &1.96&&&& 1.020& 3.718 & \\ \hline
HCM $352\times352$ cells &2.10&&&& 1.009& 3.670 & \\ \hline
HCM $704\times704$ cells &2.74&&&& 1.006& 3.649 & \\ \hline

\hline
\vspace*{-.4cm}
&&&&&&&\\
\hline

FHCM $22\times22$ cells &1.51& 136.37& 1.0001& 1.0000 & 1.176& 1.903 & 93.4\\ \hline
FHCM $44\times44$ cells &1.35& 2.4167& 1.0000& 1.0000 & 1.091& 1.443 & 99.0\\ \hline
FHCM $88\times88$ cells &1.32& 2.4167& 1.0000& 1.0000 & 1.048& 1.226 & 99.6\\ \hline
FHCM $176\times176$ cells &1.39& 1.6390& 1.0000& 1.0000 & 1.020& 1.110 & 99.8\\ \hline
FHCM $352\times352$ cells &1.57& 1.0000& 1.0000& 1.0000 & 1.009& 1.054 & 99.9\\ \hline
FHCM $704\times704$ cells &2.33& 1.0000& 1.0000& 1.0000 & 1.006& 1.028 & 100.0\\ \hline

\hline
\vspace*{-.4cm}
&&&&&&&\\
\hline 
FMSM $22\times22$ cells &1.57& 12592& 1.0441& 1.0000 & & 1.599 & \\ \hline
FMSM $44\times44$ cells &1.27& 355.53& 1.0088& 1.0000 & & 1.306 & \\ \hline
FMSM $88\times88$ cells &1.15& 355.53& 1.0055& 1.0000 & & 1.157 & \\ \hline
FMSM $176\times176$ cells &1.14& 303.61& 1.0030& 1.0000 & & 1.079 & \\ \hline
FMSM $352\times352$ cells &1.31& 134.60& 1.0012& 1.0000 & & 1.040 & \\ \hline
FMSM $704\times704$ cells &2.11& 68.199& 1.0004& 1.0000 & & 1.014 & \\ \hline
\hline
\end{tabular}
\label{tab:sinu_2_gbc}
\end{table}

\fi
 
\section{Conclusions}
\label{s:conclusions}
We have introduced three new efficient hybrid methods for Eikonal equations.
Using a splitting of the domain into a number of cells (with the speed function
approximately constant on each of them), our methods employ sweeping methods on 
individual cells with the order of cell-processing and the direction of sweeps 
determined by a marching-like procedure on a coarser scale.  
Such techniques may introduce additional errors to attain higher
computational efficiency.
Of these new methods FMSM is generally the fastest and somewhat easier to implement,
while FHCM introduces smaller additional errors, 
and HCM is usually the slowest of the three but provably converges to the exact solutions.
The numerical evidence presented in this paper strongly suggests that\\
$\bullet \,$ when $h$ and $h^c$ are sufficiently small, 
additional errors introduced by FMSM and FHCM are negligible 
compared to those already present due to discretization;\\
$\bullet \,$ for the right $(h,h^c)$-combinations, our new hybrid algorithms 
significantly outperform the prior fast methods (FMM, FSM, and LSM).\\
Of course, the rate of change of the speed function $F$ determines the suitable 
size of cells and our methods are particularly efficient for the examples where $F$
is piecewise constant.  

All of the examples considered here used predetermined uniform cell-sizes.
From a practitioner's point of view, the value of the proposed methods will 
greatly increase once we develop bounds and estimates for the additional errors 
in both FMSM and FHCM.
Such estimates would be also very useful for the computational
costs of all three hybrid methods on a given cell-decomposition. 
In the future, we intend to automate the choice of cell-sizes (based on the
speed function and user-specified tolerances for additional errors) 
and further relax the requirement that all cells need to be uniform.
A generalization of this approach to cell-subdivision of unstructured meshes 
will also be valuable.

We expect the extensions of these techniques to higher dimensional problems to be 
useful for many applications and relatively straight-forward -- especially for
FMSM and HCM.  A higher dimensional version of the ``cell boundary monotonicity check''
will be needed to extend FHCM.

Other obvious directions for future work include extensions to higher-order
accurate discretizations and parallelizable cell-level numerical methods 
for Eikonal equations.

More generally, we hope that the ideas presented here can serve as a basis
for causal domain decomposition and efficient two-scale methods
for other static nonlinear PDEs, including those 
arising in anisotropic optimal control and differential games.

\vspace{.2in}

\end{document}